\documentclass[11pt]{article}
\usepackage{etex}
\usepackage{amsmath,amsthm,amsfonts,amscd,flafter,epsf}
\usepackage{amssymb,graphicx,color,array}
\usepackage{epsfig,manfnt,mathrsfs}
\usepackage{amsfonts}
\usepackage{amssymb}
\usepackage{amsmath}
\usepackage{subfig}
\usepackage{graphicx}
\usepackage{caption}
\usepackage{tikz}
\usepackage{comment}
\usepackage{mathtools}
\usepackage{authblk}

\usetikzlibrary{matrix,arrows,decorations.pathmorphing}
\newtheorem{prop}{Proposition}[section]

\newtheorem{thm}{Theorem}[section]
\newtheorem{lem}{Lemma}[section]
\newtheorem{defi}{Definition}[section]
\newtheorem{ex}{Example}[section]
\newtheorem{rem}{Remark}[section]

\usepackage{arydshln}

\usepackage{amsthm}
\usepackage{latexsym}
\usepackage{amscd}
\usepackage{flafter}
\usepackage{epsf}
\usepackage{arydshln}
\usepackage{colortbl}

\counterwithin{equation}{section}

\usepackage[all]{xy}
\usepackage[a4paper,footskip=2.3cm,headheight=20pt,top=3cm,
bottom=2.5cm,right=3.5cm, left=3.5cm]{geometry}

\usepackage{lipsum}
\usepackage{setspace}

\title{Counting Minimal tori in Riemannian Manifolds}%
\author{Narges Bagherifard\thanks{School of Mathematics, Institute for Research in Fundamental Sciences (IPM), Tehran, Iran \texttt{n.bagherifard@gmail.com}}}%

\begin{document}
	\maketitle
	\noindent\textbf{Abstract.}
		In this paper, we introduce a function which counts minimal tori in a Riemannian manifold $(M,g)$ with $\mathrm{dim}\, M\ge8$. Moreover, we show that this count function is invariant under perturbations of the metric.

	\section{Introduction}
The problem of counting closed geodesics in Riemannian Manifolds of constant negative curvature has been of interest for many years (\cite{Margulis} and \cite{Sarnak}). 
Recently, Eftekhary in \cite{EF-CLosedGeodesics} has
introduced a method for counting closed geodesics in closed manifolds with
arbitrary curvature.

Minimal surfaces in a Riemannian manifold are higher dimensional generalizations of closed geodesics.
A natural question is how to count these types of submanifolds.

Let $\mathscr{G}$ denote the space of metrics of class $C^\infty$ on a smooth manifold $M$. 
For $g\in\mathscr{G}$, denote by $\tilde{\mathscr{L}}(g)$ the space of all closed $g-$minimal surfaces, with a specified genus, of the Riemannian manifold $(M,g)$ (See Section \ref{ModuliSpaceOfMinimalSurfaces} for the precises definition). 
Let $\overline{\mathcal{U}}_g$ be a compact, open subset of $\tilde{\mathscr{L}}(g)$.
When we refer to counting minimal surfaces, we are describing a function denoted as $n(g,\overline{\mathcal{U}}_g)$, which counts the number of elements in $\overline{\mathcal{U}}_g$, probably with appropriate weights.
Clearly, $\overline{\mathcal{U}}_g$ must be finite in order to $n(g,\overline{\mathcal{U}}_g)$ be well-defined.
This condition is met when $\tilde{\mathscr{L}}(g)$ is a $0-$dimensional manifold.
The main obstacle for $\tilde{\mathscr{L}}(g)$ to be a manifold is the existence of non-trivial Jacobi fields along a $g-$minimal surface.
By a theorem of White, \cite{White-2017}, it is known that for a generic metric $g$, every closed, immersed $g$-minimal submanifold of $M$ does not have non-trivial Jacobi fields.  
Consequently, for a generic metric $g$, the moduli space of closed, immersed $g$-minimal submanifolds of $M$ is countable, rendering the counting of elements in $\overline{\mathcal{U}}_g$ meaningful. 
However, $n(g,\overline{\mathcal{U}}_g)$, for a generic metric $g$, cannot be simply defined by counting the elements of $\overline{\mathcal{U}}_g$, as a small perturbation of $g$ may change this count.

Denote by
$$\bar{\Gamma}:[0,1]\rightarrow\mathscr{G}$$ 
a path of metrics and let $g_t:=\bar{\Gamma}(t)$ for $t\in [0,1]$.
Let $\tilde{\mathscr{L}}(\bar{\Gamma})$ denote the moduli space of minimal surfaces, with a specified genus, along the path $\bar{\Gamma}$ (See Subsection \ref{LociOfSuper-RigidityBreakdown}).
For a compact, open subset
$\overline{\mathcal{U}}\subset \tilde{\mathscr{L}}(\bar{\Gamma})$, set
$$\overline{\mathcal{U}}_i:=\overline{\mathcal{U}}\cap \tilde{\mathscr{L}}(g_i)$$ 
, $i=0,1$. 
Then the count function is invariant if for every path of metrics like $\bar{\Gamma}$, 
$n(g_0,\overline{\mathcal{U}}_0)=n(g_1,\overline{\mathcal{U}}_1)$.
What prevents the count function from being invariant under perturbations is the occurrence of a birth-death phenomenon or the emergence of the branching phenomenon.
The latter, which prevent $\tilde{\mathscr{L}}(\bar{\Gamma})$ from being a manifold, arises when a sequence of $g_t-$minimal embedded surfaces converges to a $g-$minimal surface that is a multiple cover of an embedded one.
Indeed, it is the symmetries of the multiple cover that cause problems with transversality, consequently preventing the moduli space to be a manifold.
A metric for which this phenomenon does not happen is called super-rigid (See Definition \ref{Super-rigidMetric}).
This notion arises in the context of $J-$holomorphic curves, which are solutions of the Cauchy-Riemann operator.
Super-rigidity can be restated in terms of the triviality of the kernel of an operator.
Genericity of super-rigidity for Cauchy-Riemann operator has been proved by Wendl \cite{Wendl}.
Wendl's method propose a new proof for the genericity of bumpy metrics in the sense of White \cite{White-1991}.
Indeed, we prove a stronger result:

\begin{thm}\label{Theorem-A}
	Let $M$ be a smooth, closed, connected manifold with $\mathrm{dim}\, M\ge 6$. 
	The subspace $\mathscr{G}^\bullet\subset\mathscr{G}$ of super-rigid metrics is comeager.
	In fact, the subspace of metrics which are not super-rigid is of codimension one.
\end{thm}

The locus of the failure of super-rigidity which is a codimension $1$ subset is called a wall.
Walls divide the space into chambers. In the process of counting tori, inside each chamber, tori are classified based on Taubes' viewpoint \cite{Taubes}.

Defining a count function requires a careful analysis of the moduli space around walls and birth-death phenomena, which leads us to the following:



\begin{thm}\label{Theorem-B}
	Let $M$ be a smooth manifold with $\mathrm{dim}\, M\ge 8$.
	Denote by $\tilde{\mathscr{L}}(g)$ the space of all $g-$minimal immersed maps of a torus to $M$, up to reparameterization.
	Let $\overline{\mathcal{U}}_g$ be a compact, open subset of $\tilde{\mathscr{L}}(g)$.
	There is a weight $n(g,\overline{\mathcal{U}}_g)\in\mathbb{Z}$, defined based on Taubes' classifications of tori, with the following properties:
	\begin{enumerate}
		\item If $g$ is super-rigid and $[v]\in\tilde{\mathscr{L}}(g)$ represents an embedded torus, then $n(g,[v])=\pm 1$.
		\item Let $\bar{\Gamma}:[0,1]\rightarrow\mathscr{G}$ be a path of metrics and let $g_t:=\bar{\Gamma}(t)$ for $t\in [0,1]$.
		Denote by
		\[
		\mathscr{L}(\bar{\Gamma}):=\{(t,[u])\,|\,[u]\in\tilde{\mathscr{L}}(g_t)\}
		\]
		the moduli space of minimal tori along the path $\bar{\Gamma}$.
		If $\overline{\mathcal{U}}\subset \tilde{\mathscr{L}}(\bar{\Gamma})$ be a compact, open subset and 
		$\mathcal{U}_i:=\overline{\mathcal{U}}\cap \tilde{\mathscr{L}}(g_i)$, $i=0,1$, then 
		$$n(g_0,\mathcal{U}_0)=n(g_1,\mathcal{U}_1).$$
	\end{enumerate}
\end{thm}

\textbf{Organization.} The paper is organized as follows.
In Section \ref{ModuliSpaceOfMinimalSurfaces}, the moduli space of minimal embedded surfaces, $\tilde{\mathscr{Y}}$, is introduced.
This space consists of pairs $(g,\Sigma)$, where $g$ is a metric on $M$ and $\Sigma\subset M$ is a $g-$minimal submanifold.
In \cite{White-1991}, White shows that $\tilde{\mathscr{Y}}$ is a Banach manifold. This result and its corollaries are reviewed in this section.

Due to the failure of transversality, the results of Section \ref{ModuliSpaceOfMinimalSurfaces} does not hold for immersed submanifolds.
It turns out that investigating transversality can be reduced to examining two simpler conditions called Petri's condition and the flexibility condition (See Subsection \ref{Subsection:TransversalityTheorem}).
In Subsection \ref{FlexibilityConditionSection}, we show that the latter condition is satisfied, but the former may fail in some instances.
In Subsection \ref{Subsetion:WendCondition}, we review Wendl's work, \cite{Wendl}, according to which the locus where Petri's condition fails is of infinite codimension, provided the symbol of the Jacobi operator satisfies a technical condition known as Wendl's condition \cite{DoanWalpuski}.  
Moreover, we investigate the validity of Wendl's condition for the Jacobi operator.

A sequence of $g_t-$minimal embedded surfaces may converge to an "orbifold" $g-$minimal surface.
However, the analysis of the index of the Jacobi operator along orbifold Riemann surfaces in Subection \ref{IndexOfTwistedOperator} demonstrates that this phenomenon can be prevented if the dimension of $M$ is greater than $6$. 
Moreover, the arguments in Subsection \ref{LociOfSuper-RigidityBreakdown} provide us with a proof of Theorem \ref{Theorem-A}.
A review of orbifolds can be found in Appendix \ref{Jacobi Operator}.

In Subsection \ref{Subsection:CountingMinimalToriInManifold}, the count function will be introduced.
In Subsubsections \ref{Subsection:LocalModel_CriticalPoint} and \ref{Subsection:LocalModel_WeakLimitPoint}, we analyze the moduli space around the points where the birth-death phenomenon occurs and also
around the points where a paths of metrics faces a wall, respectively.
With the local model, in Subsubsection  \ref{Subsection:InvarianceOfTheCountFunction}, we show the invariance and, consequently, the well-definedness of the counting function.

Let $v:\Sigma\rightarrow (M,g)$ represent a $g-$minimal embedded surface. 
As mentioned above, one of the major problems in investigating super-rigidity arises where a sequence $(g_t,v_t)$ of $g_t-$minimal embedded maps, converges to an element $(g,u)$, where $u$ factors as $u=v\circ\pi$ and $\pi:\Sigma\rightarrow\Sigma'$ is a covering map of degree $d>1$.

The occurrence of a multiple covering is equivalent to the non-triviality of the kernel of the Jacobi operator along $u$, i.e.  $\mathrm{Ker}\,\mathcal{J}_{g,u}$.
The operator $\mathcal{J}_{g,u}$ is found to be equivalent to $\pi^*\mathcal{J}_{g,v}$, which in turn is equivalent to $\mathcal{J}_{g,v}\otimes \underline{V}$, for an appropriate vector bundle $\underline{V}$.

The relationship between pullback and twisted operators is explored in Appendix \ref{AppendixA}.
The automorphism group of $\pi$ leads us to a stratification of the moduli space, which is a fundamental concept in addressing transversality. 
Appendix \ref{AppendixA} also provides a review of the stratification and transversality in the presence of symmetry.

\textbf{Acknowledgement}. I would like to thank Eaman Eftekhary for his insightful discussions and for proposing the problem addressed in this paper. 
 

	\tableofcontents
	\section{Moduli Space Of Minimal Surfaces}\label{ModuliSpaceOfMinimalSurfaces}
Let $M$ be a closed, smooth manifold with $\mathrm{dim}\,M=n$.
Denote by $\mathscr{G}_k$, $k\ge 3$ or $k=\infty$, the space of $C^k-$ Riemannian metrics on $M$.
If the degree of smoothness of metrics is not important to us, we drop $k$ and write $\mathscr{G}$.

Let $\Sigma$ be a surface equipped with a fixed metric $h$. For $u:\Sigma \rightarrow M$, let $W_u:=u^*TM$ be the pullback of the tangent bundle $TM$ of $M$.
$W_u$ equipped with its smooth structure defines a sheaf which is denoted by $\mathcal{W}_u$.
Similarly, the tangent bundle of $\Sigma$, also equipped with its smooth structure is another sheaf on $\Sigma$ which we denote it by $\mathcal{T}\Sigma$. Then $u$ induces a map of sheaves, denoted by
\[
\mathrm{d}u:\mathcal{T}\Sigma \rightarrow \mathcal{W}_u
\]
Since $\mathcal{T}\Sigma$ and $\mathcal{W}_u$ are locally free sheaves, $\widetilde{\mathcal{N}}_u:=\mathcal{W}_u/\mathrm{d}u(\mathcal{T}\Sigma)$ is a coherent sheaf on $\Sigma$. 
The quotient of $\widetilde{\mathcal{N}}_u$ by the torsion subsheaf, 
\[
\mathcal{N}_u:=\widetilde{\mathcal{N}}_u/\mathrm{Tor}(\widetilde{\mathcal{N}}_u)
\]
is a locally free sheaf. Hence, a smooth vector bundle is associated with it, denoted by $N_u$, which is called the \textbf{generalized normal bundle} of $u$.
Kernel of the quotient map 
\[
\mathcal{W}_u \rightarrow \mathcal{N}_u
\]
is also a locally free sheaf and its corresponding vector field, denoted by $T_u$, is called the \textbf{generalized tangent bundle} of $u$.
Note that when $u:\Sigma \rightarrow M$ is an embedding, $T_u$ is the image of tangent bundle of $\Sigma$ and $N_u=W_u/T_u$ is the \textbf{normal bundle} of $u$.

Let $V$ be a vector bundle over $\Sigma$, where $V$ can represent $W_u,T_u$ or $N_u$. and let $C^{j,\alpha}(\Sigma,V)$ denote the vector space of $C^{j,\alpha}-$sections of $V$ which is a Banach space.
The $C^{j,\alpha}-$norm is defined as follows:
\begin{align*}
	\|s\|_{j,\alpha}:=\|s\|_{j-1}+\|\nabla^js\|_{0,\alpha}.
\end{align*}
Here $\nabla$ denote a smooth Riemannian connection
and
\begin{align*}
	\|s\|_{j}:=\|s\|_{0}+\|\nabla s\|_{0}+...+\|\nabla^js\|_{0}.
\end{align*}
Moreover, for a map $f:U\rightarrow\mathbb{C}^n$, $U\subset \mathbb{C}$,
\begin{align*}
	\|f\|_{0,\alpha}:=\|f\|_{0}+\mathrm{sup}\{\frac{|f(x)-f(y)|}{|x-y|^\alpha}\,|\,x,y\in U, x\ne y\}.
\end{align*}

In \cite{White-1991}, White investigates how the critical points of the area functional vary with metric. To state this result, let
\begin{align*}
	A:\mathcal{U} \times &C^{j,\alpha}(\Sigma,V)\rightarrow \mathbb{R} \\
	A(g,s)&=\int_{\Sigma} A_g(x,s(x),\nabla s(x)) \; \mathrm{d}x
\end{align*}
be the area functional, where $\mathcal{U}$ is an open subset of $C^k-$metrics on $M$, $k\ge 3$.

\begin{thm}[\cite{White-1991}, Theorem 1.1] \label{thm:c1}
	Let $\Sigma, V$ and $\nabla$ be as above.
	\begin{enumerate}
		\item The rate of change of the area functional is expressed as follows:
		\[
		\frac{d}{dt}\Big|_{t=0} 
		\int_{\Sigma} A_g(x,s_t(x),\nabla s_t(x)) \; \mathrm{d}x=
		\int_{\Sigma} H(g,s_0)(x).\frac{d}{dt}s_0(x) \; \mathrm{d}x
		\]
		where $s_t\in C^{2,\alpha}(\Sigma,V)$ is a differentiable one parameter family of sections and 
		\[
		H(g,s)=-\mathrm{div}\, D_3A_g(x,s,\nabla s)+D_2A_g(x,s,\nabla s)
		\]
		\item  Let $\mathcal{U}$ be an open subset of $C^k-$metrics on $M$, $k\ge 3$. If $2\le j\le k$ and $s\in C^{j,\alpha}(\Sigma,V)$, then $H(g,s)\in C^{j-2,\alpha}(\Sigma,V)$ and the map
		\[
		H:\mathcal{U}\times C^{j,\alpha}(\Sigma,V) \rightarrow C^{j-2,\alpha}(\Sigma,V)
		\]
		is $C^{k-j}$.
		\item The second order partial differential operator, $\mathcal{J}_{g_0,s_0}=D_2H(g_0,s_0)$, which is the linearization of $H$ is self-adjoint.
		\\
		Moreover, if $\mathcal{J}_{g_0,s_0}$ is an elliptic operator, then
		\item  $\mathcal{J}_{g_0,s_0}$ is an index 0 Fredholm map.
		\item For each $\alpha<1$, if $s_0$ is a solutions of $H(g_0,s_0)=0$, then $s_0$ is $C^{q,\alpha}$.
	\end{enumerate}
\end{thm}

For $u:\Sigma \rightarrow M$, let $[u]$ denote the equivalence class of all parameterizations of the surface $u(\Sigma)$, where an element of $[u]$ takes the form $u\circ \phi$ with $\phi:\Sigma \rightarrow \Sigma$ being a diffeomorphism.

The space of minimal surfaces is defined as
\[
\mathscr{L}=\{(g,u) \in \mathcal{U}\times C^{j,\alpha}(\Sigma,M) \, |\, u \; \text{is a}\; g-\text{stationary immersion}\}.
\] 
Moreover, let $\tilde{\mathscr{L}}:=\mathscr{L}/\sim$, where $\sim$ is the equivalence relation define above.

Near the multiple coverings, i.e. a map like $u=v\circ \phi$ where $\phi:\Sigma \rightarrow \Sigma^\prime$ is a covering map of surfaces of degree $d>1$ and $v:\Sigma^\prime \rightarrow M$, this space may fail to be a manifold. This arises due to the potential convergence of a sequence of minimal embeddings toward a multiple covering.
However, if we consider the subspace of "almost embedded" minimal surfaces, denoted by $\mathscr{Y} \subset \mathscr{L}$, a theorem of White, \cite{White-1991}, asserts that $\tilde{\mathscr{Y}}:=\mathscr{Y}/\sim$ has a manifold structure. 
Here, an \textbf{almost embedded surface} refers to a surface that is immersed and has multiplicity one on an open dense subset.
\begin{thm}[\cite{White-1991}, Theorem 2.1] \label{thm:c2}
	Let $\mathcal{U}$ be an open subset of $C^k-$Riemannian metrics on $M$, $k\ge 3$ and let $\Sigma$ be a Riemannian manifold with $\mathrm{dim}\,\Sigma<\mathrm{dim}\,M$. 
	Then $\tilde{\mathscr{Y}}:=\mathscr{Y}/\sim$ is a $C^{k-j}$ Banach manifold modeled on $\mathcal{U}$, where
	\[
	\mathscr{Y}=\{(g,u) \in \mathcal{U}\times C^{j,\alpha}(\Sigma,M) \, |\, u \; \mathrm{is \; a}\; g-\mathrm{stationary \; almost \; embedding}\}.
	\] 
	Moreover the map
	\begin{align*}
		&\Pi:\tilde{\mathscr{Y}}\rightarrow \mathcal{U} \\
		&\Pi(g,[u])=g
	\end{align*}
	is a $C^{k-j}$ Fredholm map with index $0$.
\end{thm}
If $g$ is a regular value for $\Pi$, then the space of $g-$almost embedded minimal surfaces 
\[
\tilde{\mathscr{Y}}(g):=\Pi^{-1}(g)
\]
is a $0-$dimensional submanifold of $\tilde{\mathscr{Y}}$. 

The obstruction for $\tilde{\mathscr{Y}}(g)$ to be a manifold is the non-triviality of the kernel of $D\Pi(g,[v])$. 
Theorem \ref{thm:c2} implies this kernel is equal to the maximum number of linearly independent $g-$Jacobi fields on $[v]$ (see below).
Note that a $g-$Jacobi field along $u$ is, by definition, a section of $N_u$ that belongs to the kernel of $\mathcal{J}_{g,u}=D_2H(g,u)$. 


\begin{defi}
	Let $g\in \mathscr{G}=\mathscr{G}_k$ be a $C^k-$Riemannian metric and $u:\Sigma \rightarrow (M,g)$ be a $g-$stationary map. The dimension of $\mathrm{Ker}\, D_2H(g,u)$ is called the \textbf{nullity} of $u$ and is denoted by $\pmb{\nu}(g,u)$. The number of negative eigenvalues of $D_2H(g,u)$ is called the \textbf{index} of $u$ and is deonted by $\pmb{\iota}(g,u)$.
\end{defi}

With this definition, the dimension of $\mathrm{Ker}\,D\Pi(g,[u])$ can be expressed as follows:
\begin{thm}[\cite{White-1991}, Theorem 2.1] \label{thm:c3}
	Let $\Pi$ be the map introduced in Theorem \ref{thm:c2}. Then for $(g,[u]) \in \tilde{\mathscr{Y}}$, 
	\[
	\mathrm{dim}\;\mathrm{Ker}\,D\Pi(g,[u])=\mathrm{nullity\; of\;} [u]\; \mathrm{with\;respect\;to\;} g.
	\]
\end{thm}

A $C^k-$Riemannian metric $g\in\mathscr{G}=\mathscr{G}_k$ on $M$ with the property that every $g-$stationary map $u:\Sigma \rightarrow (M,g)$ has zero nullity with respect to $g$, meaning there are no non-zero Jacobi vector field along $u$, is called a \textbf{bumpy metric}. 
Note that if $u$ has no non-zero Jacobi field on it then $[u]\in\tilde{\mathscr{L}}(g)$ is isolated. Therefore, if $g$ is a bumpy metric, then the space of $g-$minimal surfaces 
\[
\tilde{\mathscr{L}}(g)=\{(g,[u])\in\tilde{\mathscr{L}}\,|\, u:(\Sigma,h)\rightarrow (M,g) \; \mathrm{is}\; g-\mathrm{stationary}\}
\]
is discrete. This enables us to enumerate the elements of $\tilde{\mathscr{L}}(g)$ inside a compact open subset. 

According to White's Theorem we know that almost every metric is bumpy. 

\begin{thm}[\cite{White-1991}, Theorem 2.2 and \cite{White-2017}, Theorem 2.1] \label{thm:c4}
	Let $M$ be a smooth manifold. There exists a Bair subset $\mathscr{G}_k^*\subset\mathscr{G}_k$ which consists of bumpy $C^k-$metrics, where $k\ge3$ or $k=\infty$. 
\end{thm}

Since the space of $C^\infty-$metrics is not a Banach manifold, but a Frechet manifold, one can not use directly Theorem \ref{thm:c2} to show that bumpy smooth metrics are comeager.
In \cite{White-2017}, White use an indirect method to prove Theorem \ref{thm:c4} for smooth metrics. 
There is another point of view, based on Wendl's method for proving super-rigidity conjecture, which can be applied here to give another proof of Theorem \ref{thm:c4}.

To explain it note that if $g\in\mathscr{G}^*=\mathscr{G}^*_k$, since there are no non-zero Jacobi field along every closed $g-$minimal immersed submanifold of $M$, then the kerenel of this operator is trivial.
This motivates the following definition:

\begin{defi}
	Let $(M,g)$ be a Riemannian manifold. A non-constant map $u:\Sigma\rightarrow (M,g)$ is \textbf{rigid} if $\mathrm{Ker} (\mathcal{J}_{g,u})=0$.
\end{defi}
Non-triviality of the kernel of this operator is a condition we do not want to face.

An immersed minimal map $u:(\Sigma,h) \rightarrow (M,g)$ 
factors through a smooth Riemann surface $\Sigma'$ via an almost minimal embedding $v:(\Sigma',h') \rightarrow (M,g)$ and a holomorphic covering map $\pi: (\Sigma,j) \rightarrow (\Sigma',j')$ so that $u=v\circ \pi$.
Here, $j$ and $j'$ are the unique complex structures induced by $h$ and $h'$, respectively, (See \ref{Branched Imm.}).
Moreover, by Proposition 2.6.3 of \cite{DoanWalpuski}
\begin{equation} \label{pullback of J.Operator}
	\mathcal{J}_{g,v\circ \pi}=\pi^* \mathcal{J}_{g,v}.	
\end{equation}
Equation (\ref{pullback of J.Operator}) suggests employing $\tilde{\mathscr{Y}}(g)$ to explore $\tilde{\mathscr{L}}(g)$.
Furthermore, it implies that the appropriate criterion for an immersed minimal surfaces to be isolated in $\tilde{\mathscr{L}}(g)$ is the triviality of the kernel of the Jacobi operator for all coverings of the embedded minimal surfaces i.e.  $\mathrm{Ker}(\mathcal{J}_{g,v\circ \pi})=0$, for each embedding $v$ and each covering map $\pi$. 
This motivates following definitions:
\begin{defi}
	Let $(M,g)$ be a Riemannian manifold. An embedded minimal map $v:(\Sigma,h)\rightarrow (M,g)$ is \textbf{super-rigid} if $v$ and all of its multiple covers are rigid.
\end{defi}

\begin{defi}\label{Super-rigidMetric}
	A smooth metric $g$ is called \textbf{super-rigid} if every element of $\mathscr{Y}(g)$ is super-rigid.
	The subspace of super-rigid metrics is denoted by $\mathscr{G}^\bullet$.
\end{defi}
 
In the context of counting closed geodesics on a Riemannian manifold, Eftekhary has demonstrated that $\mathscr{G}^\bullet$ forms a comeager subset of $\mathscr{G}$ (see \cite{EF-CLosedGeodesics}).
His approach to studying immersed geodesics in $M$ involves stratifying the space
\[
\mathscr{W}:=\{(g,u,\xi)\,|\, u \in \mathscr{Y}(g)\quad \mathrm{and} \quad 0\ne\xi \in \mathrm{Ker}(\mathcal{J}_{g,u})\}
\]
into submanifolds.
A similar stratification technique is employed by Wendl to prove the super-rigidity conjecture for $J-$homolorphic curves in a symplectic manifold. 
Indeed he proves the subspace of super-rigid almost complex structures is comeager.

By applying Wendl's method to minimal surface, we give a new way to prove that $\mathscr{G}^\bullet$ is also comeager. 

\begin{rem}
	Let $m\in\mathbb{N}\cup\{0\}$, $\Sigma$ a surface and $(M,g)$ be a Riemannian manifold. set
	\[
	\mathcal{M}_m(\Sigma,g):=\{(\Sigma,u,(z_1,...,z_m))\}/ \sim 
	\]
	where, $u:\Sigma\rightarrow(M,g)$ is a $g-$minimal map and $(z_1,...,z_m)$ is an ordered set of $m$ distinct points on $\Sigma$.
	Moreover,
	$(\Sigma,u,(z_1,...,z_m))\sim (\Sigma,u',(z'_1,...,z'_m))$ if and only if there exits a diffeomorphism $\phi:\Sigma\rightarrow\Sigma'$ such that $u=u'\circ \phi$ and $\phi(z_i)=z'_i$, $i=1,...,m$.
	Note that $\mathrm{dim}\,\mathcal{M}_m(\Sigma,g)=2m$.
	
    Let $\Delta=\{(p,p)|\, p\in M\}$ and
    \begin{align*}
    	\mathrm{ev}:& \mathcal{M}_1(S^2,g) \times \mathcal{M}_1(\Sigma,g) \rightarrow M\times M\\
    	& ([(S^2,u_1,z_1)],[(\Sigma,u_2,z_2)]) \mapsto (u_1(z_1),u_2(z_2)).
    \end{align*}
    If $\Delta$ is transverse to $\mathrm{ev}$, then
    \[
    \mathcal{M}_+:=\{([(S^2,u_1,z_1)],[(\Sigma,u_2,z_2)])\in \mathcal{M}_1(S^2,g) \times \mathcal{M}_1(\Sigma,g) |\, u_1(z_1)=u_2(z_2)\}
    \]
    is a submanifold. If $\mathrm{dim}\, M=n$ then
    \begin{align*}
    \mathrm{dim}\,\mathcal{M}_+& =\mathrm{dim}\,\mathcal{M}_1(S^2,g)+\mathrm{dim}\,\mathcal{M}_1(S^2,g)-\mathrm{dim}\, M \\
    &=2+2-n \\
    &=4-n.
    \end{align*}
    Therefore, If $\mathrm{dim}\, M\ge 5$, then there is no bubbling phenomenon as we compactify the moduli space of minimal surfaces.
\end{rem}
	\section{Transversality} \label{sec:Tranversality}
Bryan and Pandharipande in \cite{BryanPand} proposed a conjecture which asserts that super-rigidity for almost complex structures is generic.
Recently, Wendl has proved this conjecture \cite{Wendl}. In his work, the notion of transversality plays a crucial role. 

Holomorphic curves are solutions of the Cauchy-Riemann operator which is an elliptic operator that acts on sections of a vector bundle over a manifold. 
Elliptic operators between suitable Sobolev spaces on a compact manifold are usually Fredholm. Consequently, in many situations this leads us 
to the study of Fredholm operators. 

Let $\mathcal{F}(X,Y)$ denote the space of Fredholm operators defined between two Banach spaces $X$ and $Y$.
This space can be stratified by submanifolds determined by the dimensions of the kernel and cokernel.:
\[
\mathcal{F}_{d,e}:=\{L \in \mathcal{F}(X,Y) : \mathrm{dim} \, \mathrm{ker} \, L=d \quad \mathrm{and} \quad \mathrm{dim} \, \mathrm{coker} \, L=e \}.
\]
\begin{defi}\label{def:FamilyOfOoperators}
	Let $\mathscr{V}$ denote a Banach manifold. A family of elliptic operators of order $k$, parametrized by $\mathscr{V}$, consists of a smooth map
	\[
	D:\mathscr{V} \rightarrow \mathcal{F}(X,Y).
	\]
\end{defi}
Doan and Walpuski in \cite{DoanWalpuski} investigate questions concerning the smoothness and codimension of subsets $D^{-1}(\mathcal{F}_{d,e})$. In geometric applications $X$ and $Y$ represent the spaces of sections of some vector bundles $E$ and $F$ over a manifold $Z$.

The idea of stratification, using conditions on kernel and cokernel, in investigating transversality results in moduli spaces is not new and was initiated by Taubes \cite{Taubes}.  It has also been employed by Eftekhary \cite{EF} and later by Wendl \cite{Wendl} in the context of super-rigidity conjecture.
As geodesics and minimal submanifolds of a Riemannian manifold are solutions of an elliptic operator, namely the Jacobi operator, these ideas are applicable here.

\subsection{Transversality Theorem}\label{Subsection:TransversalityTheorem}
In this subsection, we review the definitions of Petri's conditions and the flexibility condition which imply transversality (Section $1.1$ of \cite{DoanWalpuski}).

Let $M$ be a smooth manifold and $E$ and $F$ be two real vector bundles on it equipped with orthogonal covariant derivatives. Let $j\in\mathbb{N}_0$ and $W^{k,p}(E)$ denote the Sobolev completion of the space of smooth sections of $E$. Let $j,\alpha$, $0<\alpha<1$ be such that
\[
j+\alpha=k-\frac{2}{p}.
\]
If $pk>2$, then the following embedding holds
\[
W^{k,p}(E)\subset C^{j,\alpha}(E).
\]

Consider a family of linear elliptic differential operators of order $k$ \begin{align}\label{FamilyOfOperators}
D:\mathscr{V}\rightarrow\mathcal{F}(W^{k,p}(E),W^{0,p}(F))
\end{align}
where $\mathscr{V}$ is a Banach manifold.
Let
\[
\mathscr{V}_{d,c}:=\{\mathfrak{v}\in\mathscr{V}\,|\,
\text{dim}(\text{Ker}\,D_\mathfrak{v})=d\quad\text{and}\quad \text{dim}(\text{Coker}\,D_\mathfrak{v})=c
\}
\]
where $d,c\in\mathbb{N}_0$ and $D_\mathfrak{v}=D(\mathfrak{v})$. 
We want to know when $\mathscr{V}_{d,c}$ is a submanifold and then what its codimension is.
To that end, first we recall a describtion of a neighborhood of a Fredholm operator in the space of Flredholm operators. 

Let $X$ and $Y$ be Banach spaces and $L\in\mathcal{F}(X,Y)$. 
Since $L$ is Fredholm, there are splittings $X=V\oplus K$ and $Y=I\oplus C$, where $K=\text{Ker}\,L$, $C=\text{Coker}\,L$ and $I=\text{Im}\,L$. 
This implies that $L|_{V}:V\rightarrow I$ is an isomorphism.
With respect to these decompositions, every $T\in\mathcal{F}(X,Y)$ has the form
\[
\begin{pmatrix}
	A&B\\
	C&D
\end{pmatrix}.
\]

By choosing a sufficiently small neighborhood of $L$,  denoted as $\mathcal{U}$, one can assume that for each $T\in \mathcal{U}$, $A:V\rightarrow I$ is an isomorphism.
Define
\begin{align*}
	& \mathscr{I}:\mathcal{U}\rightarrow\text{Hom}(\text{Ker}\,L,\text{Coker}\,L) \\
	& \mathscr{I}(T):=D-CA^{-1}B.
\end{align*}
It follows that for 
$\hat{\text{T}}=\begin{pmatrix}
	\mathfrak{a}&\mathfrak{b}\\
	\mathfrak{c}&\mathfrak{d}
\end{pmatrix} \in T_L\mathcal{U}=T_L\mathcal{F}(X,Y)$

\begin{equation}\label{DiffOfFredholm}
	\begin{split}
		& \mathrm{d}_{L}\mathscr{I}:T_L\mathcal{F}(X,Y)\rightarrow\text{Hom}(\text{Ker}\,L,\text{Coker}\,L) \\
		& \mathrm{d}_{L}\mathscr{I}(\hat{\text{T}})(x)=\mathfrak{d}\,x=\hat{\text{T}}x \quad \text{mod} \quad \text{Im}\,L.
	\end{split}
\end{equation}

Evidently, $\mathrm{d}_{L}\mathscr{I}$ is surjective. For $T$, with a matrix presentation as above, let
\[
\Psi=\begin{pmatrix}
	\text{Id}&-A^{-1}B\\
	0&\text{Id}
\end{pmatrix},
\quad\text{and}\quad
\Phi=\begin{pmatrix}
	A^{-1}&0\\
	-CA^{-1}&\text{Id}
\end{pmatrix},
\]
then 
$$\Phi T\Psi=\begin{pmatrix}
	\text{Id}&0\\
	0&\mathscr{I}(T)
\end{pmatrix}.$$ 
This implies that $\text{Ker}\,T\cong\text{Ker}\,\mathscr{I}(T)$ and
$\text{Coker}\,T\cong\text{Coker}\,\mathscr{I}(T)$.
$\mathscr{I}(T)$ is defined on $K$ and the dimension of $K$ is $\text{dim}(\text{Ker}\,L)$. Therefore, 
$\text{dim}(\text{Ker}\,T)\le \text{dim}(\text{Ker}\,L)$ and equality holds if and only if $\mathscr{I}(T)=0$. 
Similarly, $\text{dim}(\text{Coker}\,T)= \text{dim}(\text{Coker}\,L)$ if and only if  $\mathscr{I}(T)=0$.
If we let 
$$L\in \mathcal{F}_{d,c}=\{T:X\rightarrow Y\,|\,\text{dim}(\text{Ker}\,T)=d\quad\text{and}\quad \text{dim}(\text{Coker}\,T)=c\}$$
then implicit function theorem implies that 
\[
\mathcal{F}_{d,c}\cap\mathcal{U}=\mathscr{I}^{-1}(0)
\]
is a smooth submanifold of the space of bounded linear maps from $X$ to $Y$, i.e. $\mathscr{L}(X,Y)$, of codimension $\text{dim}\big(\mathrm{Hom}(\text{Ker}\,L,\text{Coker}\,L)\big)=dc$.

Returning to the family of operators (\ref{FamilyOfOperators}),our focus lies on the Brill-Noether locus $\mathscr{V}_{d,c}$.
The ensuing discussion is true for any family of linear elliptic differential operator.
Let $\mathfrak{v}\in\mathscr{V}$ and set $L=D_\mathfrak{v}$. For an appropriate neighborhood $\mathcal{U}$ of $L$, consider the map 
\begin{align*}
	\mathscr{I}:\mathcal{U}\rightarrow\text{Hom}(\text{Ker}\,L,\text{Coker}\,L) 
\end{align*}
defined as above.
If $0$ is a regular value for $\mathscr{I}\circ D$, then by the Regular Value Theorem, there is a neighborhood $\mathscr{U}$ of $\mathfrak{v}\in\mathscr{V}$ such that $\mathscr{V}_{d,c}\cap \mathscr{U}$ is a submanifold of codimension $dc$.
Clearly, $0$ is a regular value of $\mathscr{I}\circ D$ if 
\begin{align*}
	& \Theta_\mathfrak{v}:T_\mathfrak{v}\mathscr{V}\rightarrow\text{Hom}(\text{Ker}\,D_\mathfrak{v},\text{Coker}\,D_\mathfrak{v}) \\
	& \Theta_\mathfrak{v}(\hat{X})\,x:=\text{d}_\mathfrak{v}D(\hat{X})\,x
	\quad\text{and}\quad\text{Im}\,D_\mathfrak{v}
\end{align*}
is surjective. More generally,
\begin{thm}[Theorem $1.1.5$ of \cite{DoanWalpuski}]\label{thm:Transversality}
	Let $\{D_\mathfrak{v}\}_{\mathfrak{v}\in\mathscr{V}}$ denote a family of linear differential operators and let $d,c\in\mathbb{N}_0$.
	If $\Theta_\mathfrak{v}$ defined by
	\begin{align*}
		& \Theta_\mathfrak{v}:T_\mathfrak{v}\mathscr{V}\rightarrow\mathrm{Hom}(\mathrm{Ker}\,D_\mathfrak{v},\mathrm{Coker}\,D_\mathfrak{v}) \\
		& \Theta_\mathfrak{v}(\hat{X})\,x:=\mathrm{d}_\mathfrak{v}D(\hat{X})\,x
		\quad\mathrm{mod}\quad\mathrm{Im}\,D_\mathfrak{v}
	\end{align*}
    is surjective,
    then there is a neighborhood $\mathscr{U}$ of $\mathfrak{v}\in\mathscr{V}$ such that $\mathscr{V}_{d,c}\cap \mathscr{U}$ is a submanifold of codimension $dc$.
\end{thm}

According to Proposition 1.1.12 of \cite{DoanWalpuski}, investigating the surjectivity of  $\Theta_\mathfrak{v}$ breaks down into exploring two conditions called 
\textbf{flexibility} and \textbf{Petri's condition}.
 To elucidate this (see proof of Proposition 1.1.12 of \cite{DoanWalpuski}), let $U\subset \Sigma$ be an open subset and let $\Gamma_c(U,\mathrm{Hom}(E,E))$ denote the space of smooth sections with compact support.
There is an evaluation map for $\mathfrak{v}\in\mathscr{V}$ defined as follows
\begin{align*}
	\mathrm{ev}_\mathfrak{v}:&\Gamma_c(U,\mathrm{Hom}(E,E))\rightarrow
	\text{Hom}(\text{Ker}\,D_\mathfrak{v},\text{Coker}\,D_\mathfrak{v}) \\
	& \mathrm{ev}_\mathfrak{v}(A)x:=Ax \quad\mathrm{mod}\quad\text{Im}\,D_\mathfrak{v}
\end{align*}
If the image of $\Theta_\mathfrak{v}$ contains the image of $\Theta_\mathfrak{v}$, then the surjectivity of $\Theta_\mathfrak{v}$ implies the surjectivity of $\Theta_\mathfrak{v}$.  
the first condition i.e. $\text{Im}\,\mathrm{ev}_\mathfrak{v} \subset\text{Im}\,\Theta_\mathfrak{v}$ is called the flexibility condition.
\begin{defi}[Definition 1.1.9 of \cite{DoanWalpuski}]
	Let $U\subset \Sigma$ be an open set. A family of linear elliptic differential operators $\{D_\mathfrak{v}\}_{\mathfrak{v}\in\mathscr{V}}$ is called to be \textbf{flexible} in $U$ at $\mathfrak{v}\in\mathscr{V}$ if for every $A\in \Gamma_c(U,\mathrm{Hom}(E,E))$ there exists a $\hat{X}\in T_\mathfrak{v}\mathscr{V}$ such that for every $x\in\mathrm{Ker}\,D_\mathfrak{v}$
	\[
	\mathrm{d}_\mathfrak{v}D(\hat{X})\,x=A\,x 
	\quad\mathrm{mod}\quad\mathrm{Im}\,D_\mathfrak{v}.
	\]
\end{defi}

Let $E^\dag:=E^*\otimes\Lambda^nT^*\Sigma$ and let $\boldsymbol{\pmb{\langle}}.\,,.\boldsymbol{\pmb{\rangle}}$ denote the inner product for sections of $E$ and $E^\dag$ i.e.
\[
\boldsymbol{\pmb{\langle}}s,t\boldsymbol{\pmb{\rangle}}:=\int_{\Sigma}\langle s,t\rangle
\]
for $s\in\Gamma(E)$ and $t\in\Gamma(E^\dag)$. 
Note that 
\begin{align*}
	& \langle .\,,.\rangle:E\otimes E^\dag\rightarrow\Lambda^nT^*\Sigma \\
	& \langle .\,,.\rangle(s\otimes t)=\langle s,t\rangle=t(s).
\end{align*}
The formal adjoint of $D$ with respect to $\boldsymbol{\pmb{\langle}}.\,,.\boldsymbol{\pmb{\rangle}}$ is denoted by $D^\dag$, thus
\[
\boldsymbol{\pmb{\langle}}s,D^\dag t\boldsymbol{\pmb{\rangle}}=
\boldsymbol{\pmb{\langle}}Ds,t\boldsymbol{\pmb{\rangle}}.
\]
Note that $\text{Ker}\,D_\mathfrak{v}^\dag\cong (\text{Coker}\,D_\mathfrak{v})^*$. 
Let $\alpha_1,...,\alpha_s$ be a basis for $\text{Ker}\,D_\mathfrak{v}$ and $\beta_1,...,\beta_t$ be a basis for $\text{Ker}\,D_\mathfrak{v}^\dag$.
For every $\alpha_i$, $\mathrm{ev}_\mathfrak{v}(A)\,\alpha_i-A\,\alpha_i$ is orthogonal to every $\beta_j$, with respect to $\boldsymbol{\pmb{\langle}}.\,,.\boldsymbol{\pmb{\rangle}}$.
Therefore,
\[
\boldsymbol{\pmb{\langle}}
\mathrm{ev}_\mathfrak{v}(A)\,\alpha_i,\beta_j
\boldsymbol{\pmb{\rangle}}=
\boldsymbol{\pmb{\langle}}
A\,\alpha_i,\beta_j
\boldsymbol{\pmb{\rangle}}
\]
which means that the matrix elements of the linear map 
$\mathrm{ev}_\mathfrak{v}(A)$
are determined by 
$\boldsymbol{\pmb{\langle}}
A\,\alpha_i,\beta_j
\boldsymbol{\pmb{\rangle}}$.
Then $\mathrm{ev}_\mathfrak{v}$ is not surjective if and only if there exist non-trivial constants $\gamma_{ij}\in\mathbb{R}$ such that for every $A\in \Gamma_c(U,\mathrm{Hom}(E,E))$
\begin{align*}
	0=\sum_{i,j}\gamma_{ij}
	\boldsymbol{\pmb{\langle}}
	A\,\alpha_i,\beta_j
	\boldsymbol{\pmb{\rangle}}
	&=
	\boldsymbol{\pmb{\langle}}
	.\,,.
	\boldsymbol{\pmb{\rangle}}
	\circ
	(A\otimes\text{Id})
	\big(
	\sum_{i,j}\gamma_{ij}\alpha_i\otimes\beta_j
	\big) \\
	&=\int_{U}\langle .,. \rangle \circ
	(A\otimes \mathrm{Id})  
	\big(
	\sum_{i,j}\gamma_{ij}\alpha_i\otimes\beta_j
	\big)
	d\mathrm{vol}
\end{align*}
If 
$\sum_{i,j}\gamma_{ij}\alpha_i\otimes\beta_j\in \Gamma(U,E\otimes E^\dag)$ 
is nonzero, then there exists 
$$B\in \Gamma_c(U,\mathrm{Hom}(E,E))$$
such that the above equality is nonzero. Consequently, $\mathrm{ev}_\mathfrak{v}$ will be surjective.
Note that
$\sum_{i,j}\gamma_{ij}\alpha_i\otimes\beta_j$
belongs to the image of the following map
\begin{align*}
	\mathfrak{P}:& \text{Ker}\,D_\mathfrak{v}\otimes \text{Ker}\,D_\mathfrak{v}^*
	\rightarrow \Gamma(U,E\otimes E^\dag) \\
	& \mathfrak{P}(s\otimes t)(x):=s(x)\otimes t(x)
\end{align*}

This map is induced by \textbf{Petri map}
\begin{align*}
	\mathfrak{P}:& \Gamma(E)\otimes \Gamma(F^\dag)
	\rightarrow \Gamma(E\otimes E^\dag) \\
	& \mathfrak{P}(s\otimes t)(x):=s(x)\otimes t(x).
\end{align*} 
Therefore, the surjectivity of $\mathrm{ev}_\mathfrak{v}$ is equivalent to the injectivity of the Petri's map.
\begin{defi}
	Let $U\subset\Sigma$ be an open subset. The linear elliptic differential operator $D:W^{k,2}\Gamma(E)\rightarrow L^2\Gamma(E)$ satisfies \textbf{Petri's condition in $U$} if the Petri's map
	\begin{align*}
		\mathfrak{P}:& \mathrm{Ker}\,D\otimes \mathrm{Ker}\,D^*
		\rightarrow \Gamma(U,E\otimes E^\dag) \\
		& \mathfrak{P}(s\otimes t)(x):=s(x)\otimes t(x)
	\end{align*}
	is injective.
\end{defi}

	\subsection{Petri's Condition}\label{Subsection:PetriCondition}
In this section we prove that the locus where Petri's condition for the Jacobi operator is violated, has positive codimension. 
It turns out that it suffices to work with for a smaller subset consisting of elements for which a simpler version of this condition, known as the polynomial Petri condition, is satisfied.
To that end, let us first recall what a jet space is.

Let $\Sigma$ be a smooth manifold and  $E$ be a vector bundle on $\Sigma$.
For $x\in\Sigma$, let $\mathscr{E}_x$ denote the stalk of the sheaf of sections of $E$, defined as
\[
\{(s,U)\,|\,U\subset\Sigma \,, x\in U \;\;\text{and}\;\; s\in\Gamma(U,E)\}/\sim
\]
where $(s,U)\sim(t,V)$ if there is an open subset $W\subset U\cup V$ such that $x\in W$ and $s|_{W}=t|_{W}$.

There is a natural filtration on $\mathscr{E}_x$ defined by considering the vanishing order of elements of this space.
To define it, let us recall the definition of the intrinsic derivative. 
Let $f:F\rightarrow E$ be a vector bundle map over $\Sigma$, and choose some trivializations of them on some neghiborhood of $x\in U\subset \Sigma$, i.e., let $F|_{U}\cong U\times\mathbb{R}^k$ and $E|_{U}\cong U\times \mathbb{R}^l$.
Then $f$ can be viewed as
\[
f:U\rightarrow\text{Hom}(\mathbb{R}^k,\mathbb{R}^l).
\]
The \textbf{intrinsic derivative} of $f$ at $x$,$\text{d}_x f$, is defined as the composition map
\begin{align}\label{equ:IntrinsicDer}
T_x\Sigma\rightarrow \text{Hom}(\mathbb{R}^k,\mathbb{R}^l)\rightarrow\text{Hom}(\mathrm{Ker}\,f(x),\mathrm{Coker}\,f(x))
\end{align}
This definition does not depend on the choice of trivializations of $E$ and $F$.
As a special case, if $\rho:\Sigma\rightarrow \Sigma'$ is a smooth map, then for $F=T\Sigma$ and $E=\rho^*T\Sigma'$, $T\rho:F\rightarrow E$ is a vector bundle map over $\Sigma$, and its intrinsic derivative at $x\in \Sigma$ is
\[
\mathrm{d}_x(\mathrm{T}\rho):T_x\Sigma\rightarrow
\text{Hom}(\mathrm{Ker}\,T_x\rho,\mathrm{Coker}\,T_x\rho)
\]
which can be viewed as a map
\begin{align}\label{equ:SecondIntrinsic}
T_x\Sigma\otimes \mathrm{Ker}\,T_x\rho \rightarrow \mathrm{Coker}\,T_x\rho
\end{align}
that is symmetric when restricted to $\mathrm{Ker}\,T_x\rho \otimes \mathrm{Ker}\,T_x\rho$.
When restricted to $\mathrm{Ker}\,T_x\rho \otimes \mathrm{Ker}\,T_x\rho$, $\mathrm{d}_x(\mathrm{T}\rho)$ is denoted by $\mathrm{d}_x^2\rho$.

Let $s:U\rightarrow E$ be a representative of $[s]\in\mathscr{E}_x$. 
Locally, $s$ can be viewed as a vector bundle map between trivial vector bundles $U\times\mathbb{R}^0$ and $U\times E_x$, where $E_x$ is the fiber over $x$.
If $s(x)=0$ in $E_x$, then by (\ref{equ:IntrinsicDer}), the intrinsic derivative is equal to $\text{T}_xs$. 
Therefore, $\text{d}_xs\in\text{Hom}(T_x\Sigma,E_x)$.
Moreover, if $\text{d}_xs=0$, then by (\ref{equ:SecondIntrinsic}), $\text{d}_x^2s$ is defined on $S^2T_x\Sigma$ and belongs to $\text{Hom}(S^2T_x\Sigma,E_x)$, where $S^2T_x\Sigma$ is the symmetric tensor product of $T_x\Sigma$ with itself.
Similarly, if $\mathrm{d}_x^js$, $j=0,1,...,k-1$, vanishes, then its $k-$th intrinsic derivative, $\mathrm{d}_x^ks$ 
, is defined and belongs to $\text{Hom}(S^kT_x\Sigma,E_x)$.

The filtration on  $\mathscr{E}_x$ is defined as follows. Let $\mathscr{E}_x^k$, $k\in\mathbb{N}_0$, denote the subspace of $\mathscr{E}_x$  consists of germs of sections $[s]$ for which $\mathrm{d}_x^js$, $j=0,1,...,k-1$, vanishes.
Moreover, for $k\in\mathbb{N}_0$ define $\mathscr{E}_x^{-k}:=\mathscr{E}_x$.
Then
\begin{align}\label{VanishingFiltration}
	...= \mathscr{E}_x^{-2} = \mathscr{E}_x^{-1} = \mathscr{E}_x^{0}=\mathscr{E}_x \supseteq \mathscr{E}_x^{1} \supseteq \mathscr{E}_x^{2} \supseteq ...
\end{align} 

\begin{defi}
	Let $E$ be a vector bundle on $\Sigma$ and let $\mathscr{E}_x$ denote the stalk of the sheaf of sections of $E$ at $x\in\Sigma$. The \textbf{$k-$jet space of $E$ at $x$} is defined as
	\[
	J_x^kE:=\mathscr{E}_x/\mathscr{E}_x^{k+1}
	\]
	and the \textbf{$\infty-$jet space of $E$ at $x$} is defined as
	\[
	J_x^\infty E:=\mathscr{E}_x/\big(\bigcap_{k\in\mathbb{Z}}\mathscr{E}_x^{k}\big).
	\]
\end{defi}

A Differential operator $D:\Gamma(E)\rightarrow\Gamma(F)$ of order $l$ with smooth coefficient induces a linear map on the jet space of sections of $E$ and $F$
\[
j_x^kD:J_x^{k+l}E\rightarrow J_x^k F
\]
where $k\in\mathbb{N}_0$. $j_x^kD$ is called $k-$jet of $D$.

The Petri's map 
\begin{align*}
	\mathfrak{P}:& \Gamma(E)\otimes \Gamma(F^\dag)
	\rightarrow \Gamma(E\otimes E^\dag).
\end{align*}
also induce a map on $k-$jet spaces
\begin{align*}
	\mathfrak{P}^k:& J^k_xE\otimes J^k_xF^\dag
	\rightarrow J^k_x(E\otimes E^\dag).
\end{align*}
This map is not injective since if the order of vanishing of $s_0\in J^k_x E$ be $m<k$ and  the order of vanishing of $t_0\in J^k_x F$ be $n<k$ 
such that $m+n>k$, then the order of vanishing of $s_0\otimes t_0$ is greater that $k$, therefore, its image under $\mathfrak{P}^k$ is zero.
We are interested in an injective Petri's map. 
By Proposition $5.11$ of \cite{Wendl}, if $D$ and $D^\dag$ have strong unique continuation property at $x$ 
in some neighborhood (see below) then there is no non-trivial germ in $x$ of $\mathrm{Ker}\,D\otimes \mathrm{Ker}\,D^\dag$ which vanishes to infinite order.
This suggests the following definition.
\begin{defi}
	Let $D:\Gamma(E)\rightarrow\Gamma(F)$ be a differential operator with smooth coefficient. If
	\[
	\mathfrak{P}_{J_x^\infty D}: \ker J_x^\infty D \otimes \ker J_x^\infty D^{\dag} \rightarrow J_x^\infty (E\otimes F^\dag)
	\]
	is injective,
	we say that $J_x^\infty D:J_x^\infty E \rightarrow J_x^\infty F$ satisfies \textbf{$\infty-$jet Petri condition}.
\end{defi}

Let $x\in\Sigma$ and $U$ be an open neighborhood of $x$. A differential operator $D:\Gamma(E) \rightarrow \Gamma(F)$ has the \textbf{strong unique continuation property at $x$ in $U$} if any non-trivial solution of $Du=0$, $u\in\Gamma(U,E)$, does not vanish to infinite order at $x$. 
\begin{prop} [\cite{DoanWalpuski}, Proposition $1.6.6$]\label{J-InftySUCP}
	Let $x \in M$ and $U \subset M$ be an open neighborhood of $x$. Let $D:\Gamma(E) \rightarrow \Gamma(F)$ be a linear elliptic differential operator with smooth coefficients such that $D$ and its adjoint $D^\dagger$ have strong unique continuation property at $x$ in $U$. 
	If $J_x^\infty D$ satisfies the  $\infty-$jet Petri condition, then $D$ satisfies the $\mathcal{B}-$equivariant (and equivalently $G-$equivariant) Petri condition.
	(See Sections $1.3$ and $1.4$ of \cite{DoanWalpuski} or Appendix \ref{AppendixA} for definitions of $\mathcal{B}-$equivariant and $G-$equivariant Petri condition).
\end{prop}
Super-rigidity theorem claims that the locus of violation of Petri's condition has positive codimension. Therefore it is enough to work with the $\infty-$jet Petri condition.

According to Remark $1.6.5$ of \cite{DoanWalpuski}, a differential operator $D$ with smooth coefficients which satisfies $D^*D=\nabla^* \nabla+$lower order terms, has the strong unique continuation property at $x \in U$, whenever $U$ is connected. Therefore, Jacboi operator also has this property.

To define Petri condition in terms of polynomials let us first recall the definition of the symbol for an operation.
Let 
\begin{align*}
	&D:\Gamma(E) \rightarrow \Gamma(F)\\
	&D=\sum_{i=0}^ka_i\nabla^i
\end{align*} 
be a linear differential operator of order $k$ where $a_i$ is a smooth section of $\mathrm{Hom}(T^*\Sigma^{\otimes i}\otimes E,F)$. 
In a chart $(U,x_1,...x_n)$ in which $x$ is identified with $0$ and $E$ 
and $F$ have a trivialization, i.e. there are frames $e_1,...,e_{s}\in\Gamma(U,E)$ and $f_1,...,f_t\in\Gamma(U,F)$ then 
\begin{align*}
	D|_{U}\big(\sum_{l=1}^s\alpha_le_l\big)=\sum_{l=1}^s\sum_{h=1}^t\sum_{\text{i}}
	\mathbf{a}_{\text{i}lh}\frac{\partial^{|\text{i}|}\alpha_l}{\partial x^{\text{i}}}f_h
\end{align*}
Then the symbol of this operator is
a section of $S^kT\Sigma \otimes \mathrm{Hom}(E,F)$, where $S^kT\Sigma$ is the $k$th symmetric power of the tangent bundle of $\Sigma$.
In coordinates
\begin{align*}
	\sigma_x(D)=\sum_{|\text{i}|=k}
	\partial^{\text{i}} \otimes \mathbf{a}_{\text{i}}
\end{align*}
where $\mathbf{a}_{\text{i}}\in\mathrm{Hom}(E_x,F_x)$ is a matrix with entries $\mathbf{a}_{\text{i}lh}$ as above. 
Moreover, $\partial^{\text{i}}=\partial_1^{i_1}...\partial_n^{i_n}$ and $\{\partial_1,...,\partial_n\}$ is a basis for $T_x\Sigma$.
$\sigma_x(D)$ induces a formal differential operator
\[
\hat{\sigma}_x(D):S^\bullet T_x^*M \otimes E_x \rightarrow S^\bullet T_x^*M \otimes F_x.
\]
If $\{x_1,...,x_n\}$ be a basis for $T_x^*M$, dual to the basis $\{\partial_1,...,\partial_n\}$ such that $x_j(\partial_i)=\delta_{ij}$.
Set $\partial_ix_j:=x_j(\partial_i)$. 
Then $S^\bullet T_x^*M$ can be viewed as the polynomial ring $\mathbb{R}[x_1,...,x_n]$ which is equipped with the actions of $\partial_j'$s, i.e. $\partial_j$ acts by differentiation on polynomials.
In this basis
\[
\hat{\sigma}_x(D)=\sum_{|\text{i}|=k} 
\mathbf{a}_{\text{i}}\, \partial^{\text{i}_1}...\partial^{\text{i}_n}.
\]

For the adjoint operator, $\sigma_x(D^\dagger)=\sigma_x^\dagger(D)$, where $$\sigma_x^\dagger(D) \in S^kT_xM \otimes \mathrm{Hom}(E_x^\dagger,F_x^\dagger)$$ 
is determined as the image of $\sigma_x(D)$ under the canonical isomorphism $\mathrm{Hom}(E_x,F_x) \cong \mathrm{Hom}(E_x^\dagger,F_x^\dagger)$ times $(-1)^k$.

The polynomial Petri map is defined as follows (\cite{DoanWalpuski}, Definition 1.6.9):
\begin{align*}
	\hat{\omega}&:(S^\bullet T_x^*M \otimes E_x) \otimes (S^\bullet T_x^*M \otimes F_x^\dagger) \rightarrow S^\bullet T_x^*M \otimes E_x \otimes F_x^\dagger \\
	&\hat{\omega}((p \otimes e)\otimes (q \otimes f)):=(p.q) \otimes e \otimes f
\end{align*}
Let  $D$ be a linear differential operator. $\sigma_x(D)$ is said to satisfy the polynomial Petri condition if the restriction of $\hat{\omega}$ to $\ker \hat{\sigma}_x(D) \otimes \ker \hat{\sigma}_x^\dagger(D)$, i.e.
\[
\hat{\omega}_D:\ker \hat{\sigma}_x(D) \otimes \ker \hat{\sigma}_x^\dagger(D) \rightarrow  S^\bullet T_x^*M \otimes E_x \otimes F_x^\dagger
\]
is injective.
This type of Petri condition is related to the $\infty-$Petri condition via the following proposition:
\begin{prop}[\cite{DoanWalpuski}, Proposition 1.6.10]
	Let $D$ be a linear differential operator. If $J_x^\infty D$ does not satisfy the $\infty-$Petri condition, then $\sigma_x(D)$ fails to satisfy the polynomial Petri condition.
\end{prop}

Now we are ready to state Wendl's condition. This is a technical condition which was proved to be satisfied for the Cauchy-Riemann type operators and has been Wendel's important innovation in order to prove super-rigidity conjecture. 

Let $B \in \ker \hat{\omega}_D$ be a homogeneous element. Let
$\hat{R}$ and $\hat{R}^\dagger$ be right inverses for $\hat{\sigma}_x(D)$ and $\hat{\sigma}^\dagger_x(D)$, respectively. 
Define
\begin{equation}\label{Wendl'sLinearMap}
	\begin{split}
	\mathrm{L}_{\sigma_x(D),B}:&S^\bullet T_x^*M \otimes \mathrm{Hom}(E_x,F_x) \rightarrow S^\bullet T_x^*M \otimes E_x \otimes F_x^\dagger \\
	&\mathrm{L}_{\sigma_x(D),B}(A):=\hat{\omega}((\hat{R}A \otimes 1+1 \otimes \hat{R}^\dagger)B)
	\end{split}
\end{equation}
Note that as proved in Proposition 1.6.20 of \cite{DoanWalpuski} whenever $D$ is an elliptic operator, $\sigma_x(D)$ and $\sigma_x^\dagger(D)$ have right inverses.

Every element $B\in\mathrm{ker}(D)\otimes\mathrm{ker}(D^\dagger)$ can be written of the form
$B=\sum_{i=1}^\rho s_i\otimes t_i$
where $s_i\in\mathrm{ker}(D)$ and $t_i\in\mathrm{ker}(D^\dagger)$, $i=1,...,\rho$, are linearly independent. Then $\rho$ is defined to be the rank of $B$, i.e. $\rho=\mathrm{rk}(B)$.

\begin{defi}[\cite{DoanWalpuski}, Definition 1.6.15]\label{Wendl's Condition}
	Let be a linear differential operator of order $k$. $\sigma_x(D)$ satisfies \textbf{Wendl's condition} if for every homogeneous element $B \in \ker \hat{\omega}_D$, there exist right inverses $\hat{R}$ and $\hat{R}^\dagger$ for $\hat{\sigma}_x(D)$ and $\hat{\sigma}^\dagger_x(D)$ such that for every $l \ge l_{d,\rho}$
	\begin{align*}
		\mathrm{rk} \, \mathrm{L}_{\sigma_x(D),B}^{\le l} \ge c_{d,\rho} \, l^n.
	\end{align*}
	Here $l_{d,\rho}$ and $c_{d,\rho}$ are two constants that only depend on $d=\mathrm{deg}(B)$ and $\rho=\mathrm{rk}(B)$.
	Moreover $ \mathrm{L}_{\sigma_x(D),B}^{\le l}$ denotes the restriction of $ \mathrm{L}_{\sigma_x(D),B}$ on the space of polynomials of degree at most $l$:
	\[
	\mathrm{L}_{\sigma_x(D),B}^{\le l}:\bigoplus_{i=0}^l S^i T_x^*M \otimes \mathrm{Hom}(E_x,F_x) \rightarrow \bigoplus_{i=0}^{l+k} S^i T_x^*M \otimes E_x \otimes F_x^\dagger
	\]
\end{defi}

With this preparation, we can state Wendl's theorem that claims for a family of linear elliptic differential operators, $\{D_p\}_{p\in \mathcal{P}}$, 
the space of operators violating the $\infty-$Petri condition is small, meaning that its codimension is positive. More precisely:
\begin{thm}[Section $5.2$ of \cite{Wendl} and Theorem $1.6.17$ of \cite{DoanWalpuski}]\label{thm:Wendl'sTheorem}
	Let $\{D_p\}_{p\in \mathcal{P}}$ be a family of linear elliptic operators of order $k$ with smooth coefficients. Let
	\[
	\mathcal{R}:=\{p\in \mathcal{P} \, |\, J_x^\infty D \quad \mathrm{fails \; to \; satisfy \; the \; \infty-jet \; Petri \; condition} \}
	\]
	Fix $p_0 \in \mathcal{P}$. If for each $l\in\mathbb{N}_0$, $\{D_p\}_{p\in\mathcal{P}}$ be $l-$jet flexible at $x$ and  $p_0 \in \mathcal{P}$ and 
	$\sigma_x(D_{p_0})$ satisfies Wendl's condition,
	then for every $n\in \mathbb{N}_0$, there is a neighborhood $\mathcal{W}$ of $p_0$ such that the codimension of $\mathcal{R}\cap\mathcal{W}$ is at least $n$.

\end{thm}

\subsection{Wendl's Condition For The Jacobi Operator}\label{Subsetion:WendCondition}
The symbol of the Jacobi operator is $\Delta=\partial_1^2\otimes \mathrm{Id}+\partial_2^2 \otimes \mathrm{Id}$.
In order to investigate Wendl's condition for this operator, first let us see what the elements of $\mathrm{Ker}(\Delta)$ are.
%
\begin{lem} \label{lem:2}
	Every element in $\mathrm{ker}(\Delta)$ is a linear combination of homogeneous elements which are of the form:
	\[
	\sum_{i=0}^{d} a_i^d x_1^{d-i} x_2^{i} 
	\]
	where $a_{2i+1}^d=\frac{(-1)^i}{d} \tbinom{d}{2i+1} \, a_1^d$ and 
	$a_{2i}^d=(-1)^i \tbinom{d}{2i} \, a_0^d$.
	
\end{lem}
\begin{proof}
	Let $\sum_{i=0}^{d} a_i^d x_1^{d-i} x_2^{i}$ be a homogeneous elements in $\mathrm{ker}(\Delta)$. then
	\begin{align*}
		&\sum_{i=0}^{d-2} a_i^d(d-i-1)(d-i) x_1^{d-i-2} x_2^{i}+\sum_{i=2}^{d} a_i^d i(i-1) x_1^{d-i} x_2^{i-2} \\
		&=\sum_{i=0}^{d-2} a_i^d(d-i-1)(d-i) x_1^{d-i-2} x_2^{i}+\sum_{i=0}^{d-2} a_{i+2}^d (i+2)(i+1) x_1^{d-i-2} x_2^{i}=0.
	\end{align*} 
	This implies $a_i^d(d-i-1)(d-i)+a_{i+2}^d (i+2)(i+1)=0$. Then the claim is followed easily.
\end{proof}

From this lemma it follows that if 
$\sum_{i=0}^{d} a_i^d x_1^{d-i} x_2^{i}$
is a non-zero element in $\mathrm{ker}(\Delta)$ then either all coefficient $a_i^d$, $i$ even, must be non-zero or all coefficient $a_i^d$, $i$ odd, must be non-zero.

A straightforward calculation shows that 
	\[
	\hat{R}(x_1^m x_2^n)=\sum_{i=0}^{[n/2]} A_i x_1^{m+2+2i} x_2^{n-2i}
	+\sum_{i=0}^{[m/2]} B_i x_1^{m-2i} x_2^{n+2+2i}
	\]
is a right inverse for $\Delta$, where 
	$A_i=\frac{(-1)^i m! n!}{2 (n-2i)! (m+2+2i)!}$ and
	$B_i=\frac{(-1)^i m! n!}{2 (m-2i)! (n+2+2i)!}$.

Checking the validity of Wendl's condition requires knowing the form of homogeneous elements of $\mathrm{ker}(\hat{\omega}_\Delta)$.
\begin{lem} \label{lem:4}
	Every homogeneous element of $\mathrm{ker}(\hat{\omega}_\Delta)$ is of the form
	\[
	\sum_{l=0}^d \sum_{i=0}^{l} \sum_{j=0}^{d-l}  
	C_l a_i^l b_j^{d-l} x_1^{l-i} x_2^{i} \otimes x_1^{d-l-j} x_2^{j}   	
	\]
	such that $\sum_{l=0}^d \sum_{i=\mathrm{max}\{0,l+d^\prime-d\}}^{\mathrm{min}\{l,d^\prime\}}   
	C_l a_i^l b_{d^\prime-i}^{d-l}=0$, for every $d^\prime=0,...,d$.
\end{lem}
\begin{proof}
	An element in $\ker(\Delta)\otimes\ker(\Delta)$ is of the form
	\begin{align*}
		\sum_{l=0}^n M_l f_l(x_1,x_2) \otimes \sum_{t=0}^m N_t g_t(x_1,x_2)
	\end{align*}
	where $f_l(x_1,x_2)=\sum_{i=0}^{l} a_i^l x_1^{l-i} x_2^{i}$ and $g_t(x_1,x_2)=\sum_{j=0}^{t} b_j^t x_1^{t-j} x_2^{j}$ are homogeneous polynomials in $\ker(\Delta)$ of degrees $l$ and $t$, respectively. Therefore every element in $\ker(\hat{\omega}_\Delta)$ can be written as
	\begin{align*}
		&\sum_{l=0}^n \sum_{i=0}^{l} M_l a_i^l x_1^{l-i} x_2^{i}  \otimes \sum_{t=0}^m  \sum_{j=0}^{t} N_t b_j^t x_1^{t-j} x_2^{j} \\
		&=\sum_{l=0}^n \sum_{i=0}^{l} \sum_{t=0}^m  \sum_{j=0}^{t}  
		M_l N_t a_i^l b_j^t x_1^{l-i} x_2^{i} \otimes x_1^{t-j} x_2^{j} \\
		&=\sum_{d=0}^{n+m} \sum_{l=0}^d \sum_{i=0}^{l} \sum_{j=0}^{d-l}  
		M_l N_{d-l} a_i^l b_j^{d-l} x_1^{l-i} x_2^{i} \otimes x_1^{d-l-j} x_2^{j}   	
	\end{align*}
    Where $M_l=0$ for $l>n$ and $N_k=0$ for $k>m$.
    Here, $a_i^l=T_{il}a_{\alpha}^l$, $\alpha=i\; \mathrm{mod}\; 2$, and $b_j^{k}=R_{jk}b_{\beta}^{k}$, $\beta=j\; \mathrm{mod}\; 2$, where $T_{il}$ and $R_{jk}$ are appropriate coefficients introduced in Lemma \ref{lem:2}.
	Therefore every homogeneous element of $\mathrm{ker}(\Delta) \otimes \ker(\Delta)$ is of the form
	\[
	\sum_{l=0}^d \sum_{i=0}^{l} \sum_{j=0}^{d-l}  
	C_l a_i^l b_j^{d-l} x_1^{l-i} x_2^{i} \otimes x_1^{d-l-j} x_2^{j}. 	
	\]
	The image of this element under $\omega_\Delta$ is
	$
	\sum_{l=0}^d \sum_{i=0}^{l} \sum_{j=0}^{d-l}  
	C_l a_i^l b_j^{d-l} x_1^{d-i-j} x_2^{i+j} 
	$. 
	This belongs to $\ker(\hat{\omega}_\Delta)$ if
	\[
	\sum_{l=0}^d \sum_{i=\mathrm{max}\{0,l+d^\prime-d\}}^{\mathrm{min}\{l,d^\prime\}}   
	C_l a_i^l b_{d^\prime-i}^{d-l}=0
	\]
\end{proof}

Take a non-zero element 
$$
B=\sum_{l=0}^d \sum_{i=0}^{l} \sum_{j=0}^{d-l}  
C_l a_i^l b_j^{d-l} x_1^{l-i} x_2^{i} \otimes x_1^{d-l-j} x_2^{j} \in \ker(\omega_\Delta).
$$ 
Set 
\[
\alpha=\text{min}\{i\,|\,i\in\{0,1\}\; \text{and} \; a_i^l \; \text{is non-zero for some} \; l\}
\]
and 
\[
\beta=\text{min}\{j\,|\,j\in\{0,1\}\; \text{and} \; b_j^k \; \text{is non-zero for some} \; k\}.
\]
Let $A=x_1^m x_2^n$. 
For $\mu=1,2$, set
\[
M^\mu_{ijkl}:=C_l a_i^l b_j^{d-l} A_{lik}^\mu x_1^{l-i} x_2^{i} \otimes x_1^{d-l-j} x_2^{j}
\]
and 
\[
N^\mu_{ijkl}:=C_l a_i^l b_j^{d-l}  B_{lik}^1 
x_1^{l-i} x_2^{i} \otimes x_1^{d-l-j} x_2^{j}.
\]
Then
\begin{align*}
	(\hat{R}A\otimes 1)B=\sum_{l=0}^d \sum_{i=0}^{l} \sum_{j=0}^{d-l}  
	&\Big(
	\sum_{k=0}^{[\frac{n+i}{2}]}
	M^1_{ijkl} x_1^{m+2+2k} x_2^{n-2k} \otimes 1 \\
	&+\sum_{k=0}^{[\frac{m+l-i}{2}]}
	N^1_{ijkl} 
	x_1^{m-2k} x_2^{n+2+2k} \otimes 1
	\Big)
\end{align*}
and 
\begin{align*}
	(1\otimes \hat{R}A)B=\sum_{l=0}^d \sum_{i=0}^{l} \sum_{j=0}^{d-l}  
	&\Big(
	\sum_{k=0}^{[\frac{n+j}{2}]}
	M^2_{ijkl} 1 \otimes x_1^{m+2+2k} x_2^{n-2k} \\
	&+\sum_{k=0}^{[\frac{m+d-l-j}{2}]}
	N^2_{ijkl} 
	1 \otimes x_1^{m-2k} x_2^{2+2k}  \Big) 
\end{align*}
where
\begin{align*}
	A_{lik}^1&=\frac{(-1)^k (m+l-i)! (n+i)!}{2(m+l-i+2+2k)! (n+i-2k)!}, \\
	A_{ljk}^2&=\frac{(-1)^k (m+d-l-j)! (n+j)!}{2(m+d-l-j+2+2k)! (n+j-2k)!} 
\end{align*}
Therefore, Wendl's map, defined in (\ref{Wendl'sLinearMap}), is calculated as follows:
\begin{align*}
	L_{\Delta,B}(A)&=\hat{\omega}_\Delta (\hat{R}A \otimes 1+1 \otimes \hat{R}A)B \\
	&=\sum_{l=0}^d \sum_{i=0}^{l} \sum_{j=0}^{d-l}   C_l a_i^l b_j^{d-l} x_1^{m+d-i-j} x_2^{n+i+j} \\
	&\times \Big(
	\sum_{k=0}^{[\frac{n+i}{2}]}
	A_{lik}^1 x_1^{2+2k} x_2^{-2k}  
	+\sum_{k=0}^{[\frac{m+l-i}{2}]}
	B_{lik}^1 x_1^{-2k} x_2^{2+2k}  \\
	&+\sum_{k=0}^{[\frac{n+j}{2}]}
	A_{ljk}^2  x_1^{2+2k} x_2^{-2k} 
	+\sum_{k=0}^{[\frac{m+d-l-j}{2}]}
	B_{lik}^2 x_1^{-2k} x_2^{2+2k}  \Big) \\
\end{align*} 
To compute the rank of $L_{\Delta,,B}$, two special coefficients are important for us:
\begin{align*}
	A_{li0}^1&=\frac{1}{2(m+l-i+2)(m+l-i+1)}, \\
	A_{lj0}^2&=\frac{1}{2(m+d-l-j+2)(m+d-l-j+1)} 
\end{align*}
\begin{prop}\label{prop:Wendl'sConditionForJacobi}
	Let $B \in \mathrm{ker}(\omega_\Delta)$ be a non-zero homogeneous element of degree $d$. For $k \ge 10d+6$
	\[
	\mathrm{rk} \, L_{\Delta,B}^{\le k} \ge \frac{1}{2} k.
	\]
\end{prop}

\begin{proof}
Let 
$$
B=\sum_{l=0}^d \sum_{i=0}^{l} \sum_{j=0}^{d-l}  
C_l a_i^l b_j^{d-l} x_1^{l-i} x_2^{i} \otimes x_1^{d-l-j} x_2^{j} \in \ker(\omega_\Delta).
$$
be non-zero and set
\begin{align*}
	\tilde{A}_{ml}&=A^1_{l\alpha 0}+A^2_{l\beta 0} 
\end{align*}
Then the coefficient of $x_1^{m+d-\alpha-\beta+2} x_2^{n+\alpha+\beta}$ in $L_{\Delta,B}(A)$ is $\sum_{l=0}^dC_la_\alpha^l b_\beta^{d-l}\tilde{A}_{ml}$.

Note that $\tilde{A}_{ml}$ is of the form $\frac{P(m,l)}{p(m,l)}$. Here,
$P(m,l)=f(m)^2+g(l)^2+R$ where $f$ and $g$ are linear polynomials and $R=-\frac{1}{2}$. Moreover,
\[
p(m,l)=2(m+l-\alpha+2)(m+l-\alpha+1)(m+d-l-\beta+2)(m+d-l-\beta+1)
\]
is a polynomial of degree $4$ in $m$.
To see this, note that
\[
P(m,l)=(m+l-\alpha+2)(m+l-\alpha+1)+(m+d-l-\beta+2)(m+d-l-\beta+1)
\]
Set $E=m+l-\alpha+1$ and $F=m+d-l-\beta+1$. Then
\begin{align*}
	P(m,l)&=(E+1)E+(F+1)F \\
	      &=E^2+E+F^2+F
\end{align*}
Let $e=1-\alpha$ and $f=1+d-\beta$. Thus
\begin{align*}
	P(m,l)&=(m+l+e)^2+(m+l+e)+(m-l+f)^2+(m-l+f) \\
	&=2m^2+2l^2+e^2+f^2+2(e+f)m+2(e-f)l+e^2+f^2+e+f \\
	&=2(m+\frac{e+f+1}{2})^2+2(l+\frac{e-f}{2})^2-\frac{1}{2}
\end{align*}
It is easy to see $P(m,l)\ne 0$, for $m,l\in\mathbb{N}_0$.
Assume that
$$\sum_{l=0}^{d}C_la_\alpha^l b_\beta^{d-l} q(m,l)=0, \quad m=m_1,...,m_{4d+3}$$
where 
\[
q(m,l)=(f(m)^2+g(l)^2+R) \prod_{k=0,k \ne l}^{d} p(m,k)
\] 
is a polynomial of degree $4d+2$ obtained by taking the common denominator.
$\sum_{l=0}^{d}C_la_\alpha^l b_\beta^{d-l} q(m,l)$ is a polynomial of degree at most $4d+2$ and vanishes at $4d+3$ distinct points and thus is identically zero. But the polynomials
\[
\big\{ q(m,l) \big\}_{l=0}^{d}
\]
are linearly independent.
To see this let $\tilde{q}(m,l):=\frac{q(m,l)}{(m+l-\alpha+2)^3}$.
If $\big\{ \tilde{q}(m,l) \big\}_{l=0}^{d}$ were linearly dependent then 
\[
\sum_{l=0}^d{\tilde{C}_l(m+l-\alpha+2)^3}\tilde{q}(m,l)=0
\]
for some $\tilde{C}_l$.
By differentiating three times of the left hand side of this expression and the fact that $\tilde{q}(m,l)$ is the only polynomial which does not vanishes at the value $m=-l+\alpha-2$, it follows that $\tilde{C}_l=0$. 
It proves the linear independence of the polynomials $q(m,l)$, $l=0,...,d$.
Thus $C_la_\alpha^l b_\beta^{d-l}=0$ for $l=0,...,d$,
from which it follows that $C_l=0$. This contradicts with the fact that $B$ is non-zero. 
This contradiction, proves that for at most $4d+2$ distinct amount of $m$,
$\sum_{l=0}^{d}C_la_\alpha^l b_\beta^{d-l} \tilde{A}_{ml}$ can vanish.

Let $m+n+d+2 \le k$ and set
\[
\Omega=\{(m,0)\in \mathbb{N}_0^2 \,|\, m+d+2\le k \quad \mathrm{and} \quad \sum_{l=0}^{d}C_la_\alpha^l b_\beta^{d-l} \tilde{A}_{ml}\ne 0 \}
\]
Then the restriction of $L_{\Delta,B}^{\le k}$ on $\Omega$ is injective and according to what we said above
\[
|\Omega| \ge k-5d-3.
\]
If $k\ge 10d+2$, then $|\Omega| \ge \frac{1}{2}k$.
This completes the proof.
\end{proof}
Therefore it is enough to set $c_0(\rho,d)=\frac{1}{2}$ and $l_0(\rho,d)=5d+3$ in Definition \ref{Wendl's Condition} in order to Wendl's condition is satisfied.

	\subsection{Flexibility Condition}\label{FlexibilityConditionSection}
\subsubsection{Floer's $C^\epsilon$ Topology}
Consider a family of Jacobi operators as in Definition \ref{def:FamilyOfOoperators}, where $\mathscr{V}$ is defined to be the the universal moduli space
\[
\{(g,[v]) \in \mathscr{U}\times C^{j,\alpha}(\Sigma,M) \, |\, v \; \mathrm{is \; a}\; g-\mathrm{stationary \; almost \; embedding}\}.
\]
Here, $\mathcal{U}\subset\mathscr{G}_k$ is an open subset of $C^k-$Riemannian metrics on $M$, $k\ge 3$ or $k=\infty$.
For $k=\infty$, $\mathscr{G}_k$ is not a manifold but a Fr\'{e}chet manifold.
In order to use Theorem \ref{thm:Transversality}, it is necessary for $\mathscr{V}$ to be a Banach manifold.
Floer got around this problem by introducing $C^\epsilon$ topology.

To define Floer's $C^\epsilon$ let $E$ be a Euclidean vector bundle over a Riemannian manifold $N$ and assume it is equipped with an orthogonal connection. For $\eta \in \Gamma(E)$ define
\[
\|\eta\|_{C^\epsilon}:=\sum_{k=0}^{\infty} \epsilon_k \| \nabla^k\eta\|_{C^0}
\]
where $(\epsilon_k) \in \mathfrak{s}=(0,\infty)^{\mathbb{N}_0}$. Then according to Theorems $B.2$ and $B.5$ of \cite{Wendl-Lecture},
\[
C^\epsilon\Gamma(E):=\{\eta \in \Gamma(E) | \|\eta\|_{C^\epsilon} < \infty \} 
\]
is a separable Banach space.

There is a preorder $\prec$ on $\mathfrak{s}$ defined as follows
\[
(\nu_k) \prec (\mu_k) \quad \iff \quad \mathrm{lim\,sup}_{k \to \infty} \frac{\nu_k}{\mu_k} < \infty.
\]
The following proposition summarizes the useful properties we need.
\begin{prop}[\cite{Wendl}, Lemma 5.28 and \cite{DoanWalpuski}, Propositions $2.2.2$ and $2.2.4$] \label{prop:2}
	Let $\mathfrak{s}$ and $E$ be as above, then
	\item 1. \label{lowerbound}
	Let $\mathfrak{s}_0 \subset \mathfrak{s}$ be a countable subset. There is an $\epsilon \in \mathfrak{s}$ such that $\epsilon \prec \nu$ for every $\nu \in \mathfrak{s}_0$.
	\item 2.
	For each $\epsilon \in \mathfrak{s}$, $C^\epsilon\Gamma(E) \hookrightarrow \Gamma(E)$ is continuous.
	\item 3.
	For every $\eta \in \Gamma(E)$, $\|\eta\|_{C^\epsilon} < \infty$ provided $\epsilon\in \mathfrak{s}$ decays sufficiently fast. 
\end{prop}    

Let $g_0 \in \mathscr{G}$. $T_{g_0}\mathscr{G}$ is the space of smooth symmetric $2$-tensors on $M$. 
Every $\mathfrak{g} \in T_{g_0}\mathscr{G}$ belongs to $C^\epsilon\Gamma(\mathrm{Sym^2T^*M})$, for some $\epsilon$, which is by the preceding discussion a Banach space.
Moreover, for a sufficiently small $\delta>0$ there is a diffeomorphism
\[
\mathrm{exp}_{g_0}:\{\mathfrak{g} \in T_{g_0}\mathscr{G} | \, \|\mathfrak{g}\|_{C\epsilon}  <\delta \} \rightarrow \mathscr{G}.
\]
For $\epsilon \in \mathfrak{s}$ this gives rise to a smooth Banach manifold
\[
\mathcal{U}(g_0,\epsilon):=\{\mathrm{exp}_{g_0}(\mathfrak{g}) | \, \|\mathfrak{g}\|_{C^0}  <\delta\quad\mathrm{and}\quad \|\mathfrak{g}\|_{C^\epsilon}  <\infty \}
\]
which is separable and metrizable and moreover embeds continuously into $\mathscr{G}$.

For a subset of a Banach manifold, there is a notion of codimension which is defined as follows:
\begin{defi}[Definition $1.B.1$ of \cite{DoanWalpuski}]
	A subset $S$ of a Banach manifold $X$ is said to have codimension at least $k\in\mathbb{N}_0$ if there exists a $C^1$ Banach manifold $Z$ and a $C^1$ Fredholm map $\zeta:Z\rightarrow X$
	  such that $X\subset \mathrm{Im}\,\xi$ and
	  \[
	  \mathrm{sup}_{z\in Z}\,\mathrm{index}\;\mathrm{d}_z\zeta \le -k
	  \]
	  Then the codimension of $S$ is defined by
	  \[
	  \mathrm{codim}\; S:=\mathrm{sup}\{
	  k\in\mathbb{N}_0\,|\, S \text{ codimension at least } k
	  \}\in \mathbb{N}_0\cup \{\infty\}.
	  \]
\end{defi}
\begin{prop}[Proposition $1.B.2$ of \cite{DoanWalpuski}]\label{CodimOfSet}
	Let $f:X\rightarrow Y$ be a Fredholm map between Banach manifolds. Let $S \subset X$. Then
	\[
	\mathrm{codim}\; f(S) \ge \mathrm{codim}\; S-\mathrm{inf}_{x\in X}\; \mathrm{index}\; \mathrm{d}_x f.
	\]
\end{prop} 

\subsubsection{Proof of The Flexibility Condition}
Let $(M^n,g)$ be a Riemannian manifold and let $u:\Sigma^k \rightarrow M^n$ be a $g$-minmal embedding, where $\Sigma$ is a manifold of dimension $k \le n$.
Let $\nabla$ denote the Levi-Civita connection on $M$. Then the second fundamental form, $h$, of $\Sigma$ is given by
\[
h(V,W):=(\nabla_V W)^N,
\]
where $V,W$ are tangent to $\Sigma$. 
Let $\{e_1,...,e_k\}$ be an orthogonal frame for the tangent bundle of $\Sigma$. Then 
\[
\mathcal{J}_{g,u} \, X=\Delta^\perp X+\big(\sum_{i=1}^{k} R(e_i,X)e_i\big)^\perp+\sum_{i,j} g(h_{ij},X) h_{ij}
\]
where $X\in \Gamma(N_u)$ and $\Delta^\perp$ is the Laplace operator acting on normal vector fields and is defined by the induced normal connection.
Moreover, $R(e_i,X)e_i$ denotes the Riemann tensor and $h_{ij}:=h(e_i,e_j)$, c.f. \cite{Simons}. 
Note that $\mathcal{J}_{g,u}$ is a self-adjoint elliptic operator.

For $u:\Sigma \rightarrow M$ as above,
set $E=F=Nu$.

\begin{lem} \label{lem:1}
	Let $g_0$ be a smooth metric on $M$. Let $\Sigma$ be a $k-$dimensional manifold and $v:\Sigma \rightarrow M$ be an almost $g_0-$minimal embedding. Let $U$ be the open set consisting of injective points of $v$ i.e.
	\[
	U:=\{x \in \Sigma | v^{-1}(v(x))=\{x\} \quad \mathrm{and} \quad \mathrm{d}_xv \ne 0\}.
	\]
    Let
	\[
	A \in \Gamma(\mathrm{Hom}(Nv,Nv))
	\]
	be a section with support in $U$.
	Moreover, at each point, $A$ is a symmetric matrix. 
	Then there is a path of metrics $(g_t)_{t\in (-T,T)}$ in $\mathcal{U}(g_0,\epsilon)$, for some $T>0$, such that $v$ is an almost $g_t-$minimal embedding and
	\[
	\frac{\mathrm{d}}{\mathrm{dt}}|_{t=0} \mathcal{J}_{g_t,v}X=AX
	\]
	for every $X \in \Gamma(Nv)$.
\end{lem}
Note that because $\mathcal{J}$ is a self-adjoint operator, we can take $A$ to be symmetric.

\renewcommand*{\proofname}{Proof}
\begin{proof}
	Let $x_0 \in \Sigma$. One can choose a local coordinates around $p_0=v(x_0)$ in such a way that $v(\Sigma)$ is given by $(x_1,...,x_k,0,...,0)$ and $g_0|_{v(\Sigma)}$ is the identity matrix in this local coordinate.
	
	For a given $A$ as in the Lemma, let $f(x)=\sum_{i,j > k}^{n} a_{ij}(x_1,...,x_k) x_i x_j$ in which $a_{ij}$ are entries of $A$. 
	Note that $\partial_l\partial_tf=a_{lt}$, for $l,t>k$
	
	Let $\mathfrak{h}_f \in T_{g_0} \mathscr{G}$ be an element of the form $\mathfrak{h}_f=(\mathfrak{h}_{ij})$ where $\mathfrak{h}_{11}=-2f$, $\mathfrak{h}_{ij}=0$ for $i,j>1$. Note that for all $i,j$, $\mathfrak{h}_{ij}$ and its first derivative are zero in a neighborhood of	$p_0$ along $v(\Sigma)$. In particular,  $\mathfrak{h}_f|_{v(\Sigma)}=0$.
	
	For $0<T \ll 1$ define $(g_t)_{t \in (-T,T)}$ by
	\[
	g_t=\mathrm{exp}_{g_0}(t\mathfrak{h}_f)
	\]
	Therefore, for all $t\in(-T,T)$ and all $x \in \Sigma$, $g_t(v(x))=g_0(v(x))$. 
	Thus, if $t \in (-T,T)$,  $v$ is an almost $g_t-$minimal embedding.
	
	Let $w_{ij}=\frac{\partial}{\partial t}((g_t)_{ij})$ and $w=(w_{ij})$. 
	Note that 
	$$\frac{\partial}{\partial t} \Gamma_{ij}^k=\frac{1}{2} g^{kl}(\nabla_i w_{jl}+\nabla_j w_{il}-\nabla_l w_{ij}) $$
	$$\frac{\partial}{\partial t} R_{ijil}=\frac{1}{2} (\nabla_j \nabla_i w_{il}+\nabla_i \nabla_l w_{ji}-\nabla_i \nabla_i w_{jl}-\nabla_j \nabla_l w_{ii})+\frac{1}{2} R_{iji}^p w_{pl}-\frac{1}{2} R_{ijl}^p w_{pi},$$
	$$\frac{\partial}{\partial t} \Delta X^j=-<w,\mathrm{Hess} X^j>_g-g(\mathrm{div} \, w-\frac{1}{2} d(\mathrm{tr_g}\, w),d X^j), $$
	where, $\Gamma_{ij}^k$ denotes the Christoffel symbol and $R_{ijkl}=\langle R(e_i,e_j)e_k,e_l \rangle$.
	
	Since  in a neighborhood of	$p_0$ along $v(\Sigma)$,  $w_{ij}$ and its first derivative are zero, it is not difficult to see that the $l-$th component of $\frac{\mathrm{d}}{\mathrm{dt}}|_{t=0} \mathcal{J}_{g_t,v}X$ 
	is $\sum_{j > k}^{n} X^j \partial_j \partial_l f$.
	$\partial_l\partial_tf=a_{lt}$ and $A$ is symmetric, therefore, 
	\begin{align*}
	\big(\frac{\mathrm{d}}{\mathrm{dt}}|_{t=0} \mathcal{J}_{g_t,v}X\big)_l
	&=\sum_{j}a_{lj}X^j \\
	&=(AX)_l
	\end{align*}
	Provided $\epsilon\in\mathfrak{s}$ decays sufficiently fast, $(g_t)_{t \in (-T,T)}$ is in $\mathcal{U}(g_0,\epsilon)$, c.f. Proposition \ref{prop:2} and Theorem $2.4.2$ of \cite{DoanWalpuski}.
	This completes the proof.
\end{proof}
	\section{Failure Of Super-Rigidity}
\subsection{Index Of Twisted Operator For Orbifold Riemann Surfaces} \label{Index Operator}
Let $v:(\Sigma,h)\rightarrow (M,g)$ denote an embedding of a Riemann surface into a Riemannian manifold $(M,g)$, and let $\pi:(\tilde{\Sigma},\tilde{j}) \rightarrow (\Sigma,j)$ be a non-constant holomorphic map, where $j$ represents the unique complex structure on $\Sigma$ induced by the metric $h$. Let $N_v$ denote the normal vector bundle over $\Sigma$.
There is a one-to-one correspondence between $\Gamma(\pi^*N)$ and $\Gamma(N \otimes W)$, where $W$ is an appropriate Euclidean local system.( See Appendix \ref{AppendixA} for the definition of a Euclidean local system and the aforementioned correspondence) 

Let $\varpi:\Sigma \rightarrow \mathbb{N}$ be a function such that 
\[
Z_\varpi:=\{x \in \Sigma | \varpi(x) >1\}
\]
is a discrete set. $\varpi$ with this property is called a \textbf{multiplicity function}. One can associate an orbifold Riemann surface, $\Sigma_\varpi$, to $\varpi$, which has $\Sigma$ as its underlying topological space as follows: 
Let  $x \in \Sigma$ and $\phi:D\rightarrow\Sigma$ be a holomorphic chart with $\phi(0)=x$, where $D$ is the unit disk in $\mathbb{C}$. 
An orbifold chart around $x \in \Sigma_\varpi$ is the triple $(D,\mu_k,\phi_k)$, where $k=\varpi(x)$, 
$$\mu_k:=\{\xi \in \mathbb{C} | \xi^k=1 \}$$ 
and $\phi_k: D \rightarrow \Sigma$ is the $\mu_k-$invariant continuous map defined by 
\[
\phi_k(z):=\phi(z^k).
\]
Note that $D/ \mu_k \cong \phi_k(D)$.

There is a holomorphic orbifold map $\beta_\varpi:\Sigma_\varpi \rightarrow \Sigma$ which is the identity map on the underlying topological space. With respect to these charts, this map locally is given by $\phi^{-1} \circ \beta_v \circ \phi_k(z)=z^k$. (See Appendix \ref{Jacobi Operator} for the definition of an orbifold and its related concepts required for this paper, as well as the properties of the twisted Jacobi operator for orbifolds.)

We would like to know how the index of the Jacobi operator over an orbifold Riemann surface changes as it is twisted with a Euclidean local system. For this purpose we investigate the local behavior.

\begin{prop}\label{IndexOfTwistedOperator}
	Let $(\Sigma,h)$ be a closed Riemann surface and denote by $j$ the unique complex structure on $\Sigma$ induced by $h$.
	Let $\varpi$ be a multiplicity function on $\Sigma$ and the map $\beta_\varpi:(\Sigma_\varpi,j_\varpi)\rightarrow(\Sigma,j)$ be as above.
	Let $v:(\Sigma.h)\rightarrow(M,g)$ be a minimal embedding. 
	For a Euclidean local system $\underline{V}$ on $\Sigma_\varpi$ and $x\in Z_\varpi$ let $\rho_x$ be the monodromy of $\underline{V}$ around $x$ and $V^{\rho_x}\subset V$ denote the subspace of $\rho_x-$invariant vectors. 
	Then
	\begin{align*}
		\mathrm{index}(\mathcal{J}_{g,v;\varpi}^{\underline{V}})&=-\mathrm{rk}_\mathbb{C}\,\textbf{\textit{N}}_{v,\varpi}\sum_{x\in Z_\varpi}\mathrm{dim}(V/V^{\rho_x})
	\end{align*}
	Where $\textbf{\textit{N}}_{v,\varpi}=N_{v,\varpi}\otimes \mathbb{C}$ is the complexification of $N_{v,\varpi}=\beta_\varpi^* N_v$ and $\mathcal{J}_{g,v;\varpi}:=\beta_\varpi^*\mathcal{J}_{g,v}$. 
	Moreover, we can assume that for each $x\in Z_\varpi$, $\mathrm{dim}(V/V^{\rho_x})\ge1$.
\end{prop}

In particular, if $\mathrm{dim}\,M=n$, then by Proposition \ref{IndexOfTwistedOperator},
\[
\mathrm{index}(\mathcal{J}_{g,v;\varpi}^{\underline{V}}) \le -(n-2).
\]


Let $\underline{V}$ be a Euclidean local system where its fiber, $V$, is a vector space over $\mathbb{C}$. 
The subsequent argument closely follows the Section 2.B of \cite{DoanWalpuski} with slight adjustments to ensure compatibility with the Jacobi operator.

Let $\rho:\mu_k \rightarrow \mathrm{GL}(V)$ be a representation. This induces an action of $\mu_k$ on $D \times V$, defined as follows:
\[
\zeta.(z,v):=(\zeta z,\rho(\zeta)v), \quad \text{where}\quad
\zeta\in\mu_k,\quad (z,v)\in D\times V
\]
Let $V_\rho$ denote the quotient of $D \times V$ by this action.
By Examples \ref{ex:Orbifold_example_1} and \ref{ex:Orbifold_example_4}, this quotient is an orbifold and the projection on the first component gives an orbifold bundle 
$$V_\rho \rightarrow D/\mu_k.$$
There is also a trivial bundle over over $D/\mu_k$ with fiber $V$. 

Every holomorphic vector bundle over a non-compact Riemann surface is holomorphically trivial [Section $3.30$ of \cite{Forster}]. Thus, on $\dot{D}:=D\setminus \{0\}$, there is a trivialization $\alpha:V_\rho|_{\dot{D}} \rightarrow \dot{D} \times V$.
This trivialization is characterized by using the aforementioned representation as follows:
Since $\rho$ is a complex representation, it can be diagonalized by choosing an isomorphism $V\cong \mathbb{C}^r$: 
\[
\rho(z)=\mathrm{diag}(z^{w_1},...,z^{w_r}),
\]
where $w_i \in \mathbb{Z}/k\mathbb{Z}$, $i=1,...,r$.
By lifting $w_i$ to $\hat{w}_i \in \mathbb{Z}$ one obtains a representation $\hat{\rho}:\mathbb{C}^* \rightarrow \mathrm{GL}(V)$ which extends $\rho$.
Then, define 
\begin{equation}\label{eq:Trivialization}
	\begin{split}
	&\alpha: V_\rho|_{\dot{D}} \rightarrow \dot{D} \times V \\
	& \alpha([z,v]):=[z,\hat{\rho}(z^{-1})v].
	\end{split}
\end{equation}

Let $\mathcal{V}_\rho$ and $\mathcal{V}$ denote the corresponding sheaves of holomorphic orbifold sections of these bundles, respectively.
There is an exact sequence 
\[
0 \rightarrow \mathcal{V}_{\rho} \rightarrow \mathcal{V} \rightarrow V/V^\rho \otimes \mathcal{O}_0 \rightarrow 0
\]
Here $\mathcal{O}_0$ is the structure sheaf of the point $[0]$ and $V^\rho \subset V$ denotes the $\rho-$invariant subspace under the action of $\rho$ on $V$. 
The injection $\mathcal{V}_{\rho} \hookrightarrow \mathcal{V}$ in the above short exact sequence is induced by $\alpha$ and the finial map is the evaluation map at $0$. 

This construction can be generalized to a holomorphic vector bundle $E$ on $\Sigma$.
Let $\varpi$ be a multiplicity function on $\Sigma$ and consider the map $\beta_\varpi:(\Sigma_\varpi,j_\varpi)\rightarrow(\Sigma,j)$.
Set $E_\varpi=\beta_\varpi^*E$. One can modifies this bundle over $Z_\varpi$ 
via a collection of representations $\rho=(\rho_x:\mu_{\varpi(x)} \rightarrow \mathrm{GL}(E_x))_{x\in Z_\varpi}$ of the fibers of $E$ at each of these points. 
This process is called a \textbf{Hecke modification} of $E$ of type $\rho$. The modified bundle is denoted by $E_\rho$. More precisely, outside $Z_\varpi$ there is an isomorphism 
\[
\alpha: E_{\varpi,\rho}|_{\Sigma_\varpi \setminus Z_\varpi} \rightarrow E|_{\Sigma_\varpi \setminus Z_\varpi}
\]
which locally is given by \ref{eq:Trivialization} for every $x\in Z_\varpi$.
In this situation the following Riemann-Roch formula for orbifolds holds:

\begin{thm}[Theorem 2.B.6 of \cite{DoanWalpuski}]\label{Riemann-Roch} 
	$
	\chi(E_\varpi,\rho)=\chi(E_\varpi)-\sum_{x\in Z_\varpi}\mathrm{dim}_{\mathbb{C}}(E_x/E_x^{\rho_x}).
	$
\end{thm}

\renewcommand*{\proofname}{Proof of Proposition \ref{IndexOfTwistedOperator}}
\begin{proof}
	Let $V$ be a real vector field. We want to compute $\mathrm{ind}\, \mathcal{J}_{g,v;\varpi}^{\underline{V}}$. 
	To that end consider the monodromy representation $\rho_x: \mu_{\varpi(x)} \rightarrow \mathrm{GL}(V)$ for each $x \in Z_v$ and complexify $V$ and $\underline{V}$, $V^{\mathbb{C}}:=V\otimes \mathbb{C}$ and $\underline{V}^\mathbb{C}:=\underline{V} \times \underline{\mathbb{C}}$. 
	This gives a complexification of the representation: $\rho_x^\mathbb{C}: \mu_{v(x)} \rightarrow \mathrm{GL}_\mathbb{C}(V^\mathbb{C})$.
	By the previous discussion there is a holomorphic vector bundle $\mathcal{V}$ over $\Sigma$ such that the Hecke modification of type $\rho^\mathbb{C}$ of $\mathcal{V}_\varpi:=\beta_\varpi^* \mathcal{V}$ is isomorphic to $\underline{V}^\mathbb{C}$ i.e., $\underline{V}^\mathbb{C} \cong \mathcal{V}_{\varpi,\rho^\mathbb{C}}$, c.f. proof of Proposition $2.8.6$ of \cite{DoanWalpuski}.
	
	Let $N_{v,\varpi}=\beta_\varpi^* N_v$. 
	The elliptic operator $\mathcal{J}_{g,v;\varpi}^{\underline{V}}$ gives the following complex
	\[
	0\rightarrow\Gamma(\Sigma_\varpi,N_{v,\varpi}\otimes\underline{V})
	\xrightarrow{\mathcal{J}_{g,v;\varpi}^{\underline{V}}}
	\Gamma(\Sigma_\varpi,N_{v,\varpi}\otimes\underline{V})\rightarrow 0.
	\]
	Denote by $\textbf{\textit{N}}_{v,\varpi}$ the complexification of $N_{v,\varpi}$, i,e. $\textbf{\textit{N}}_{v,\varpi}=N_{v,\varpi}\otimes \mathbb{C}$. 
	Index of $\mathcal{J}_{g,v;\varpi}^{\underline{V}}$ is two times the Euler characteristics of the complex $\textbf{\textit{N}}_{v,\varpi}\otimes_\mathbb{C}\underline{V}^\mathbb{C}$. 
	Therefore, by Theorem \ref{Riemann-Roch} for $\underline{V}^\mathbb{C} \cong \mathcal{V}_{\varpi,\rho^\mathbb{C}}$,
	\begin{align*}
		\mathrm{index}(\mathcal{J}_{g,v;\varpi}^{\underline{V}})
		&:=\chi(N_{v,\varpi}\otimes\underline{V}\otimes\mathbb{C})\\
		&=\chi(N_{v,\varpi}\mathbb{C}\otimes_\mathbb{C}\mathbb{C}\otimes\underline{V})\\
		&=\chi(\textbf{\textit{N}}_{v,\varpi}\otimes_{\mathbb{C}}\underline{V}^\mathbb{C})\\
		&=\chi(\textbf{\textit{N}}_{v,\varpi}\otimes_{\mathbb{C}} \mathcal{V}_{\varpi,\rho^\mathbb{C}})\\
		&=\chi(\textbf{\textit{N}}_{v,\varpi}\otimes_{\mathbb{C}} \mathcal{V}_{\varpi})-\mathrm{rk}_\mathbb{C}\,\textbf{\textit{N}}_{v,\varpi}
		\sum_{x\in Z_\varpi}\mathrm{dim}(V/V^{\rho_x})\\
	\end{align*}
	Let $X$ be a (complex) curve and $\mathcal{F}$ be a locally free sheaf $\mathcal{O}_X-$module of rank $m$. In this situation a notion of degree for $\mathcal{F}$ is defined as follows:
	\[
	\mathrm{deg}(\mathcal{F}):=\chi(X,\mathcal{F})-m\chi(X,\mathcal{O}_X)
	\]
	If $\mathcal{E}$ be a coherent $\mathcal{O}_X-$module, then it is known that
	\[
	\chi(X,\mathcal{F}\otimes\mathcal{E})=r\,\mathrm{deg}(\mathcal{F})+m\,\chi(X,\mathcal{E})
	\]
	where $r=\mathrm{dim}_{\mathbb{C}(x)}\mathcal{E}_x$.
	For the current situation $X=\Sigma_\varpi$, $\mathcal{F}=\mathcal{V}_\varpi$ of rank $\mathrm{dim}\,V$ and $\mathcal{E}=\textbf{\textit{N}}_{v,\varpi}$ with $r=\mathrm{rk}_\mathbb{C}\,\textbf{\textit{N}}_{v,\varpi}$. It follows that
	\begin{align*}
		\mathrm{index}(\mathcal{J}_{g,v;\varpi}^{\underline{V}})&=\mathrm{dim}\,V\,\chi(\textbf{\textit{N}}_{v,\varpi})+\mathrm{rk}_\mathbb{C}\,\textbf{\textit{N}}_{v,\varpi}\Big(\mathrm{deg}(\mathcal{V}_\varpi)-\sum_{x\in Z_\varpi}\mathrm{dim}(V/V^{\rho_x})\Big)\\
		&=\mathrm{dim}\,V\,\mathrm{index}(\mathcal{J}_{g,v;\varpi})+\mathrm{rk}_\mathbb{C}\,\textbf{\textit{N}}_{v,\varpi}\Big(\mathrm{deg}(\mathcal{V}_\varpi)-\sum_{x\in Z_\varpi}\mathrm{dim}(V/V^{\rho_x})\Big)
	\end{align*}
	By the proof of Proposition $2.8.6$ of \cite{DoanWalpuski}
	\[
	\mathrm{deg}(\mathcal{V}_\varpi)=\frac{1}{2}\sum_{x\in Z_\varpi}\mathrm{dim}(V/V^{\rho_x}).
	\]
	By Proposition \ref{prop:IndexOrbifoldJacobiOperator}, $\mathrm{index}(\mathcal{J}_{g,v;\varpi})=\mathrm{index}(\mathcal{J}_{g,v})$. Therefore,
	\begin{align*}
		\mathrm{index}(\mathcal{J}_{g,v;\varpi}^{\underline{V}})&=\mathrm{dim}\,V\,\mathrm{index}(\mathcal{J}_{g,v;\varpi})-\frac{1}{2}\mathrm{rk}_\mathbb{C}\,\textbf{\textit{N}}_{v,\varpi}\sum_{x\in Z_\varpi}\mathrm{dim}(V/V^{\rho_x})
	\end{align*}
	As $\mathcal{J}_{g,v}$ is an elliptic operator of index zero, therefore,
	\begin{align*}
		\mathrm{index}(\mathcal{J}_{g,v;\varpi}^{\underline{V}})&=-\frac{1}{2}\mathrm{rk}_\mathbb{C}\,\textbf{\textit{N}}_{v,\varpi}\sum_{x\in Z_\varpi}\mathrm{dim}(V/V^{\rho_x}).
	\end{align*}
	
	According to an argument of [\cite{DoanWalpuski}, proof of Proposition 2.8.5], we can always modify the multiplicity function $\varpi$, if necessary, so that for each $x\in Z_\varpi$, $\mathrm{dim}_{\mathbb{C}}(V/V^{\rho_x})\ge 1$: 
	Let $\underline{V}$ be an irreducible local system on $\Sigma_\varpi$. If $\underline{V}$ has trivial monodromy around some points of $Z_\varpi$, $\varpi$ is modified as follows
	\begin{align*}
		\varpi'(x):=
		\begin{cases}
			\varpi(x), \quad \text{if $\underline{V}$ has trivial monodromy around $x$}, \\
			1 \quad \text{otherwise.}
		\end{cases}
	\end{align*}
	Let $\underline{V}'$ denote the induced irreducible local system by $\underline{V}$ on $\Sigma_{\varpi'}$. Then $\mathrm{ker}\,\mathcal{J}_{g,v;\varpi}^{\underline{V}}\cong \mathrm{ker}\,\mathcal{J}_{g,v;\varpi'}^{\underline{V'}}$ and the monodromy of $\underline{V}'$ around each point of $Z_{\varpi'}$ is non-trivial.
\end{proof}

\subsection{Loci of Failure of Super-Rigidity}\label{LociOfSuper-RigidityBreakdown}
Let $x_0\in\Sigma$. Let $B=\{\underline{V}_i\}_{i=1}^k$, where each $\underline{V}_i$ is an irreducible Euclidean local system such that its monodromy representation factors through a finite quotient of  $\pi_1(\Sigma,x_0)$. 
For $i=1,...,k$, let $K_i=\mathrm{End}(\underline{V}_i)$ be the algebra of parallel endomorphisms of $\underline{V}_i$ and set $k_i:=\mathrm{dim}_\mathbb{R}K_i$.

\begin{defi} [Definition 2.9.2 of \cite{DoanWalpuski}]\label{def:FailureOfSuperRigidity}
	Let $\mathscr{V}\subset\tilde{\mathscr{Y}}$ and $s\in \mathbb{N}_0$. Let
	\begin{align*}
	 \mathcal{M}_s(\mathscr{V}):=
	 \{
	  (g,[v;\varpi])\,|\, & (g,[v])\in\mathscr{V}\quad\mathrm{and}\quad 
	  \varpi:\Sigma\rightarrow\mathbb{N}
	   \text{ is }\\ & \text{a multiplicity function with } \sharp Z_\varpi=s
	 \}
	\end{align*}
	
	Define $\mathcal{W}_s(\mathscr{V})$ as the subset of elements $(g,[v;\varpi])\in \mathcal{M}_s(\mathscr{V})$ for which there is an irreducible Euclidean local system $\underline{V}$ on $\Sigma_\varpi$ with the following properties:
	\begin{enumerate}
		\item the monodromy representation of $\underline{V}$ factors through a finite quotient of $\pi_1(\Sigma,x_0)$.
		\item the monodromy of $\underline{V}$ around each $x\in Z_\varpi$ is non-trivial.
		\item $\mathrm{ker}(\mathcal{J}_{g,v;\varpi}^{\underline{V}})\ne 0$.
	\end{enumerate}

	Define $\mathcal{W}_s^{\mathrm{top}}(\mathscr{V})$ as the subset of elements $(g,[v;\varpi])\in \mathcal{W}_s(\mathscr{V})$ 
	for which there is an irreducible local system $\mathscr{V}$ with properties $(1)$ and $(2)$ such that 
	\begin{enumerate}
    	  \setcounter{enumi}{3}
		\item $\mathrm{ker}(\mathcal{J}_{g,v;\varpi}^{\underline{V}})=1$
		\item If $V'$ is any other irreducible local system which is not isomorphic to $V$ and satisfies properties $(1)$ and $(2)$, then
		$
		\mathrm{ker}(\mathcal{J}_{g,v;\varpi}^{\underline{V}'})=0.
		$
		\item If $\mathrm{dim}\,M=6$, then for every $x\in Z_\varpi$, $\mathrm{dim}(V/V^{\rho_x})=1$. If $\mathrm{dim}(M)>6$, then $Z_\varpi=\emptyset$.
	\end{enumerate}
	
\end{defi}

Let $s\in\mathbb{N}_0$ and $\epsilon\in\mathfrak{s}$ be a sequence which converges to zero.
Let $(g_0,[v_0])\in\mathscr{Y}$ and $\mathscr{V}$ denote an open neighborhood of this element which maps to $\mathcal{U}(g_0,\epsilon)$ under the natural projection
\begin{align*}
	\tilde{\mathscr{Y}}\longrightarrow\mathscr{G} \\
	(g,[v])\mapsto g.
\end{align*}

\begin{prop}[Proposition 2.9.4 of \cite{DoanWalpuski}]\label{FailureOfSuperRigidity}
	Let $M$ be a Riemannian Manifold with $\mathrm{dim}\,M\ge 6$. Let $s\in\mathbb{N}_0$. There is a closed subset of $\mathcal{M}_s(\mathscr{V})$ of infinite codimension which is denoted by $\mathcal{W}_s^\infty(\mathscr{V})$ such that
	\begin{enumerate}
		\item $\mathcal{W}_s^{top}(\mathscr{V})\setminus\mathcal{W}_s^\infty(\mathscr{V})$ is contained in a submanifold of codeimension $2s+1$.
		\item 
		$\mathcal{W}_s(\mathscr{V})\setminus\big(\mathcal{W}_s^\infty(\mathscr{V})\cup \mathcal{W}_s^{top}(\mathscr{V})\big)$ has codeimension at least $2s+2$.
	\end{enumerate}
\end{prop}
\renewcommand*{\proofname}{Proof}
\begin{proof}
	The proof is the same as that of Proposition $2.9.4$ of \cite{DoanWalpuski} except here the index formula of Proposition \ref{IndexOfTwistedOperator} is used
	and in the proof, $n-1$ substituted with $n-2$. 
	Here, the proof of Proposition $2.9.4$ of \cite{DoanWalpuski} is rewritten, in order the paper be self-contained.


By varying $v$, the various normal bundles $N_v$ define a vector bundle $\mathfrak{N}$ over $\mathscr{V}\times\Sigma$. 
By shrinking $\mathscr{V}$, if necessary, one can be ensured that $\mathfrak{N}\cong \mathrm{pr}_2^*N_{v_0}$ via the pullback through the following projection map
\begin{align*}
	\mathrm{pr}_2:\mathscr{V}\times\Sigma\rightarrow\Sigma.
\end{align*}
Here, $v_0:\Sigma\rightarrow(M,g)$ is an almost embedding.
This gives the following family of Jacobi operators
\begin{align*}
	\mathcal{J}:\mathscr{V}\rightarrow\mathcal{F}(\Gamma(N_{v_0}),\Gamma(N_{v_0}))
\end{align*}
Let $U$ denote the open subset of injective points of $v_0$.
By Lemma \ref{lem:1}, for sufficiently small $\epsilon$, with respect to the preorder $\prec$, $\mathcal{J}$ is $B-$equivariantly flexible in $U$.
Moreover, by Theorem \ref{thm:Wendl'sTheorem} and Proposition \ref{prop:Wendl'sConditionForJacobi}, the subset of elements $(g,[v])\in\mathscr{V}$ for which $\mathcal{J}_{g,v}$ does not satisfy the Petri condition in $U$ is of infinite codimension.

Let $\underline{V}_1$ and $\underline{V}_2$ denote two irreducible Euclidean local system on $\Sigma$ such that their monodromy representations factor through a finite quotient of $\pi_1(\Sigma,x_0)$.
Let $B=\{\underline{V}_1,\underline{V}_2\}$, as in Theorem \ref{thm:B-Transversality}.
Let 
\begin{align*}
	\mathfrak{F}_{s,B}(\mathscr{V}):=\{
	\mathfrak{v}:=(g,[v])\in \mathscr{V}\,|\, \Theta_\mathfrak{v}^B 
	\text{ is not surjective and } \sharp Z_\varpi=s
	\}.
\end{align*}
where, $\Theta_\mathfrak{v}^B$ has been introduced in Theorem \ref{thm:B-Transversality}.
Then $\mathfrak{F}_{s,B}(\mathscr{V})$ is a closed subset and by the previous paragraph has infinite codimension.
The subset $\mathcal{W}_s^\infty(\mathscr{V})$ is the union of $\mathfrak{F}_{s,B}(\mathscr{V})$ for all possible choices of $B$.
Let $\textbf{d}=\{d_1,d_2\}\in\mathbb{N}_0^2$, and set
\begin{align*}
	\mathfrak{F}_{s,B}^\textbf{d}(\mathscr{V}):=\{
	(g,[v])\in\mathscr{V}\setminus\mathfrak{F}_{s,B}(\mathscr{V})\,|\, \mathrm{dim}_{K_i}\,\mathrm{Ker}\,\mathcal{J}_{g,v;\varpi}^{\underline{V}_i}=d_i, i=1,2
	\}
\end{align*}
Theorem \ref{thm:B-Transversality} implies that $\mathfrak{F}_{s,B}^\textbf{d}(\mathscr{V})$ is a submanifold of codimension
\[
\sum_{i=1}^2k_i d_i(d_i-\mathrm{index}_{K_i}\,\mathcal{J}_{g,v;\varpi}^{\underline{V}_i})
\]
By Proposition \ref{IndexOfTwistedOperator},
\begin{align*}
	\mathrm{index}_{K_i}\,\mathcal{J}_{g,[v];\varpi}^{\underline{V}_i}&=-\frac{n-2}{2}\sum_{x\in Z_\varpi}\mathrm{dim}_{K_i}(V/V^{\rho_x})
\end{align*}
where $n=\mathrm{dim}\,M$. Assume that the monodromy of each $\underline{V}_i$ around every $x\in Z_\varpi$ is non-trivial. Then
$\mathrm{index}_{K_i}\,\mathcal{J}_{g,v;\varpi}^{\underline{V}_i}\le -\frac{n-2}{2}$.
Then, provided $d_i\ge 1$,
\[
\mathrm{codim}\,\mathfrak{F}_{s,B}^\textbf{d}(\mathscr{V}) \ge \frac{n-2}{2}s+1\ge 2s+1.
\]
The union of countably many subsets of the form $\mathfrak{F}_{s,B}^\textbf{d}(\mathscr{V})$ for which at least one $d_i\ge 1$, $i=1,2$ construct $\mathcal{W}_s(\mathscr{V})\setminus\mathcal{W}_s^\infty(\mathscr{V})$.
Therefore, the codimension of $\mathcal{W}_s(\mathscr{V})\setminus\mathcal{W}_s^\infty(\mathscr{V})$ is at least $2s+1$.

$\mathcal{W}_s^{top}(\mathscr{V})\setminus\mathcal{W}_s^\infty(\mathscr{V})$ consists of the union of subsets of the form $\mathfrak{F}_{s,B}^\textbf{d}(\mathscr{V})$ for which there is an $i=1,2$ such that
\begin{itemize}
	\item $d_i=1, K_i=\mathbb{R}$ and $d_j= 0$, for $j\ne i$, and
	\item if $n=6$, then for every $x\in Z_\varpi$, $\mathrm{dim}(V/V^{\rho_x})=1$, otherwise $Z_\varpi=\emptyset$.
\end{itemize}
Pursuing the above inequalities, reveals that these conditions are satisfied if and only if 
$$\mathrm{codim}\,\mathfrak{F}_{s,B}^\textbf{d}(\mathscr{V})=2s+1.$$
This proves part $(1)$ of the proposition.

Since the union of subsets of the form $\mathfrak{F}_{s,B}^\textbf{d}(\mathscr{V})$ whose codimension is at least $2s+2$ constructs the subset $\mathcal{W}_s(\mathscr{V})\setminus\big(\mathcal{W}_s^\infty(\mathscr{V})\cup \mathcal{W}_s^{top}(\mathscr{V})\big)$, part $(2)$ of the proposition is proved.
\end{proof}

\renewcommand*{\proofname}{Proof of Theorem \ref{Theorem-A}}
\begin{proof}
Note that $\mathcal{M}_s(\mathscr{V})$ is a Banach manifold and the map
\begin{align*}
\sqcap_s:\mathcal{M}_s(\mathscr{V}) &\rightarrow \mathcal{U}(g_0,\epsilon) \\
(g,[v;\varpi]) &\mapsto g
\end{align*}
is Fredholm of index $2s$ (Section $2.9$ of \cite{DoanWalpuski}).
Let $\mathcal{W}$ denote the subset of those elements $g\in\mathscr{G}$ which are not super-rigid. 
The intersection of $\mathcal{W}$ with $\mathcal{U}(g_0,\epsilon)$ is contained in 
$\bigcup_{s\in\mathbb{N}_0} \sqcap_s(\mathcal{W}_s(\mathscr{V}))$. 
Proposition \ref{FailureOfSuperRigidity}, Proposition \ref{CodimOfSet} and the fact that $\sqcap_s$ is of index $2s$ imply that $\mathcal{W}$ has codimension at least one.
This proves the Theorem.
\end{proof}

Therefore, for a generic path of metrics, to avoid minimal embedded surfaces which have a "branched" cover that fails to be rigid, it is necessary to assume $\mathrm{dim}\,M>6$. 

Let $\pmb{\mathscr{G}}$ denote the space of $C^\infty-$paths $\bar{\Gamma}:[0,1]\rightarrow\mathscr{G}_\infty$. From now on set $\mathscr{G}=\mathscr{G}_\infty$. 
For a path of metrics, like $\bar{\Gamma}$, one can form the fiber product of $[0,1]$ and $\mathscr{Y}$ along $\mathscr{G}$
\begin{align*}
	\mathscr{Y}(\bar{\Gamma})&:=[0,1]\times_{\mathscr{G}}\mathscr{Y}  \\
	&=\{(t,g,[u])\,|\,\bar{\Gamma}(t)=\Pi(g,[u])\}
\end{align*}
which is a $1-$parameter space of minimal embedded surfaces.
Denote $(t,g,[u])$ by $(g_t,[u])$.

\begin{defi}[Definition $2.10.1$ of \cite{DoanWalpuski}]
	Define $\pmb{\mathscr{G}}^\bullet$ as the subset of all $\bar{\Gamma}\in \pmb{\mathscr{G}}$ with the following properties:
	\begin{enumerate}
		\item $\mathscr{Y}(\bar{\Gamma})$ is a $1-$dimensional manifold with boundary.
		\item The subset of elements $t\in[0,1]$ for which $\bar{\Gamma}(t)$ is not super-rigid, is countable and belongs to
		\[
		\{t\in[0,1]\,|\, g_t\in \bigcup_{s\in\mathbb{N}_0} \sqcap_s(\mathcal{W}_s^{top}(\mathscr{V}))\}.
		\]
	\end{enumerate}
\end{defi}

\begin{prop}\label{PathOfSuper-rigidMetrics}(Theorem $2.10.2$ of \cite{DoanWalpuski} and also see Section $2.4$ of \cite{Wendl}).
	Let $M$ be a smooth manifold. Then
	$\pmb{\mathscr{G}}^\bullet$ is a comeager subset of $\pmb{\mathscr{G}}$.
\end{prop}


	\section{Counting Function For Minimal Tori}

\subsection{Counting Minimal Tori In Riemanian Manifolds}\label{Subsection:CountingMinimalToriInManifold}
We want to count minimal tori in a Riemannian manifold $(M,g)$, where $\mathrm{dim}\,M\ge 7$. 
Let $T\cong S^1\times S^1$ and $h$ be a smooth metric on it. Let $u:(T,h)\rightarrow (M,g)$ denote a $g-$minimal immersed torus. 
By Appendix \ref{Branched Imm.}, $u=v\circ\pi$ where $\pi:(T,j)\rightarrow (T',j')$ is a holomorphic covering map and $v:(T',h')\rightarrow (M,g)$ ia a $g-$minimal almost embedded map. 
Here, $T'\cong S^1\times S^1$ and $j$ and $j'$ are the unique complex structures induced by $h$ and $h'$, respectively. 

By Proposition $2.6.3$ of \cite{DoanWalpuski}, the Jacobi operator
\[
\mathcal{J}_{g,u}:\Gamma(N_u)\rightarrow\Gamma(N_u)
\]
equals to 
\[
\pi^*\mathcal{J}_{g,v}:\Gamma(\pi^*N_v)\rightarrow\Gamma(\pi^*N_v)
\] 
which in turn is equivalent to $\mathcal{J}_{g,v}$ twisted by a local system $\underline{V}$, i.e. 
\[
\mathcal{J}_{g,v}^{\underline{V}}:\Gamma(N_v\otimes\underline{V})\rightarrow\Gamma(N_v\otimes\underline{V})
\]
where $\underline{V}=\pi_*\underline{\mathbb{R}}$ (See Appendix \ref{AppendixA}). 
Note that according to Definition $2.6.2$ of \cite{DoanWalpuski} $\pi^*\mathcal{J}_{g,v}$ is characterized by $(\pi^*\mathcal{J}_{g,v})(\pi^*s)=\pi^*(\mathcal{J}_{g,v}\,s)$.

Let $\bar{\Gamma}:[0,1]\rightarrow\mathscr{G}$ denote a path of metrics in $\pmb{\mathscr{G}}^\bullet$ and $(g_{t_k},[v_k])\in \mathscr{Y}(\bar{\Gamma})$ be a sequence such that $t_k\rightarrow t$. Here  $v_k$ is an embedding of a torus $T_k$ into $M$. 
Moreover, let $v_k:T_k\rightarrow M$ converge to $v_0:T_0\rightarrow M$ which factors as $v\circ\pi$, where $\pi:T_0\rightarrow T$ is a covering map of degree $d>1$ and $v:T\rightarrow M$ is an embedding. 
Following \cite{Taubes}, $(g_t,[v])$ obtained in this way is called a weak limit of the sequence $\{(g_{t_k},[v_k])\}$.
Since $\mathrm{dim}(M)>6$, by Proposition \ref{FailureOfSuperRigidity}, we can assume that $(g_t,[v])\in\mathcal{W}_0^{\text{top}}$.
Consequently, Remark \ref{rem:ActionOnKernel} implies that $d=2$.

Note that the isomorphism classes of double covers of a torus $T$ are isomorphic to $H^1(T,\mathbb{Z}_2)$. $H^1(T,\mathbb{Z}_2)$ also classifies the real line bundles on $T$.
Let $\epsilon_\iota$ denote the real line bundle corresponding to $\iota\in H^1(T,\mathbb{Z}_2)$ and let $\mathcal{J}_{g_t,v,\iota}$ denote the operator which acts on $N_v\otimes\epsilon_\iota$ and is defined as $\mathcal{J}_{g_t,v}$ twisted by $\epsilon_\iota$. 
For this twisted operator the index $\pmb{\iota}(g_t,v;\iota)$ is defined to be the index of the $g_t-$minimal map $v_0$.

A sign $\mathrm{sign}(\mathrm{det}(\mathcal{J}_{g,v,\iota}))$ (see Section 2 of \cite{Taubes}) is associated to each minimal map:
Choose a minimal map $u_0:T\rightarrow(M,g_0)$ with zero index and for this minimal map define $\mathrm{sign}(\mathrm{det}(\mathcal{J}_{g_0,u_0})):=+1$. Then $\mathrm{sign}(\mathrm{det}(\mathcal{J}_{g,v,\iota}))$ will be $(-1)^{\pmb{\iota}(g_t,v;\iota)}$.

Remark \ref{rem:ActionOnKernel} and Proposition \ref{FailureOfSuperRigidity} 
imply the following:
\begin{prop}[Lemma $5.16$ of \cite{Taubes}]
	Let $\bar{\Gamma}\in\pmb{\mathscr{G}}^\bullet$. Then 
	$(g,[v])\in\mathscr{Y}(\bar{\Gamma})$, $v:T\rightarrow M$, is a weak limit point of a sequence $(g_{t_k}[v_k])\in \mathscr{Y}(\bar{\Gamma})$ if and only if there exists non-zero $\iota\in H^1(T,\mathbb{Z}_2)$ such that the kernel of $\mathcal{J}_{{g},v,\iota}$ is non-trivial.
\end{prop} 

By the proof of Theorem \ref{Theorem-A}  $$\mathscr{G}_0^{top}:=\bigcup_{s\in\mathbb{N}_0}\Pi_s(\mathcal{W}_s^{top}(\mathscr{Y}))$$
is a codimension one subset of $\mathscr{G}$.
A codimension one subspace is called a \textbf{wall}. Walls divide the space of metrics into chambers. Therefore, if $g$ is inside a chamber then the nullity of $v_0:T_0\rightarrow (M,g)$ is zero and as the metric $g$ changes and cross the wall the index changes by $\pm 1$.

In \cite{Taubes}, Taubes introduces a classification of tori, based on which he defines a function for counting pseudo-holomorphic surfaces: 
Within a chamber, $g-$minimal embedded tori are classified into four types according to the number of $\iota\in H^1(T,\mathbb{Z}_2)$ for which 
$\mathrm{sign}(\mathrm{det}(\mathcal{J}_{g,v,\iota}))$ is negative .
More precisely, let $v:T\rightarrow(M,g)$ denote a $g-$minimal embedded torus and
\[
\delta_v:H^1(T,\mathbb{Z}_2)\rightarrow \mathbb{Z}_2
\]
be a map defined by assigning to each $\iota\in H^1(T,\mathbb{Z}_2)$ the sign of $\mathrm{det}(\mathcal{J}_{g,v,\iota})$.
Note that as far as $g$ is inside a chamber $\delta_v$ does not depend on the choice of $g$.

Call $v$ of type $+k$, $k=0,1,2,3$, if the image of the trivial element of $H^1(T,\mathbb{Z}_2)$
under $\delta_v$ is $+1$ and exactly $k$ other elements of $H^1(T,\mathbb{Z}_2)$ be sent to $-1$ by $\delta_v$. 
Note that the trivial element of $H^1(T,\mathbb{Z}_2)$ represents the trivial double cover.
Similarly, we say $T$ is of type $-k$, $k=0,1,2$, if $\delta_v$ sends the trivial element to $-1$ and exactly $k$ other elements of $H^1(T,\mathbb{Z}_2)$ to $-1$.

\begin{defi}\label{def:1}
	For a superrigid metric $g$, let $v:(T,h)\rightarrow (M,g)$ be a $g-$minimal embedded map of a torus. For $d\in \mathbb{N}$ define
	\[
	n_{\pm k}^d(g,v):=
	\begin{cases}
		\epsilon, \quad & \mathrm{if}\; d=1, \\
		\pm k, \quad & \mathrm{if}\; d=2, \\
		\pm [\frac{k}{2}], \quad & \mathrm{if}\; d=4, \\
		0, \quad & \mathrm{otherwise}.
	\end{cases}
	\]
	where, $\epsilon=\pm 1$ is the sign of $\mathrm{det}(\mathcal{J}_{g,v})$
\end{defi}

Let $g$ be a super-rigid metric on $M$ and let $u:(T,h)\rightarrow (M,g)$ denote  a $g-$minimal immersed map which factors as $u=v\circ \pi$, where $v:(T',h')\rightarrow (M,g)$ is a $g-$minimal embedding of type $\pm k$ and $\pi:T\rightarrow T'$ is a degree $d$ covering map. Then the wight associated to $u$ is defined as
$
n(g,u):=n^d_{\pm k}(g,v).
$

Let $\mathcal{U}\subset \mathscr{L}(g)$ be a finite subset. Define
\[
n(g,\mathcal{U}):=\sum_{u\in\mathcal{U}}n(g,u).
\]
\subsection{Local Description Of The Moduli Space}\label{Subsection:LocalDescriptionOfModuliSpace}
Let $(g_0,[v_0])\in\tilde{\mathscr{Y}}$ and let $\mathrm{exp}:N_{v_0}\rightarrow M$ be the exponential map. 
There is a neighborhood $V$ of the zero section of $N_{v_0}$ such that $\mathrm{exp}$ maps $V$ diffeomorphically onto a neighborhood of $v_0(\Sigma)$ 
and if $v\in C^{j,\alpha}(\Sigma,M)$ is sufficiently close to $v_0$ then $[v]=[\mathrm{exp}\circ s]$ for some section $s$ of $V$.
Then, as discussed in the proof of Theorem $2.1$ of \cite{White-1991}, 
\begin{align*}
	\{(g,[\mathrm{exp}\circ s])\;|\; s\in C^{j,\alpha}(\Sigma,N_{v_0}), H(g,s)=0, \|g-g_0\|+\|s\|_{j,\alpha}<\epsilon\}
\end{align*}
is a neighborhood of $(g_0,[v_0])\in \tilde{\mathscr{Y}}$. Here, $H$ is the function introduced in Theorem \ref{thm:c1}. (See proof of Theorem $2.1$ of \cite{White-1991}.)

For $(g,v)\in \mathscr{G}\times C^{j,\alpha}(\Sigma,N_{v_0})$, set $\mathscr{E}_{g,v}=C^{j-2,\alpha}(\Sigma,M)$.
Consider the bundle
\[
\mathscr{E}\rightarrow \mathscr{G}\times C^{j,\alpha}(\Sigma,N_{v_0})
\]
with fibers $\mathscr{E}_{g,v}$. Then $H$ can be viewed as a section of this bundle. 
Note that $\mathscr{L}$ is the zero set of $H$.

Let $g_0,g_1\in\mathscr{G}^\bullet$ and $\bar{\Gamma}:[0,1]\rightarrow\mathscr{G}$ denote a path of metrics in $\pmb{\mathscr{G}}^\bullet$ such that $\bar{\Gamma}(i)=g_i$, for $i=0,1$. There is a projection map
\begin{align*}
&\mathrm{pr}:\tilde{\mathscr{Y}}(\bar{\Gamma})\rightarrow [0,1] \\
&\mathrm{pr}(g_t,[v])=t
\end{align*}
By Proposition \ref{PathOfSuper-rigidMetrics} and Lemma $5.8$ of \cite{Taubes}, $\bar{\Gamma}$ can be choose such that
\begin{prop}\label{prop:ComeagerPaths}
	The comeager subset $\pmb{\mathscr{G}}^\bullet$ can be refined so that the following conditions hold:
	\begin{enumerate}
		\item The set of weak limit points of $\tilde{\mathscr{Y}}(\bar{\Gamma})$ is distinct from the set of critical points of $\mathrm{pr}$ and the image of these points under $\mathrm{pr}$ is distinct from $0$ and $1$.
		Moreover, both sets—the set of weak limit points and the set of critical points—are finite.
		\item Let $t\in (0,1)$ and $\mathrm{pr}^{-1}(t)$ does not contain any critical point of $\mathrm{pr}$ or weak limit point, then $\mathrm{pr}^{-1}(t)$ is a finite set of super-rigid elements. 
	\end{enumerate}
\end{prop}	

Thus, for all but finitely many points $t\in\{t_1,...,t_m\}$, $\bar{\Gamma}(t)$ is a super-rigid metric.

\subsubsection{Local Model Around A Critical Point}\label{Subsection:LocalModel_CriticalPoint}
First let us assume that $t=t_k$ is the image of a critical point, say $(g_t,[v])$ of $\mathrm{pr}$.
Following arguments is basically the ones in \cite{EF-CLosedGeodesics} and \cite{Taubes}.
Let us identify $g_t$ with $t$ and a $1-$dimensional neighborhood of $g_t$ with $(t-\epsilon_0,t+\epsilon_0)$.
We want to describe a neighborhood of $(t,[v])$. 
Note that if $(t+\epsilon,[v_\epsilon])\in \tilde{\mathscr{Y}}(\bar{\Gamma})$ is near to $(t,[v])$ then 
\begin{align}\label{eq:5}
H(t+\epsilon,s_\epsilon)=0
\end{align}
where $[v_\epsilon]=[\mathrm{exp}\circ s_\epsilon]$ and $H:(t-\epsilon_0,t+\epsilon_0)\times C^{j,\alpha}(\Sigma,M)\rightarrow C^{j-2,\alpha}(\Sigma,M)$ is obtained from the restriction of $H$. 
Recall that by Theorem \ref{thm:c1}, $D_2H(s,w)$ is a $0-$index Fredholm operator at $(s,w)\in(t-\epsilon_0,t+\epsilon_0)\times C^{j,\alpha}(\Sigma,M)$.

Let $\mathcal{J}=D_2H(t,0)$. As $(t,[v])$ is a critical point of $\mathrm{pr}$, $\mathrm{Ker}(\mathcal{J})$ is not zero, since otherwise the implicit function theorem implies that for any $\epsilon$ sufficiently close to $0$ there is a unique solution $s_\epsilon$ for Equation \ref{eq:5} which contradicts with the fact that $(t,[v])$ is a critical point of $\mathrm{pr}$.
Therefore $\mathrm{Ker}(\mathcal{J})\ne 0$. As discussed in the previous section we can assume $\mathrm{Ker}(\mathcal{J})$ is $1-$dimensional. Let $\theta$ denotes the generator.

Note that
\begin{align*}
	\mathrm{Ker}(DH)=\mathbb{R}\oplus\mathrm{Ker}(\mathcal{J}) \quad \mathrm{and} \quad 
	\mathrm{Coker}(DH)=\mathrm{Coker}(\mathcal{J})
\end{align*}

Let $X$ denote $C^{j,\alpha}(\Sigma,M)/\left \langle \theta \right \rangle$ and $Y=\mathrm{Im}(DH)$. Then
\begin{align*}
	&C^{j,\alpha}(\Sigma,M)\cong\mathbb{R}\oplus X=\mathrm{Ker}(\mathcal{J})\oplus X  \\
	&C^{j-2,\alpha}(\Sigma,M)\cong\mathbb{R}\oplus Y=\mathrm{Coker}(\mathcal{J})\oplus Y
\end{align*}

Let
\begin{align*}
	& H:[0,1]\times \mathbb{R}\times X\rightarrow \mathbb{R}\times Y \\
	& H(\epsilon,r,x)=(H_1(\epsilon,r,x),H_2(\epsilon,r,x))
\end{align*}

Since the differential of $H_2$ with respect to $x$ at $(t,0,0)$ is invertible, therefore, by implicit function theorem for $\epsilon$ and $r$ close to zero, there is a function $x=x(\epsilon,r)$ which solves $H_2(t+\epsilon,r,x)=0$. This implies that $H(t+\epsilon,r,x)=0$ is equivalent to $g(\epsilon,r):=H_1(t+\epsilon,r,x(\epsilon,r))=0$.

Since $\mathbb{R}\cong\left \langle \theta \right \rangle=\mathrm{Ker}(\mathcal{J})$, therefore, $\frac{\partial\,g}{\partial\,r}(t,0)=0$ (Differentiation of $g$ with respect to $r$ corresponds to $\mathcal{J}$).

Since we can assume $\mathrm{dim}(\mathrm{Ker}(\mathcal{J}))=1$, therefore, $\frac{\partial^2 g}{\partial r^2}(t,0)\ne 0$. 
Since $\tilde{\mathscr{Y}}(\bar{\Gamma})$ is a manifold, therefore $\frac{\partial g}{\partial \epsilon}(t,0)\ne 0$. 
Then the Taylor's expansion of $g$ around $(t,0)$ is of the form
\[
g(\epsilon,r)=a_1\epsilon+b_2 r^2+...
\]
where $a_1$ and $b_2$ are non-zero.
This leads us to the following Proposition.
\begin{prop}[Lemma $5.9$ of \cite{Taubes}]\label{prop:LocalModelCriticalPoint}
There is an open interval $I$ of the form $(t-\epsilon,t)$ or $(t,t+\epsilon)$ and there are two maps 
\begin{align*}
&\rho_{\pm}:I\rightarrow\tilde{\mathscr{Y}}(\bar{\Gamma}) \\
&\rho_{\pm}(s)=(g_s,[v_{\pm}^s])
\end{align*} 
so that the following properties hold (see Figure \ref{fig:fig1}):

	\begin{enumerate}
		\item $\mathrm{pr}\circ \rho_{\pm}(s)=s$. 
		\item The two curves $v_{\pm}^s$ converges to $v$ and $g_s$ converges to $g_t$.
		\item For $v_{\pm}^s$, the sign of $\mathrm{det}(\mathcal{J}_{g_s,v_{\pm}^s,\iota}^{\pm})$ changes only when $\iota=0$. Moreover, $\mathrm{sign}(\mathrm{det}(\mathcal{J}_{g_s,v_{+}^s,0}^+))>0$ and $\mathrm{sign}(\mathrm{det}(\mathcal{J}_{g_s,v_{-}^s,0}^-))<0$.
	\end{enumerate}
\end{prop}
\begin{figure}[t]
	\includegraphics[width=7cm,height=5cm]{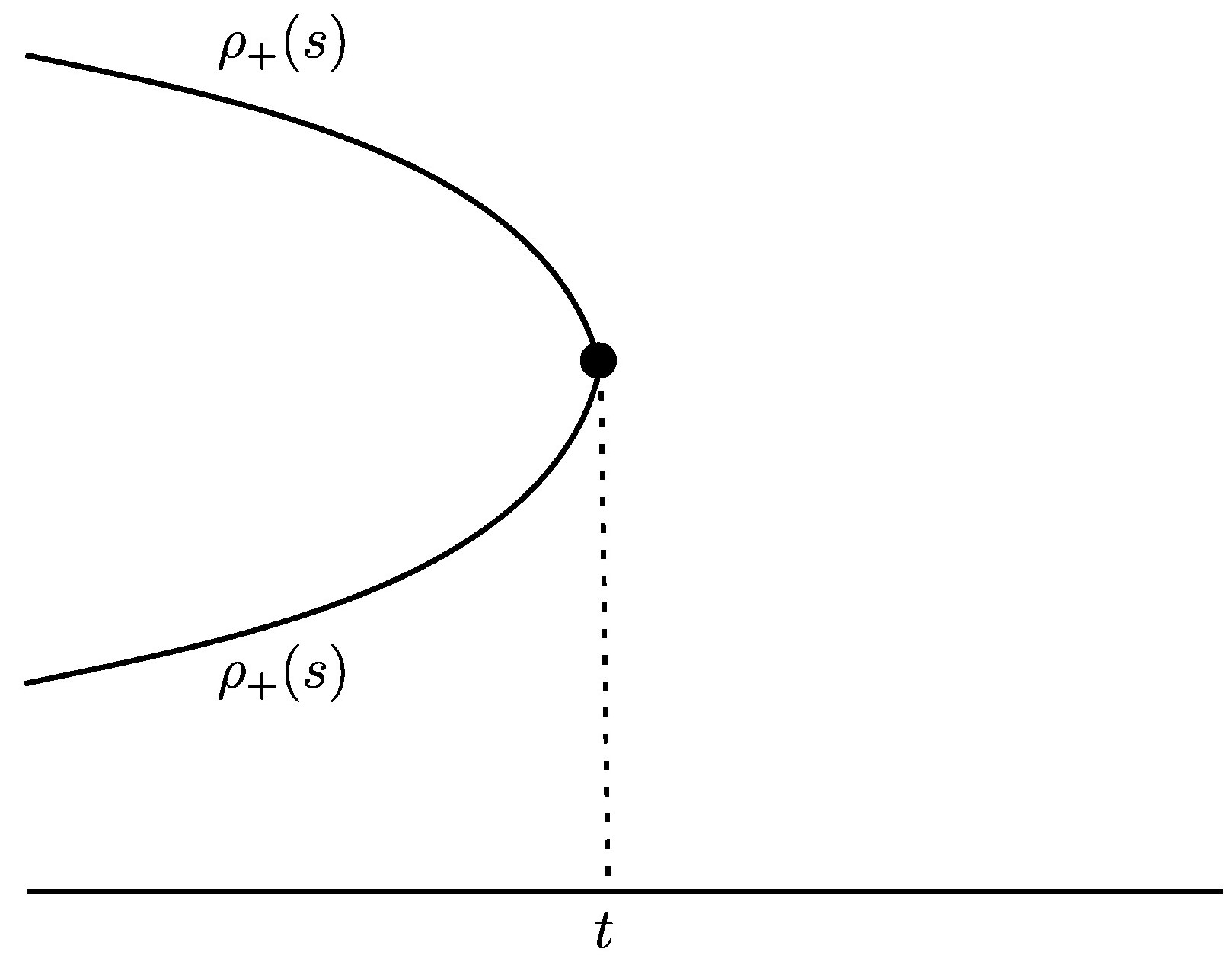}
	\caption{Two sequences in $\tilde{\mathscr{Y}}(\bar{\Gamma})$ which converge to a curve $(g_t,[v])$.}
	\label{fig:fig1}
\end{figure}
In particular, the counting function in Definition \ref{def:1} remains unchanged as $s$ passes through $t$ i.e. for $s\in(t-\epsilon,t+\epsilon)$ and $s\ne t$, $n(g_s,\mathcal{U}_s)$ is constant.

\subsubsection{Local Model Around A Weak Limit Point}\label{Subsection:LocalModel_WeakLimitPoint}
Let consider the case where for $t=t_k$ and $(g_t,[v])$ is a weak limit point i.e. there is a sequence $(g_{t_k},v_k)\in \mathscr{Y}(\bar{\Gamma})$ such that $t_k\rightarrow t$ and the embeddings $v_k:T_k\rightarrow M$ converge to $v_0:T_0\rightarrow M$ and $v_0$ factors as $v\circ\pi$, where $\pi:T_0\rightarrow T$ is a covering map and $v:T\rightarrow M$ is an embedding. 
Following the discussion in the previous subsection, $\pi$ is a $2$ to $1$ covering map.
\begin{figure}[h]
	\includegraphics[width=8cm,height=4cm]{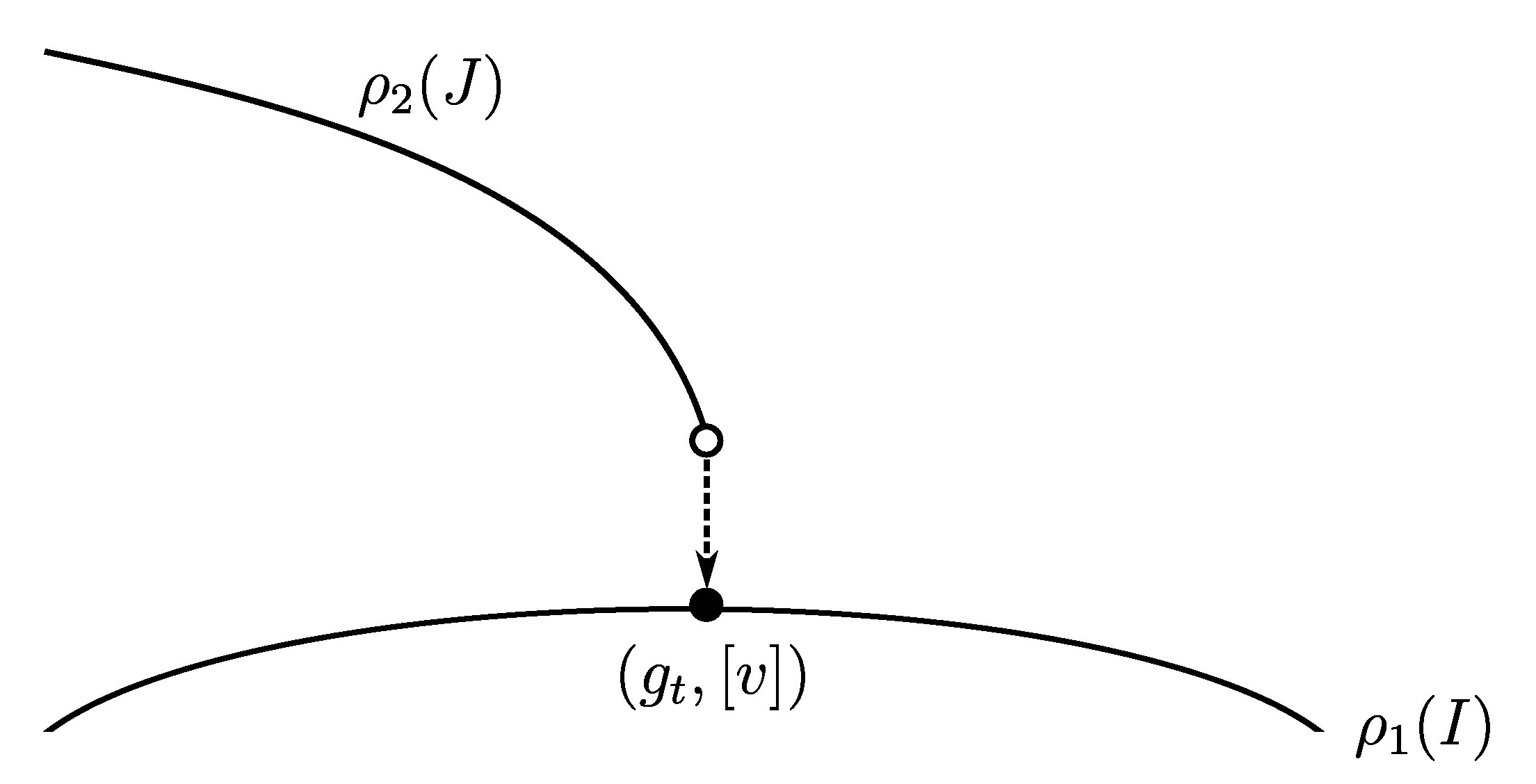}
	\caption{A sequence in $\tilde{\mathscr{Y}}(\bar{\Gamma})$ which converge to the double cover of $(g_t,[v])$.}
	\label{fig:fig2}
\end{figure}
In this situation, according to the discussion in proof of Theorem $3.2$ of \cite{EF-CLosedGeodesics}, a local model for a neighborhood of $(g_t,[v])$ is given by the zero set of
\[
g(\epsilon,r)=r((\epsilon-t)f(\epsilon,r)-r^2h(r))
\]
where, $f:I\times \mathbb{R}\rightarrow\mathbb{R}$ and $h:\mathbb{R}\rightarrow\mathbb{R}$ with $f(t,0),g(0)\ne 0$.
Therefore, the following holds:
\begin{prop}[Lemma $5.10$ of \cite{Taubes}]\label{prop:LocalModelWeakLimitPoint}
There is an interval $J$ of the form $(t-\epsilon,t)$ or $(t,t+\epsilon)$ and there are two maps 
\begin{align*}
	\rho_1:I\rightarrow\tilde{\mathscr{Y}}(\bar{\Gamma}) \quad \mathrm{and} \quad
	\rho_2:J\rightarrow\tilde{\mathscr{Y}}(\bar{\Gamma})
\end{align*} 
where $I=(t-\epsilon,t+\epsilon)$, such that the following properties hold (see Figure \ref{fig:fig2}):
\begin{enumerate}
	\item $\mathrm{pr}\circ\rho_1=\mathrm{id}_I$. Moreover, $\rho_1(t)=(g_t,[v])$ and $(g_t,[v])$ is not the limit point of any sequence in $\tilde{\mathscr{Y}}(\bar{\Gamma})\setminus\rho_1(I)$.
	\item $\mathrm{pr}\circ\rho_2=\mathrm{id}_J$. Moreover,
	\[
	\lim_{s \to t} \rho_2(s)=(g_t,[v_0])
	\]  
	and $(g_t,[v_0])$ is not the limit point of any sequence in $\tilde{\mathscr{Y}}(\bar{\Gamma})\setminus\rho_2(J)$.
	\item Let $\rho_1(s)=(g_s,[v_s])$. Then for all $\iota \in H^2(T,\mathbb{Z}_2)$, the kernel of $\mathcal{J}_{g_s,v_s,\iota}$ is trivial unless $\iota=\iota_0$ where $\iota_0$ is the element which classifies $\pi$. 
	For $\iota_0$, $\mathrm{Ker}(\mathcal{J}_{g_s,v_s,\iota_0})$ is trivial unless $s=t$, for which this kernel is $1-$dimensional.
	\item As $s$ passes through $t$, the sign of $\mathrm{det}(\mathcal{J}_{g_s,v_s,\iota})$ does not change unless $\iota=\iota_0$.
\end{enumerate}
\end{prop}
In particular, this suggests that $n_{\pm k}^d(g,v)=0$ unless $d=2^m$, for some $m\in\mathbb{N}$.

$\pi:T_0\rightarrow T$ induces a $2$ to $1$ map on cohomology
\[
\pi^*:H^1(T,\mathbb{Z}_2) \rightarrow H^1(T_0,\mathbb{Z}_2).
\]
$\iota_0$ which classifies $\pi$ will be sent to $0$ by $\pi^*$.
Let $\rho_2(s)=(g_s,[v'_s])$. The following rules tell us how the sign of $\mathrm{det}(\mathcal{J}_{\rho_2(s),\iota})$, which is the sign of a torus in the image of $\rho_2$, is determined by the sign of $\mathrm{det}(\mathcal{J}_{\rho_1(s),\iota_0})$ 
and the sign of $\mathrm{det}(\mathcal{J}_{\rho_1(s),\iota})$, which is the sign of a torus in the image of $\rho_1$ (cf. Lemma $5.11$ \cite{Taubes}):
\begin{itemize}
	\item
	\begin{align}\label{eq:2}
	 \mathrm{sign}(\mathrm{det}(\mathcal{J}_{\rho_2(s)}))=-\mathrm{sign}(\mathrm{det}(\mathcal{J}_{\rho_1(s),\iota_0})). \mathrm{sign}(\mathrm{det}(\mathcal{J}_{\rho_1(s)}))
    \end{align} 
	
	\item Let $\iota \in \mathrm{Im}(\pi^*)$ be a non-zero class, then
	\begin{align}\label{eq:3}
	\mathrm{sign}(\mathrm{det}(\mathcal{J}_{\rho_2(s),\iota}))=\prod_{\kappa:\pi^*(\kappa)=\iota} \mathrm{sign}(\mathrm{det}(\mathcal{J}_{\rho_1(s),\kappa}))
	\end{align}
	\item Let $\iota \in H^1(T,\mathbb{Z}_2)\setminus\mathrm{Im}(\pi^*)$, then
	\begin{align}\label{eq:4}
	\mathrm{sign}(\mathrm{det}(\mathcal{J}_{\rho_2(s),\iota}))=1
	\end{align}
\end{itemize}

These rules tell us how the type of a torus changes as one passes through a non-rigid metric. For example, let $T$ denote an embedded torus of type $+0$ 
(the torus at the bottom left of Figure \ref{fig:fig2}). Then 
\begin{align*}
\mathrm{sign}(\mathrm{det}(\mathcal{J}_{\rho_1(s),\iota_0}))=+1 
\quad \mathrm{and} \quad
\mathrm{sign}(\mathrm{det}(\mathcal{J}_{\rho_1(s)}))=+1.
\end{align*}
By Equation \ref{eq:2}, for the tours on top left of Figure \ref{fig:fig2} the image of trivial element under $\delta$ is $-1$ and by Equations \ref{eq:3} and \ref{eq:4}, the image of other elements is $+1$. 
Therefore, this torus (the tours on top left of Figure \ref{fig:fig2}) is of type $-0$.

\subsubsection{The Invariance of the Count Function}\label{Subsection:InvarianceOfTheCountFunction}
\begin{figure}[h]
	\begin{tabular}{c:cc}
		
		\begin{tikzpicture}
			
			\node[rectangle,
			draw=lightgray,
			text = olive] (r) at (0,0) {	\includegraphics[width=4.2cm,height=2.6cm]{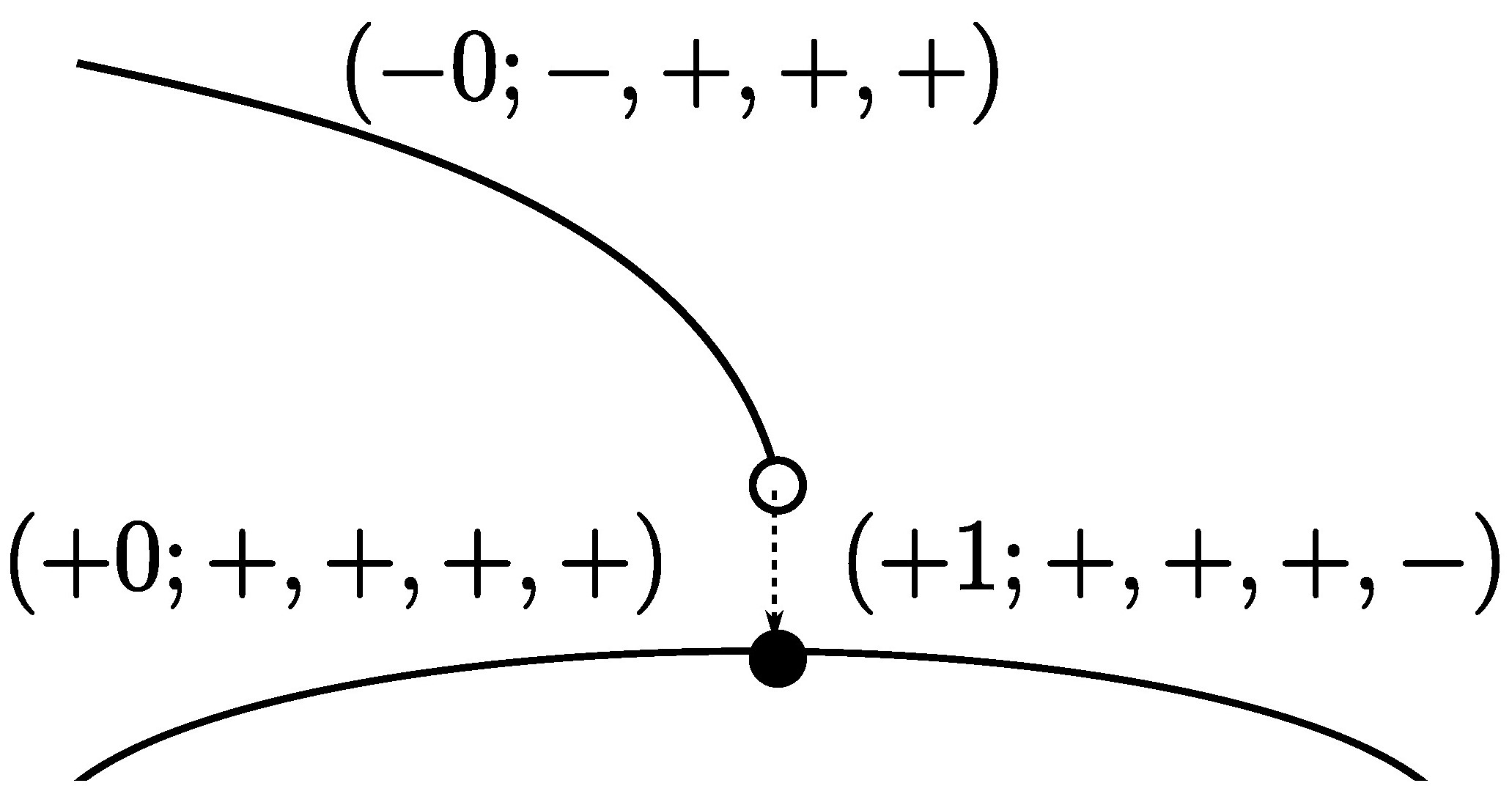}{a}};
			
		\end{tikzpicture}

		& 
		\begin{tikzpicture}
			
			\node[rectangle,
			draw=lightgray,
			text = olive] (r) at (0,0) {	\includegraphics[width=4.2cm,height=2.6cm]{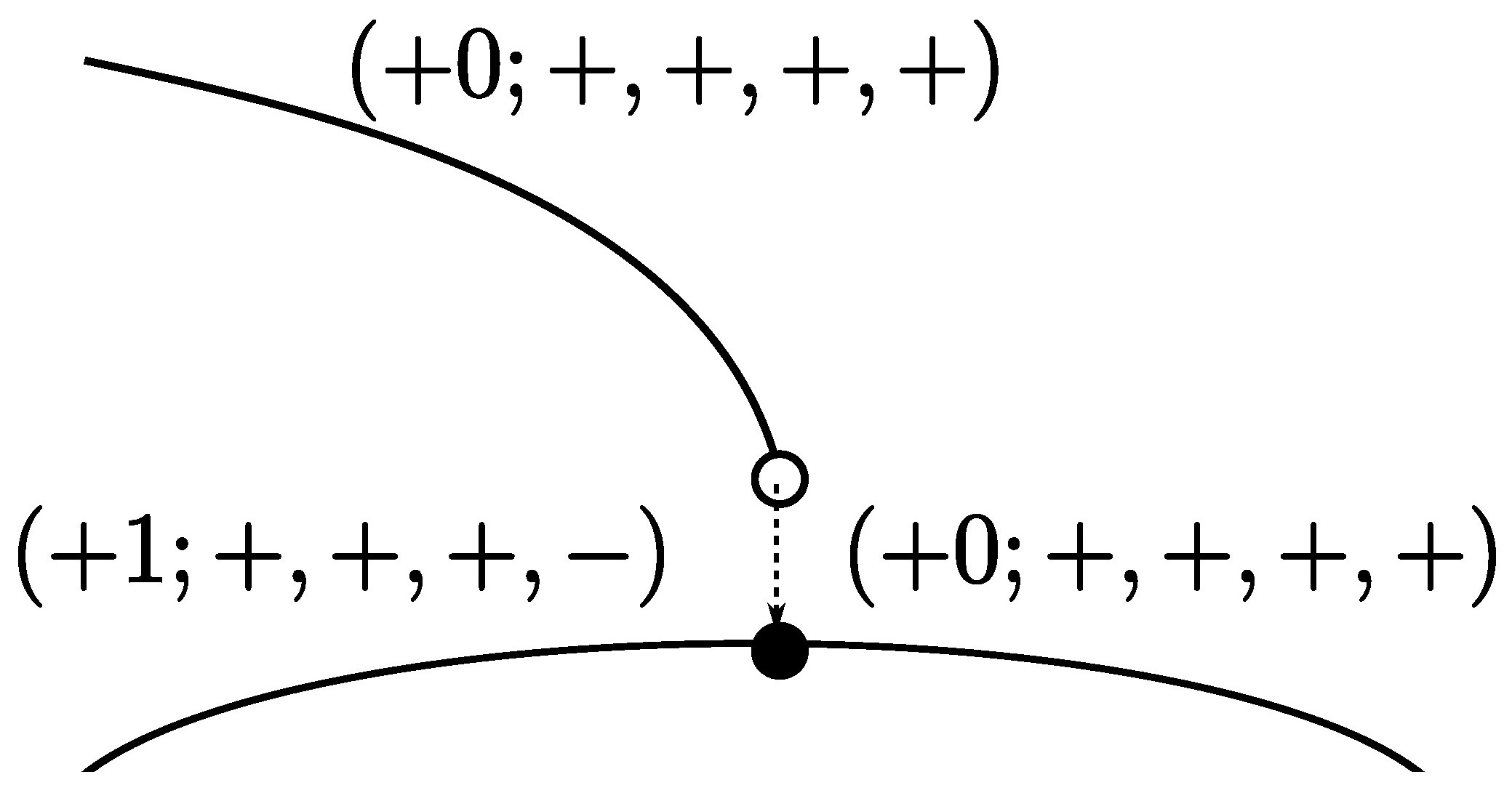}{b}};
			
		\end{tikzpicture}
		&
		\begin{tikzpicture}
			
			\node[rectangle,
			draw=lightgray,
			text = olive] (r) at (0,0) {	\includegraphics[width=4.2cm,height=2.6cm]{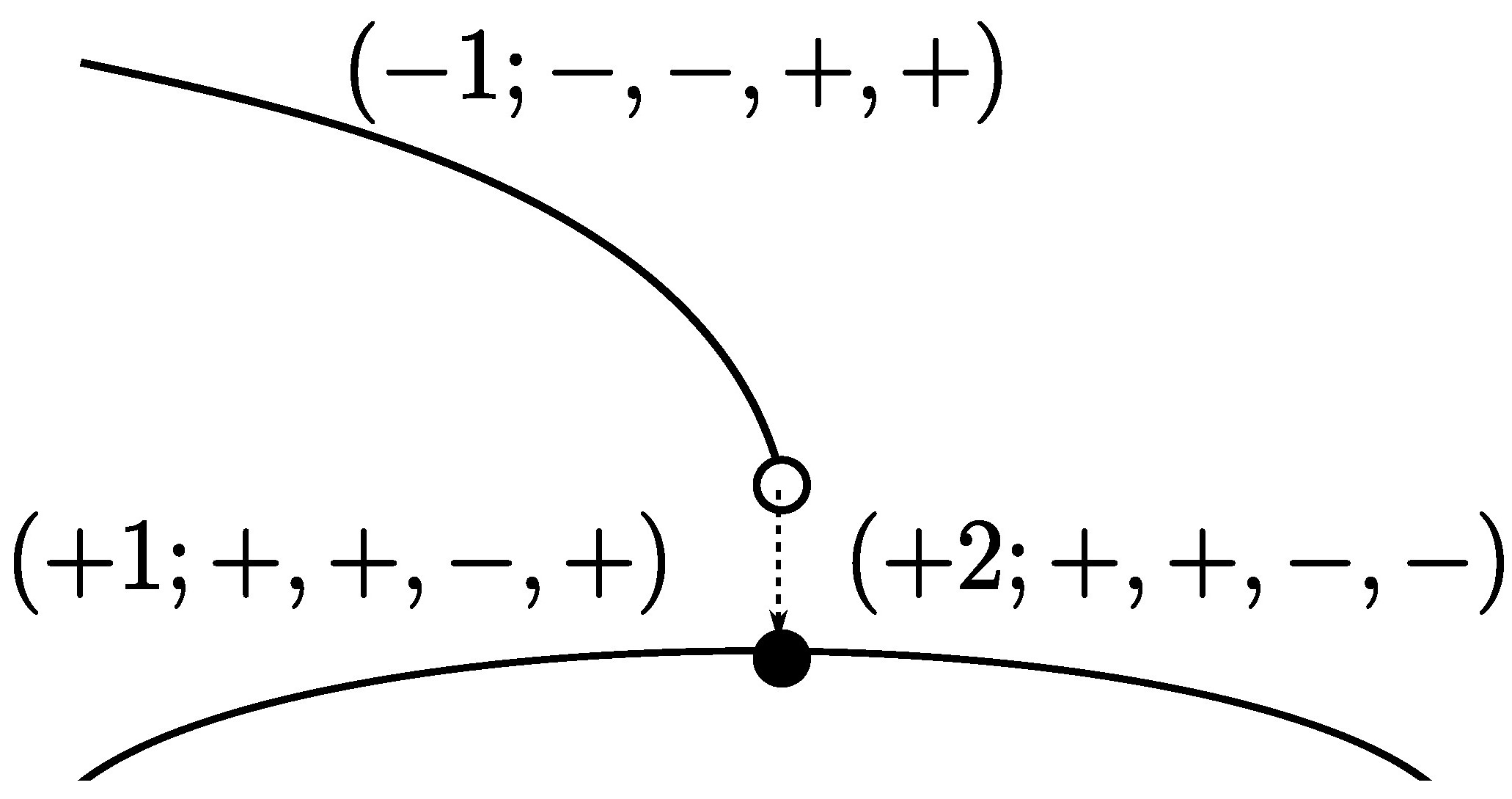}{c}};
			
		\end{tikzpicture}
		\\
		\begin{tikzpicture}
			
			\node[rectangle,
			draw=lightgray,
			text = olive] (r) at (0,0) {	\includegraphics[width=4.2cm,height=2.6cm]{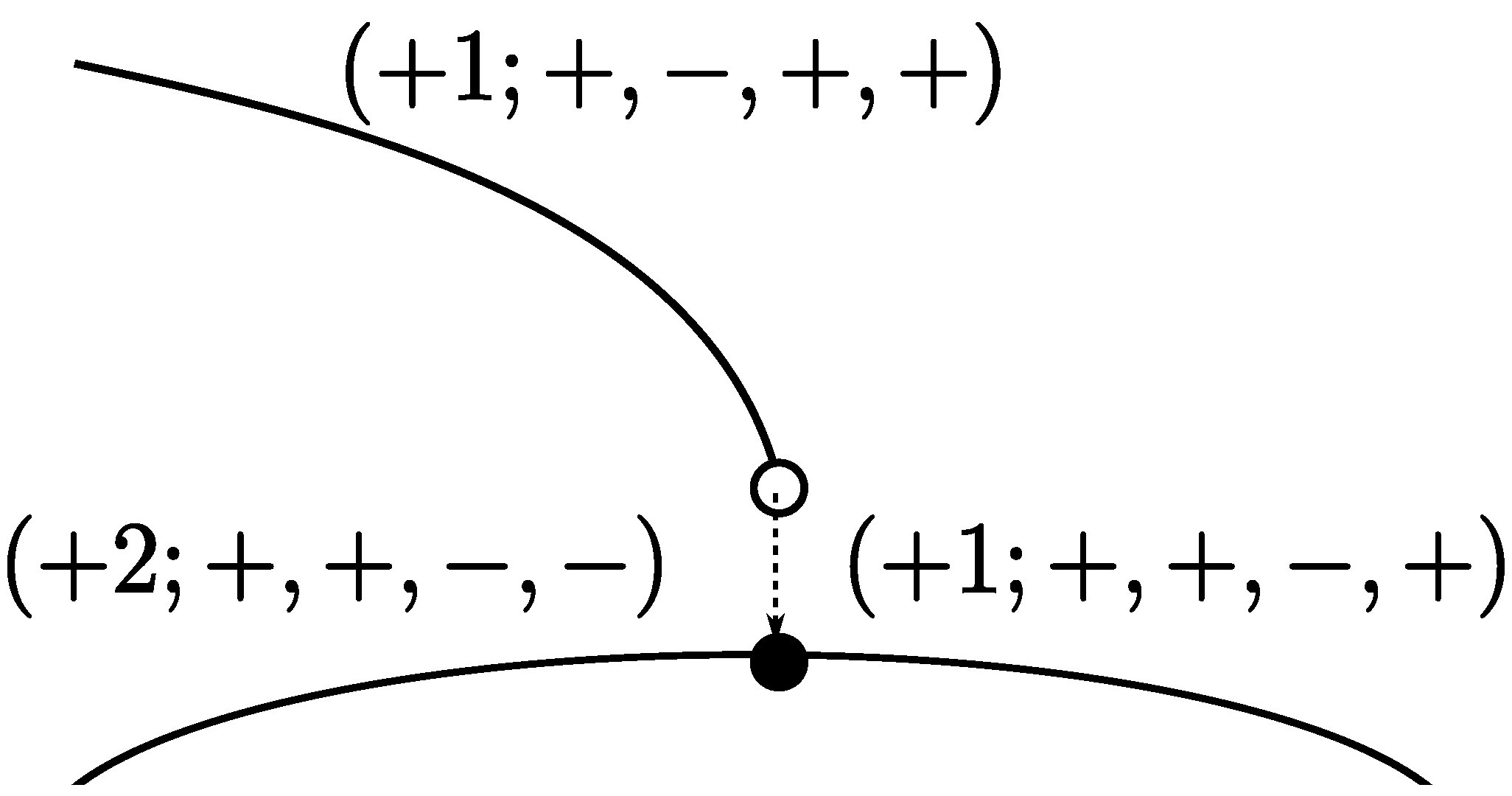}{d}};
			
		\end{tikzpicture}
		& 
		\begin{tikzpicture}
			
			\node[rectangle,
			draw=lightgray,
			text = olive] (r) at (0,0) {	\includegraphics[width=4.2cm,height=2.6cm]{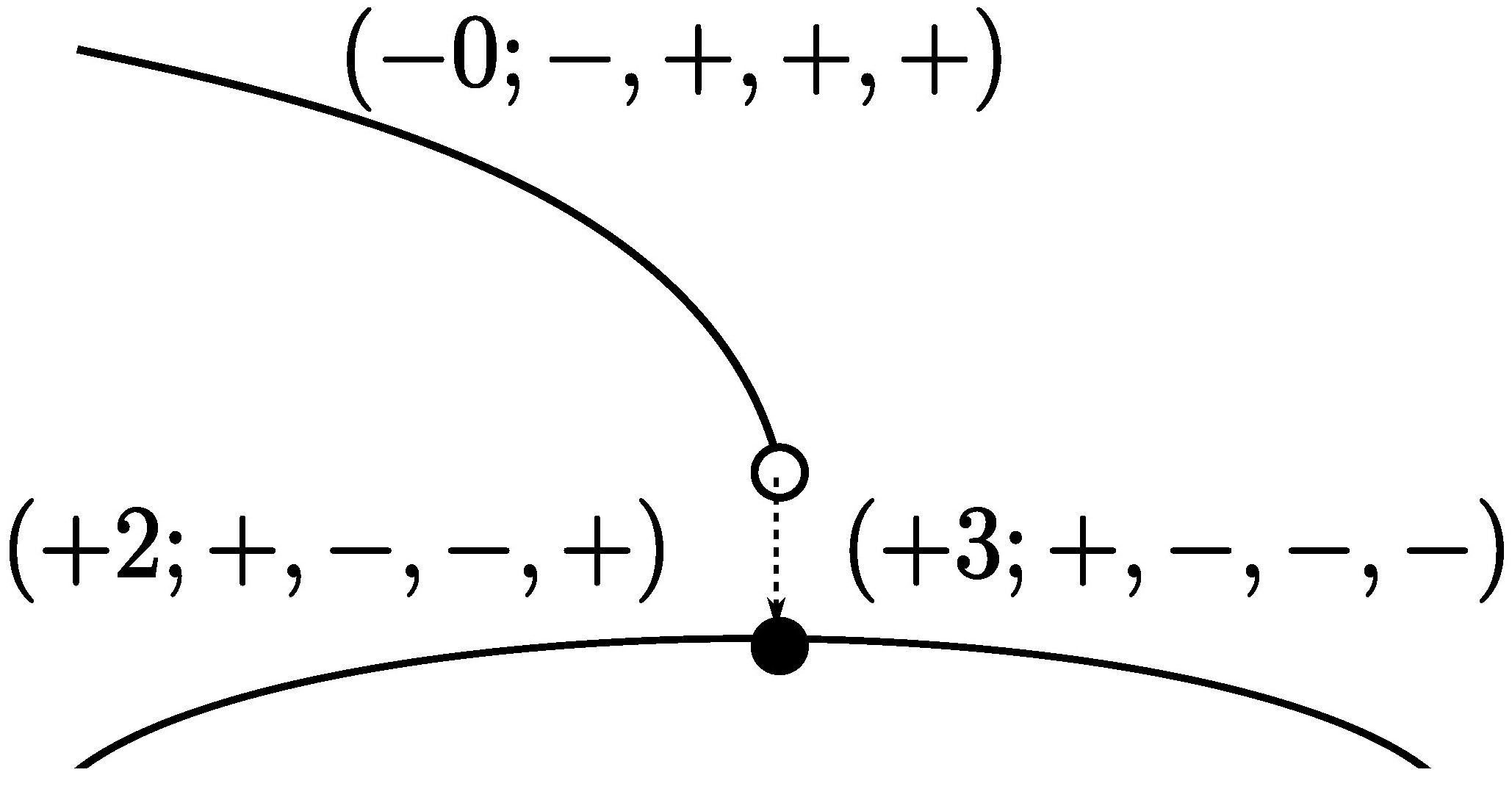}{e}};
			
		\end{tikzpicture}
		&
		\begin{tikzpicture}
			
			\node[rectangle,
			draw=lightgray,
			text = olive] (r) at (0,0) {	\includegraphics[width=4.2cm,height=2.6cm]{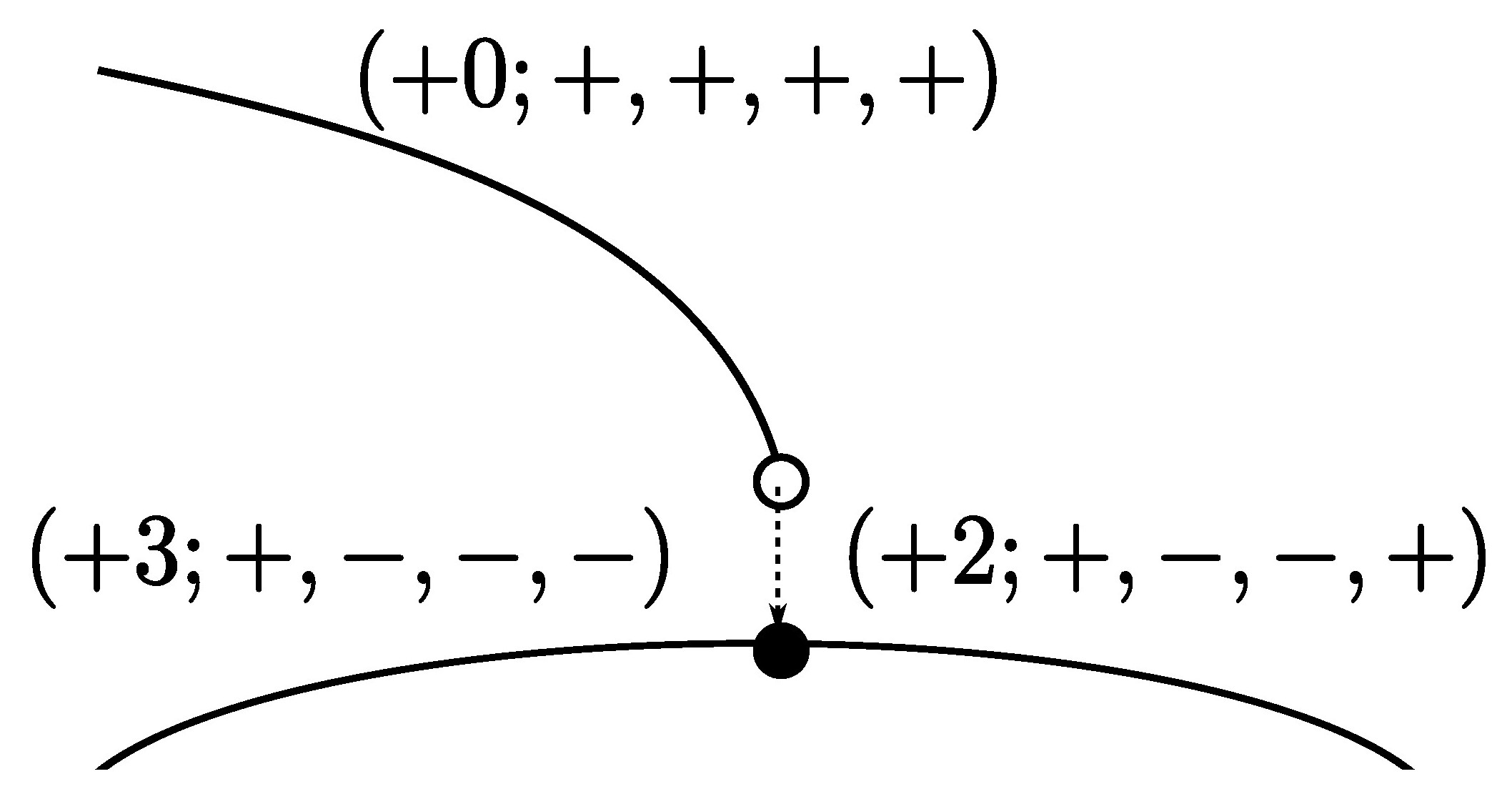}{f}};
			
		\end{tikzpicture}
		\\ 
		\begin{tikzpicture}
			
			\node[rectangle,
			draw=lightgray,
			text = olive] (r) at (0,0) {	\includegraphics[width=4.2cm,height=2.6cm]{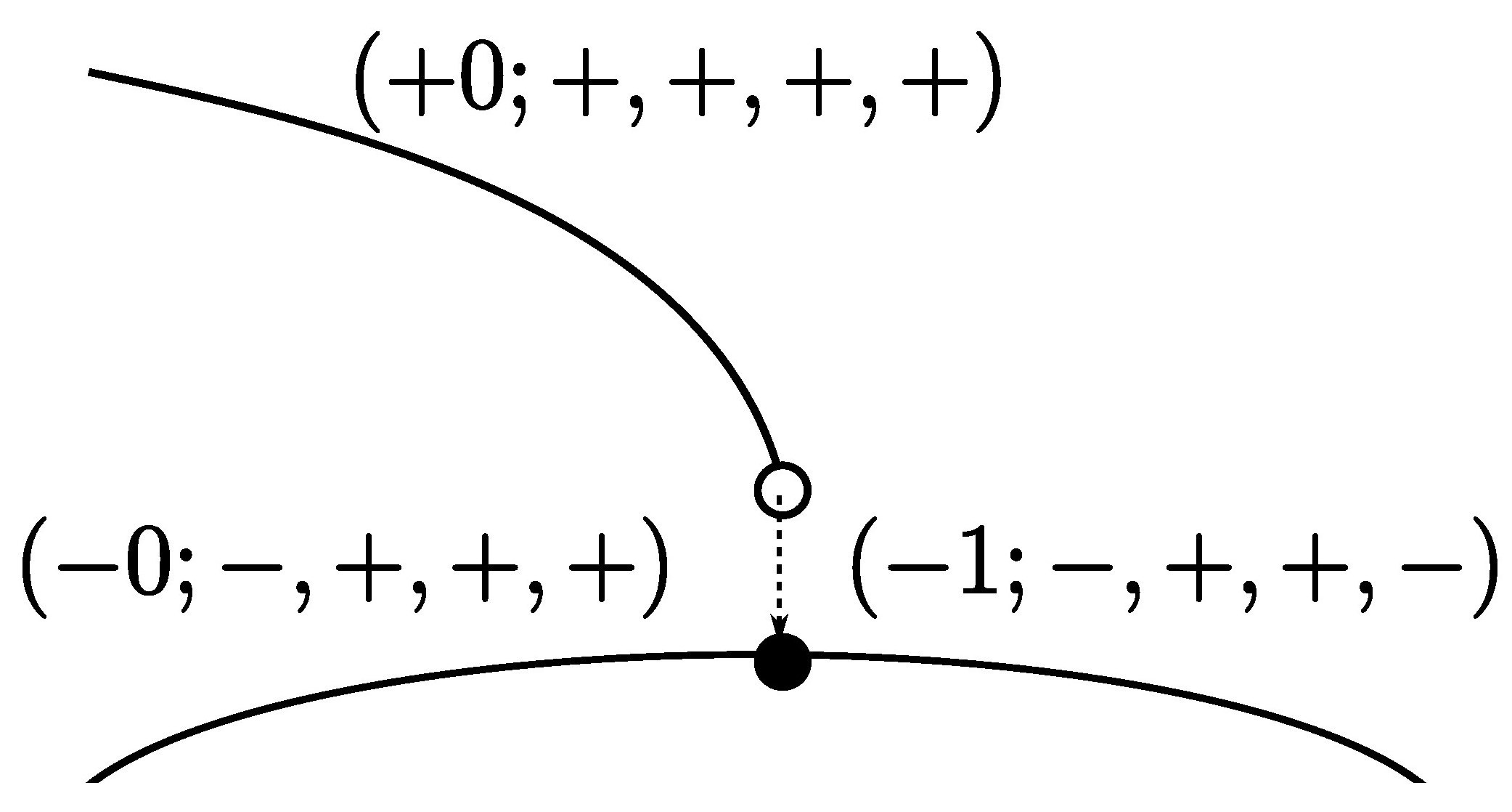}{g}};
			
		\end{tikzpicture}
		& 
		\begin{tikzpicture}
			
			\node[rectangle,
			draw=lightgray,
			text = olive] (r) at (0,0) {	\includegraphics[width=4.2cm,height=2.6cm]{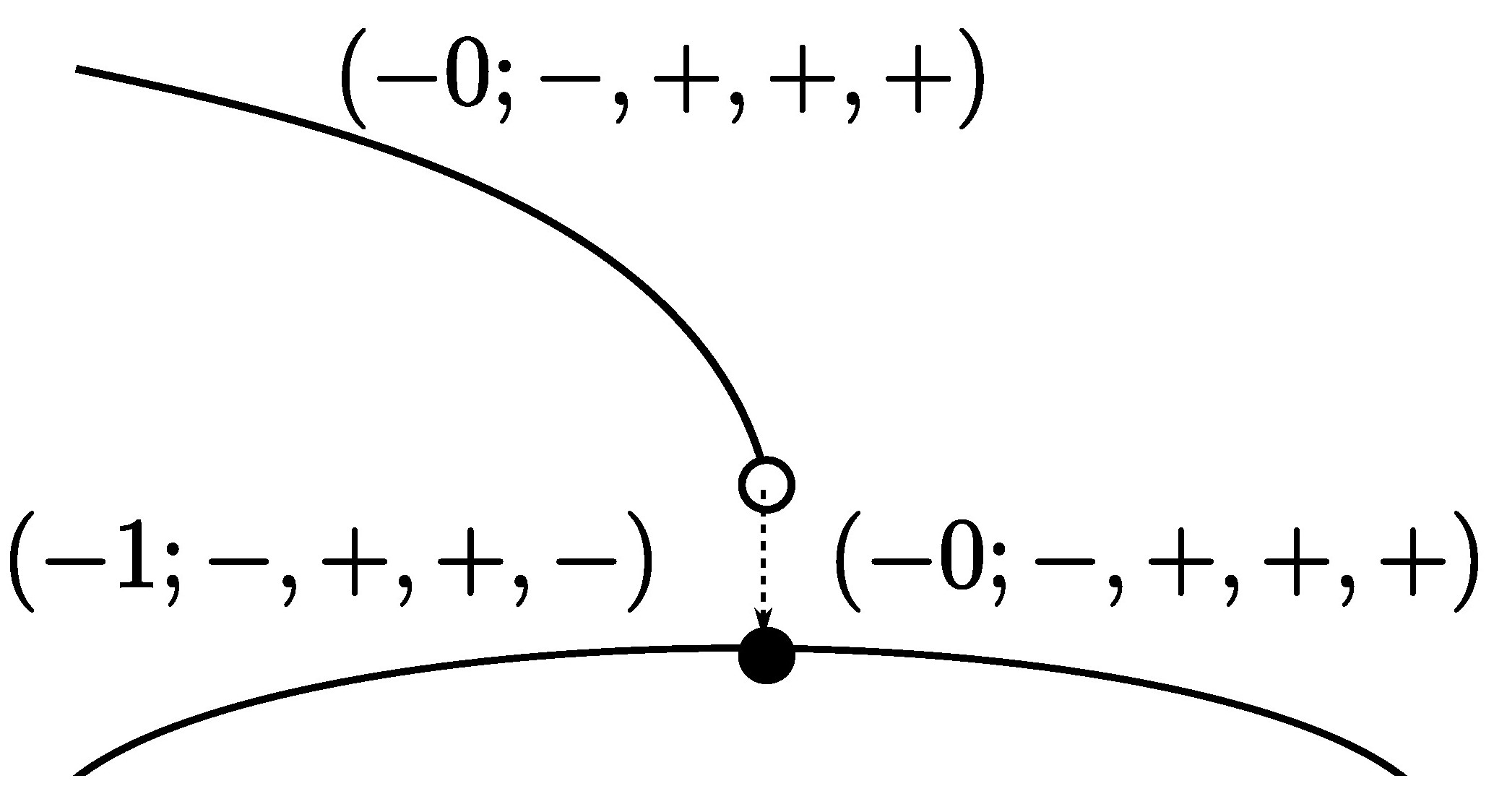}{h}};
			
		\end{tikzpicture}
		&
		\begin{tikzpicture}
			
			\node[rectangle,
			draw=lightgray,
			text = olive] (r) at (0,0) {	\includegraphics[width=4.2cm,height=2.6cm]{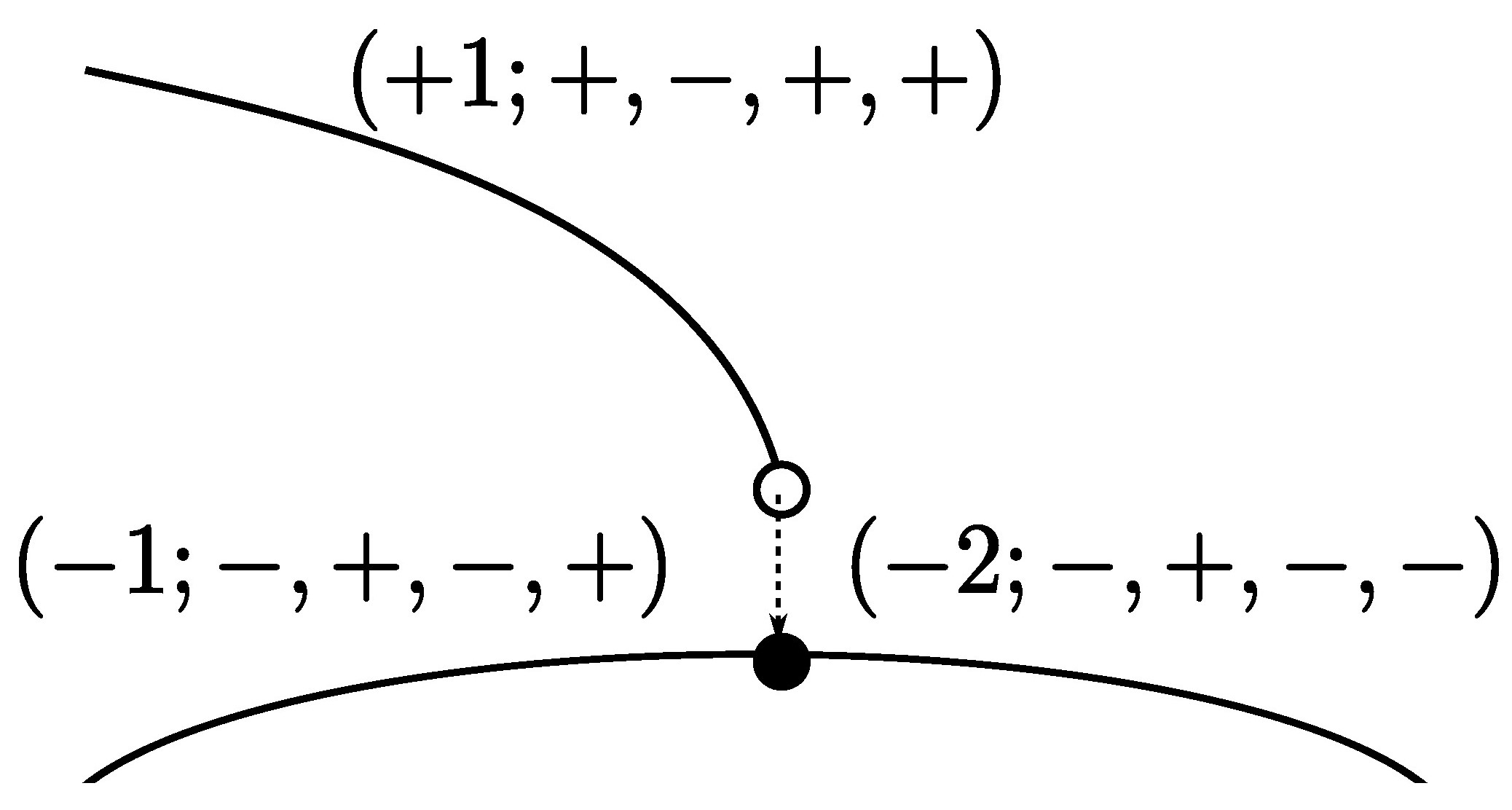}{i}};
			
		\end{tikzpicture}
		\\
		\begin{tikzpicture}
			
			\node[rectangle,
			draw=lightgray,
			text = olive] (r) at (0,0) {	\includegraphics[width=4.2cm,height=2.6cm]{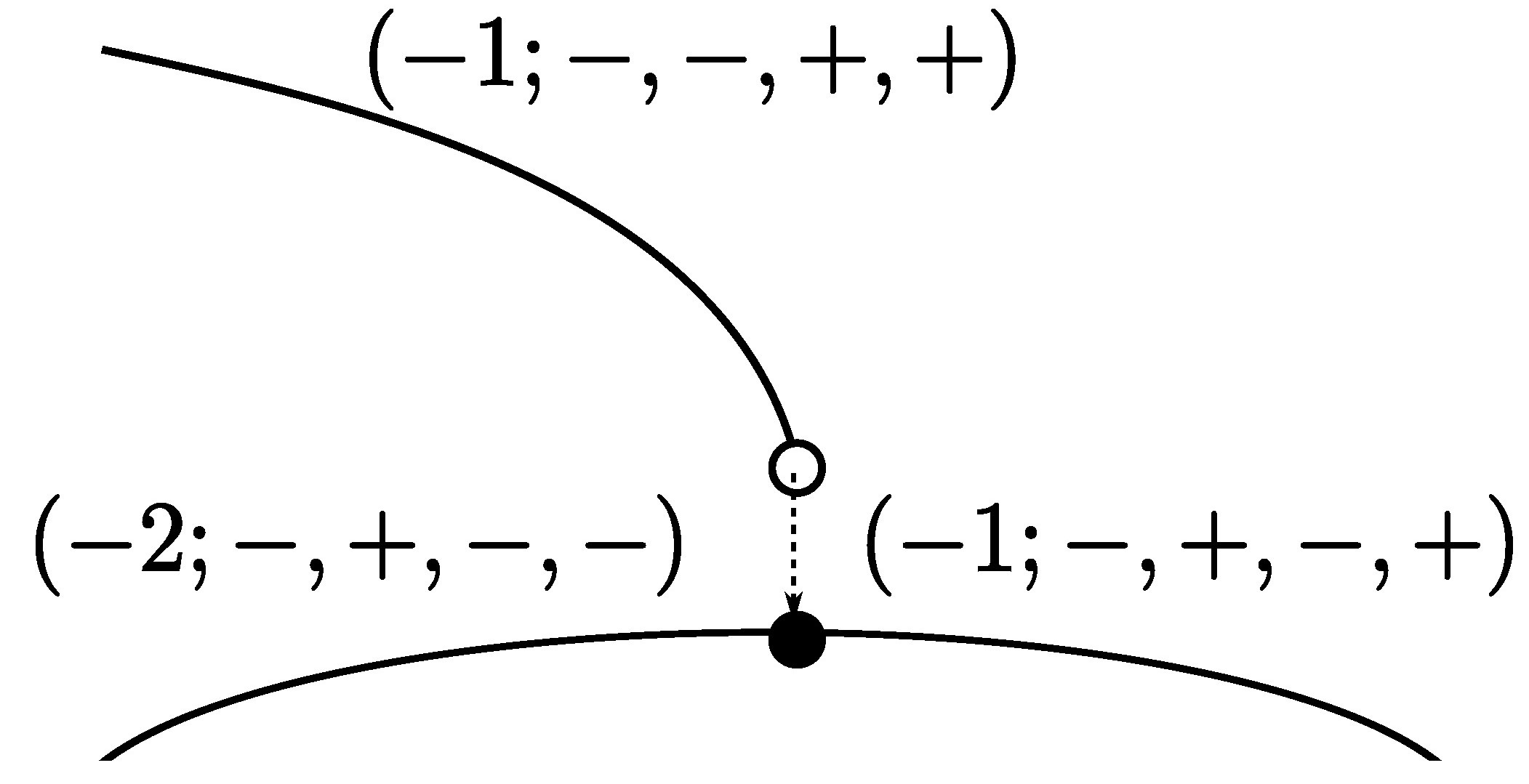}{j}};
			
		\end{tikzpicture}
		& 
		\begin{tikzpicture}
			
			\node[rectangle,
			draw=lightgray,
			text = olive] (r) at (0,0) {	\includegraphics[width=4.2cm,height=2.6cm]{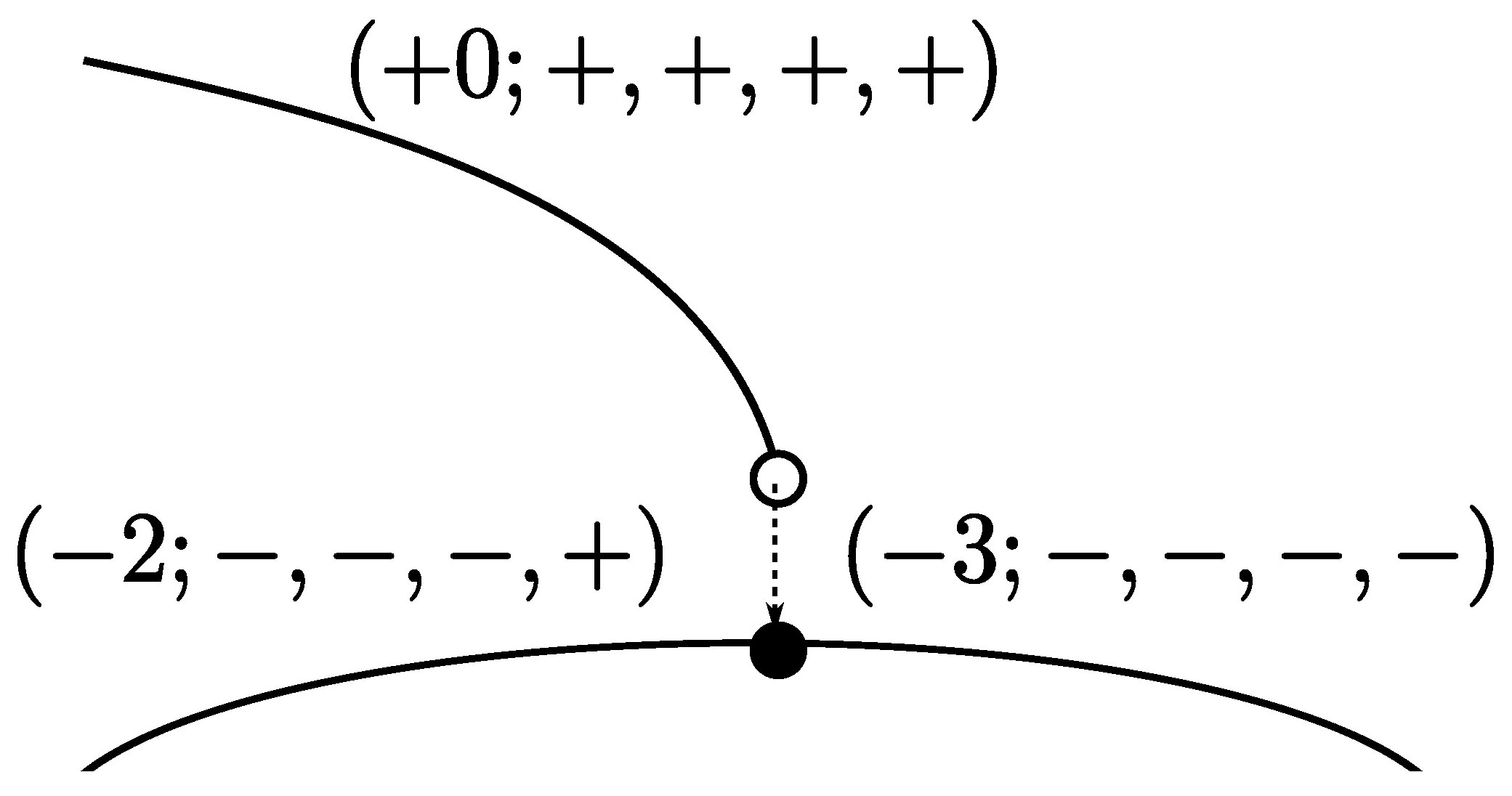}{k}};
			
		\end{tikzpicture}
		&
		\begin{tikzpicture}
			
			\node[rectangle,
			draw=lightgray,
			text = olive] (r) at (0,0) {	\includegraphics[width=4.2cm,height=2.6cm]{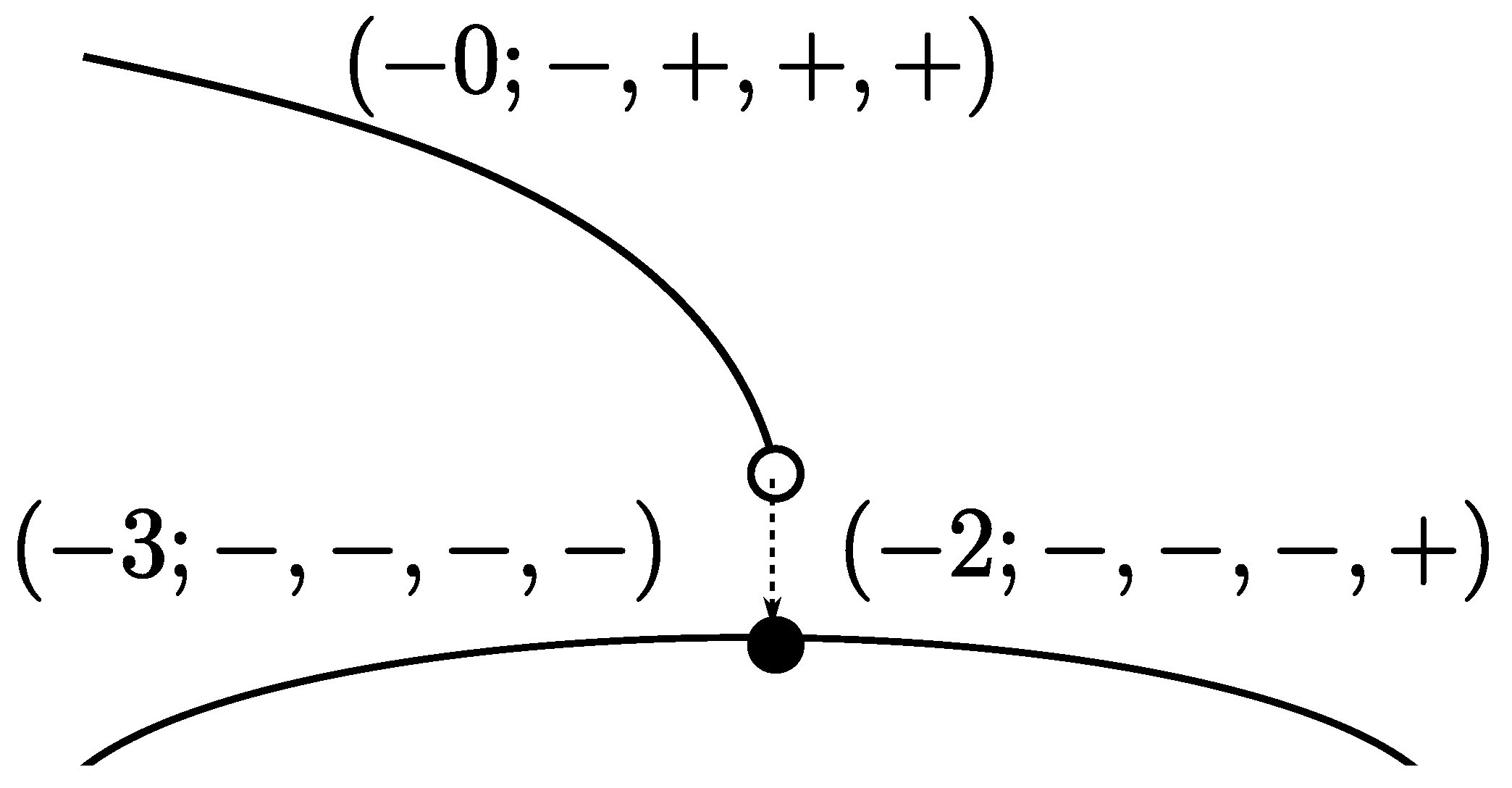}{l}};
			
		\end{tikzpicture}
	\end{tabular}
	\caption{Type of the torus in a neighborhood of $(g_t,[v])$, i.e. in $\rho_1(I)$, and in a neighborhood of $(g_t,[v\circ \pi])$, i.e. in $\rho_2(J)$ is shown in various cases that may occur.}
	\label{fig:fig3}
\end{figure}

Here, we  demonstrate the invariance of the Definition \ref{def:1}. 

\begin{thm}\label{thm:Well-definedness}
	Let $g_0,g_1\in\mathscr{G}^\bullet$ and 
	\[
	\bar{\Gamma}:[0,1]\rightarrow\mathscr{G}
	\] 
	denote a path in $\pmb{\mathscr{G}}^\bullet$ with $\bar{\Gamma}(i)=g_i$, $i=0,1$.
	Take a compact and open subset $\bar{\mathcal{U}}$ in $\mathscr{L}(\bar{\Gamma})$ and let $\mathcal{U}_t:=\bar{\mathcal{U}}\cap\mathscr{L}(\bar{\Gamma}(t))$.
	Then
	\[
	n(g_0,\mathcal{U}_0)=n(g_1,\mathcal{U}_1).
	\]
\end{thm}
\renewcommand*{\proofname}{Proof}
\begin{proof}
	The proof is a rewrite of the proof of Theorem $3.2$ of \cite{EF-CLosedGeodesics} with some modifications suitable for minimal tori.
	Let us denote by $g_t$ the image of $t$ under $\bar{\Gamma}$.
	Note that each element $u\in\bar{\mathcal{U}}$ factors as $u=v\circ\pi$ where $\pi$ is a covering map and $v$ is an almost embedding.
	Therefore, 
	$$\bar{\mathcal{U}}=\cup_{k}\,\bar{\mathcal{U}}^k,$$
	where $\bar{\mathcal{U}}^k$ denotes the subset of all degree-$k$ covers of somewhere injective torus.
	Let
	\[
	\mathcal{V}^k:=\{(g,v)\in\mathscr{Y}(\bar{\Gamma})\,|\,(g,v\circ\pi_k)\in\mathscr{L}(\bar{\Gamma})\}
	\quad\text{and}\quad
	\mathcal{V}:=\cup_k\,\mathcal{V}^k
	\]
	where $\pi_k$ denotes a degree-$k$ covering map.
	Note that $\mathcal{V}$ is an open subset.
	
	Denote by $\{t_1,...,t_m\}$, $t_1<...<t_m$, the set of all critical points of the projection map
	\[
	\mathrm{pr}:\tilde{\mathscr{Y}}(\bar{\Gamma})\rightarrow [0,1]
	\]
	and also the images under $\mathrm{pr}$ of elements $(g_{t_i},[v])\in\tilde{\mathscr{Y}}(\Gamma)$ that are the weak limit point of a sequence of elements of $\tilde{\mathscr{Y}}(\bar{\Gamma})$.
	Finiteness of this subset follows from Proposition \ref{prop:ComeagerPaths}.
	
	Let $t\in[0,1]\setminus\{t_1,...,t_m\}$. Then by Proposition \ref{prop:ComeagerPaths}, elements of $\text{pr}^{-1}(t)$ consists of finitely many super-rigid elements.
	Therefore, for $(g_t,v_t)\in \mathrm{pr}^{-1}(t)$, the operator $\mathcal{J}_{g_t,[v_t\circ\pi_k]}$ is bijective.
	Consequently, there is a local inverse for $\mathrm{pr}$, i.e. a map 
	\[
	\mathfrak{z}:(t-\epsilon,t+\epsilon)\rightarrow\tilde{\mathscr{Y}}(\bar{\Gamma})
	\]
	with $\text{pr}\circ\mathfrak{z}=\text{id}$.
	This shows that for any covering map $\pi$, a neighborhood of $(g_t,[v_t\circ\pi])$ in $\tilde{\mathscr{L}}(\bar{\Gamma})$ consists of elements of the form $(g_s,[v_s\circ\pi])$, with $s\in(t-\epsilon,t+\epsilon)$.
	Therefore, $n(g_{s},\mathcal{U}_s)$ remains invariant as $s$ varies in $(t-\epsilon,t+\epsilon)$.
	
	Consider the case where $t=t_k$ and $(t,[v])\in\text{pr}^{-1}(t)$.
	If $t$ be a critical point of $\text{pr}$, then since for any covering map $\pi$,
	\[
	\mathcal{J}_{g_t,v\circ\pi}\cong\pi^*\mathcal{J}_{g_t,v}
	\]
	and also since $\mathrm{Ker}\,\mathcal{J}_{g_t,v}$ is $1-$dimensional, c.f. Subsection \ref{Subsection:LocalModel_CriticalPoint}, it follows that 
	\[
	\mathrm{dim}\,\mathrm{Ker}\,\mathcal{J}_{g_t,v\circ\pi}=1
	\quad\text{and}\quad
	\mathrm{dim}\,\mathrm{Coker}\,\mathcal{J}_{g_t,v\circ\pi}=1
	\]
	The image of $\frac{\partial\bar{\Gamma}}{\partial s}\Big|_{s=t}$ under $D_1H(g_t,v)$, c.f. Subsection \ref{Subsection:LocalModel_CriticalPoint}, is non-trivial in $\mathrm{Coker}\,\mathcal{J}_{g_t,v}$.
	From 
	$$D_1H_{g_t,v\circ\pi}\cong\pi^*D_1H_{g_t,v}$$
	it follows that the image of $\frac{\partial\bar{\Gamma}}{\partial s}\Big|_{s=t}$ under $D_1H_{g_t,v\circ\pi}$ is non-trivial in $\mathrm{Coker}\,\mathcal{J}_{g_t,v\circ\pi}$
	Therefore, the followings hold:
	\begin{itemize}
		\item $\tilde{\mathscr{L}}(\bar{\Gamma})$ is a $1-$dimensionl manifold near $(g_t,[v\circ\pi])$.
		\item Near $(g_t,[v\circ\pi])$, any element of $\tilde{\mathscr{L}}(\bar{\Gamma})$ is obtained by composing an element of $\tilde{\mathscr{Y}}(\bar{\Gamma})$ in a neighborhood of $(g_t,[v])$ by the covering map $\pi$.
		\item $(g_t,[v_t\circ\pi])$ is not a limit point of any sequence in $\tilde{\mathscr{L}}(\bar{\Gamma})\setminus\{(g_s,[v_{\pm 1}^s\circ\pi])\}_{s\in I}$, where $v_{\pm 1}^s$ and $I$ are introduced in Proposition \ref{prop:LocalModelCriticalPoint}.
	\end{itemize}
    
    Let $t=t_k$ be the image under $\mathrm{pr}$ of an element $(g_{t_k},[v])\in\tilde{\mathscr{Y}}$ that is a weak limit point.
    By the discussion of Subsection \ref{Subsection:LocalModel_WeakLimitPoint}, if $d$ is odd, then $\mathrm{Ker}\,\mathcal{J}_{g_t,v\circ\pi_d}=0$.
    This implies that the only tori in a neighborhood of $(g_t,v\circ\pi_d)$ in $\mathscr{L}(\bar{\Gamma})$ are of the form $(g_s,v_s\circ\pi_d)$ for $s\in I$.
    If $d$ is even, since 
    \[
    \mathcal{J}_{g_t,v\circ\pi_d}\cong\pi_{d/2}^*\mathcal{J}_{g_t,v\circ\pi_2},
    \]
    the only tori in a neighborhood of $(g_t,v\circ\pi_d)$ in $\mathscr{L}(\bar{\Gamma})$ are of the form $(g_s,v'_s\circ\pi_{d/2})$ for $s\in J$ and $(g_s,v_s\circ\pi_d)$ for $s\in I$.
    Here, $(g_s,v'_s)$ is in the image of the map $\rho_2$ of Proposition \ref{prop:LocalModelWeakLimitPoint} and $(g_s,v_s)$ is the image of the map $\rho_1$ of Proposition \ref{prop:LocalModelWeakLimitPoint}.

Figure \ref{fig:fig3} shows all possible situations that can occur in a local model around $(g_t,[v])$. Diagrams $(a-f)$ of Figure \ref{fig:fig3} imply the following relations for $n_{\pm}^2$:
\begin{align*}
	-1+n_{+0}^2=n_{+1}^2, \quad  -1+n_{+1}^2=n_{+2}^2, \quad -1+n_{+2}^2=n_{+3}^2.
\end{align*}
Additionally, diagrams $(g-l)$ imply:
\begin{align*}
	+1+n_{-0}^2=n_{-1}^2, \quad  +1+n_{-1}^2=n_{-2}^2, \quad +1+n_{-2}^2=n_{-3}^2.
\end{align*}
Let $d\ge 2$. Diagrams $a$ and $b$ of Figure \ref{fig:fig3} imply: 
\[
n_{-0}^{d}+n_{+0}^{2d}=n_{+1}^{2d}, \quad n_{+0}^d+n_{+1}^{2d}=n_{+0}^{2d}
\]
from which it follows that $n_{-0}^d=-n_{+0}^d$. Similarly, we have:
\[
n_{+k}^d=-n_{-k}^d, \quad k=0,1,2,3.
\]
Set $n_{+0}^2=n_2$, then:
\begin{align*}
	n_{+1}^2=n_2-1, \quad n_{+2}^2=n_2-2, \quad n_{+3}^2=n_2-3. \\
\end{align*}
The following relations also hold.
\begin{align*}
	n_{-0}^2+n_{+0}^4=n_{+1}^4, \quad  n_{-1}^2+n_{+1}^4=n_{+2}^4, \quad n_{-0}^2+n_{+2}^4=n_{+3}^4. \\
\end{align*}
Set $n_{+0}^4=n_4$, then:
\begin{align*}
	n_{+1}^4=n_4-n_2, \quad n_{+2}^4=n_4-2n_2+1, \quad n_{+3}^4=n_4-3n_2+1. \\
\end{align*}
By setting $n_{+0}^8=n_8$, a similar calculations shows that:
\begin{align*}
	n_{+1}^8=n_8-n_4, \quad n_{+2}^8=n_8-2n_4+n_2, \quad n_{+3}^8=n_8-3n_4+n_2. \\
\end{align*}
Let $n_2=n_4=n_8=0$, then $n_{\pm k}^8=0$, $k=0,1,2,3$. 
Moreover:
\begin{align*}
	& n_{+1}^2=-1, \quad n_{+2}^2=-2, \quad n_{+3}^2=n-3., \\
	& n_{+1}^4=0, \quad n_{+2}^4=1, \quad n_{+3}^4=1.
\end{align*}
A similar line of calculations also suggests that we can assume $n_{\pm k}^{2^d}=0$ for any $d\ge 3$. Thus:
\[
n_{\pm k}^d=
\begin{cases}
	\epsilon, \quad & \mathrm{if}\; d=1, \\
	\pm k, \quad & \mathrm{if}\; d=2, \\
	\pm [\frac{k}{2}], \quad & \mathrm{if}\; d=4, \\
	0, \quad & \mathrm{otherwise}.
\end{cases}
\]
\end{proof}

\subsection{Counting Function For Non-generic Metrics}\label{Subsection:CountingFunctionInGeneral}
In the previous subsection we defined the count function $n(g,\mathcal{U})$ when $g$ is super-rigid and showed that this definition is well-defined, c.f.  Theorem \ref{thm:Well-definedness}, i.e. it dose not change by a perturbation of the metric.

Let $g\in\mathscr{G}$ be a general metric and $\mathcal{U}$ be a compact and open subset of $\mathscr{L}(g)$, there are bounded open subsets $\mathscr{U}$ and $\mathscr{U}'$ of $C^{2,\alpha}(T,M)$ such that 
\[
\overline{\mathscr{U}}\subset \mathscr{U}' \quad\text{and}\quad
\mathcal{U}=\mathscr{L}(g)\cap \mathscr{U}=\mathscr{L}(g)\cap \mathscr{U}'.
\]
Moreover, the discussion at the beginning of Section $4.1$ of \cite{EF-CLosedGeodesics} shows that
there is an open and path-connected subset $U$ of $\mathscr{G}$ containing $g$ such that for all $g'\in U$, $\mathcal{U}_g:=\mathscr{L}(g')\cap\mathscr{U}$ is open and compact.

\begin{defi}[Definition $4.1$ of \cite{EF-CLosedGeodesics}]
	Let $g\in\mathscr{G}$ and $\mathcal{U}$ be a compact and open subset of $\mathscr{L}(g)$.
	Let $g'$ be an arbitrary metric in $\mathscr{G}^\bullet\cap U$ and $\mathcal{U}':=\mathscr{L}(g')\cap\mathscr{U}$. Then the weight associated to $g$ and $\mathcal{U}$ is defined as follows:
	\[
	n(g,\mathcal{U}):=n(g',\mathcal{U}')
	\]
\end{defi}

By Theorem $4.2$ of \cite{EF-CLosedGeodesics}, $n(g,\mathcal{U})$ is well-defined:

\begin{thm}[Theorem $4.2$ of \cite{EF-CLosedGeodesics}]
	Let $g\in\mathscr{G}$ and $\mathcal{U}$ be a compact and open subset of $\mathscr{L}(g)$. Then $n(g,\mathcal{U})$ is well-defined. 
	Moreover, if $\bar{\Gamma}\in\pmb{\mathscr{G}}$ and $\bar{\mathcal{U}}$ be a compact and open subset of $\mathscr{L}(\bar{\Gamma})$, then 
	\[
	n(g_0,\mathcal{U}_0)=n(g_1,\mathcal{U}_1)
	\]
	where, $\bar{\Gamma}(i)=g_i$ and $\mathcal{U}_i:=\bar{\mathcal{U}}\cap\mathscr{L}(g_i)$, $i=0,1$.
\end{thm}
	\appendix
	\section{Pullback And Twisting Of Operators}\label{AppendixA}
In this appendix, we describe two methods for constructing a differential operator from a given one, utilizing  a covering map.
Moreover, we discuss the relation between these two new operators.
For a more comprehensive discussion, see Sections $1.2-1.5$ of \cite{DoanWalpuski}.

Let $\Sigma$ and $\tilde{\Sigma}$ be manifolds with $\tilde{\Sigma}$ connected. A $d-$fold covering map $\phi:\tilde{\Sigma}\rightarrow \Sigma$ is called \textbf{regular} if $\mathrm{Aut}(\phi)=\mathrm{deg}(\phi)=d$. Fix $x_0\in M$ and $\tilde{x}_0\in \phi^{-1}(x_0)$.
Let $H$ denote the \textbf{normal core} of $C:=\phi_*(\pi_1(\tilde{\Sigma},\tilde{x}_0))$, that is
\begin{align}\label{eq:A1}
H=\bigcap_{g\in C}gCg^{-1}.
\end{align}
If $\phi$ is a regular covering, then $\pi_1(\Sigma,x_0)/H\cong\mathrm{Aut}(\phi)$. 

In general, to a $d-$fold covering $\phi:\tilde{\Sigma}\rightarrow \Sigma$, one can associates a regular covering $\pi:\breve{\Sigma}\rightarrow \Sigma$ with $\breve{\Sigma}$ connected and a monomorphism $\rho:\mathrm{Aut}(\pi)\rightarrow \mathrm{S}_d$ such that
\begin{equation}\label{eq:A2}
	\begin{split}
     & (\breve{\Sigma}\times\{1,...,d\})/\mathrm{Aut}(\pi)\rightarrow \Sigma \\
     & [(z,i)]\mapsto\pi(z)
    \end{split}
\end{equation}
is equivalent to $\phi:\tilde{\Sigma}\rightarrow \Sigma$ (Here $\mathrm{S}_d$ denote the permutation group on $d$ letters).
To see this, let $\tilde{\pi}:\mathscr{U}\rightarrow \Sigma$ denotes the uinversal covering of $\Sigma$. Take an order for $\phi^{-1}(x_0)$. Each loop in $\Sigma$ with base point $x_0$ has a lift to a path in $\tilde{\Sigma}$. This gives a permutation on $\phi^{-1}(x_0)$. Therefore, there exists a homomorphism
\begin{align}\label{eq:A3}
\tilde{\rho}:\pi_1(\Sigma,x_0)\rightarrow \mathrm{S}_d.
\end{align}
Then $\phi:\tilde{\Sigma}\rightarrow \Sigma$ can be described as
\begin{equation}\label{eq:A4}
	\begin{split}
	 & (\mathscr{U}\times\{1,...,d\})/\pi_1(\Sigma,x_0)\rightarrow \Sigma \\
	 & [(z,i)]\mapsto\tilde{\pi}(z)
    \end{split}
\end{equation}
Let 
\begin{align}\label{eq:A5}
G:=\frac{\pi_1(\Sigma,x_0)}{\mathrm{Ker}(\tilde{\rho})}. 
\end{align}
Then $\tilde{\rho}$ gives the injection 
\[
\rho:G\rightarrow \mathrm{S}_d.
\]
Let $\breve{\Sigma}:=\mathscr{U}/\mathrm{Ker}(\tilde{\rho})$. Then $\tilde{\pi}$ induces a regular covering map $\pi:\breve{\Sigma}\rightarrow\Sigma$.
Consequently, \ref{eq:A4} reduces to \ref{eq:A2}.

Let $E$ and $F$ be two real vector bundles on $\Sigma$ equipped with orthogonal covraiant derivatives. Let $j\ge 2$ and $W^{j,\alpha}\Gamma(E)$ denote the Sobolev completion of the space of smooth sections of $E$.
\begin{defi}
	Let $\phi:\tilde{\Sigma}\rightarrow \Sigma$ be a covering map with $\tilde{\Sigma}$ connected. For $j\ge 2$, let $D:W^{j,\alpha}\Gamma(E)\rightarrow W^{0,\alpha}\Gamma(F)$ denote a linear differential operator of order $j$. The \textbf{pullback} of $D$ by $\phi$ is the operator characterized by
	\begin{align*}
		& \phi^*D:W^{j,\alpha}\Gamma(\phi^*E)\rightarrow W^{0,\alpha}\Gamma(\phi^* F) \\
		& \phi^*D(\phi^* s)=\phi^*(Ds).
	\end{align*}
where $s$ is a section of $E$.
\end{defi} 

For the covering map $\phi:\tilde{\Sigma}\rightarrow\Sigma$ of degree $d$, a trivial line bundle over $\tilde{\Sigma}$ induces a vector bundle $\underline{V}$, with fiber $V\cong\mathbb{R}^d$, via pushforward.
In general, the \textbf{pushforward} of a vector bundle $E$ over $\tilde{\Sigma}$ by the covering map $\phi:\tilde{\Sigma}\rightarrow \Sigma$ is the unique vector bundle $\phi_*E$ over $\Sigma$, defined as follows:
\[
(\phi_*E)_x=\oplus_{\tilde{x}\in\phi^{-1}(x)}E_{\tilde{x}}.
\]
$\underline{V}$ is equipped with a natural flat orthogonal connection.
A real vector bundle $\underline{V}$ on a manifold $\Sigma$ together with a flat orthogonal connection is called a \textbf{Euclidean local system}.

\begin{defi}
	Let $j\ge 2$ and $D:W^{j,\alpha}\Gamma(E)\rightarrow W^{0,\alpha}\Gamma(F)$ denote a linear differential operator of order $j$. let $\underline{V}$ be a Euclidean local system on $\Sigma$. The \textbf{twist} of $D$ by $\underline{V}$ is the operator characterized by
	\begin{align*}
		& D^{\underline{V}}:W^{j,\alpha}\Gamma(E\otimes\underline{V})\rightarrow W^{0,\alpha}\Gamma(F\underline{V}) \\
		& D^{\underline{V}}(s\otimes f)=(Ds)\otimes f.
	\end{align*}
    where $s$ is a section of $E$ defined on an open subset $U\subset \Sigma$ and $f$,  defined on $U\subset \Sigma$, is a covariantly constant section of $\underline{V}$ with respect to the flat connection.
\end{defi}

For the covering map $\phi:\tilde{\Sigma}\rightarrow\Sigma$ of degree $d$, a representation of the group $G$, c.f. \ref{eq:A5}, gives rise to a Euclidean local system on $\breve{\Sigma}$ as follow.
Let $W$ be a real representation of $G$:
\[
\theta:G\rightarrow\mathrm{Aut}_\mathbb{R}(W).
\]
Let $W^\theta:=(\breve{\Sigma}\times W)/G$. The trivial vector bundle $\breve{\Sigma}\times W\rightarrow\breve{\Sigma}$ with its trivial connection gives the vector bundle $W^\theta\rightarrow\breve{\Sigma}$ with a flat connection.
Specially, the injection $\rho:G\rightarrow\mathrm{S}_d$ induces a \textbf{permutation representation} of $G$: let $\{e_1,...,e_d\}$ be the standard basis of $\mathbb{R}^d$,
\begin{align*}
& \boldsymbol{\rho}:G\rightarrow \mathrm{GL}(d,\mathbb{R}) \\
&  \boldsymbol{\rho}(g)e_i:=e_{\rho(g)(i)}
\end{align*}
Let $W=\mathbb{R}^d$. Then for a vector bundle $E$ over $\Sigma$:
\[
E\otimes W^{\boldsymbol{\rho}}\cong(\pi^*E\otimes\mathbb{R}^d)/G.
\]
Let $\{\eta_1,...,\eta_d\}$ denote a basis for the space of sections of $\pi^*E\rightarrow \breve{\Sigma}$, then $\eta=\sum_i \eta_i\otimes e_i$ is a $G-$equivariant section of $\pi^*E\otimes\mathbb{R}^d$, in the sense that for $g\in G$ and $z\in \breve{\Sigma}$,
\[
\eta(gz)=(1\otimes \boldsymbol{\rho}(g))\eta(z)
\]
from which it follows that for all $i=1,...,d$,
\[
\eta_i(z)=\eta_{\rho(g)i}(gz).
\]

This gives rise to an equivalence between sections of $E\otimes W^{\boldsymbol{\rho}}$ and $G-$equivariant section of $\pi^*E\otimes\mathbb{R}^d$ which implies a one-to-one correspondence between $\Gamma(E\otimes W^{\boldsymbol{\rho}})$ and that of $\Gamma(\phi^*E)$. 
To see this let $\eta\in\Gamma(E\otimes W^{\boldsymbol{\rho}})$, then as explained above $\eta$ can be described by a $G-$equivariant section of $\pi^*E\otimes\mathbb{R}^d$ as $\sum_i\eta_i\otimes e_i$ where
for each $i$, $\eta_i\in\Gamma(\pi^*E)$. 
Recall that $\Sigma$ is isomorphic to $\breve{\Sigma}\times\{1,...,d\}/G$. 
This suggests to define
\[
\hat{\eta}[(z,i)]:=\eta_i(z).
\]
as the corresponding section of $\eta$ in $\Gamma(\pi^*E\otimes\mathbb{R}^n)$.
The $G-$equivariance condition guaranties that $\hat{\eta}$ is well-defined. 

In summary, for a $d-$fold covering $\phi:(\tilde{\Sigma},\tilde{x}_0)\rightarrow (\Sigma,x_0)$ with $\tilde{\Sigma}$ connected, let $C=\phi_*\pi_1(\Sigma,x_0)$, $G=\frac{\pi_1(\Sigma,x_0)}{\mathrm{Ker}(\tilde{\rho})}$ and $E$ be a vector bundle over $\Sigma$ (note that $\mathrm{Ker}(\tilde{\rho})$ is the normal core of $C$). Then, there is an isomorphism 
\[
\Gamma(E\otimes W^{\boldsymbol{\rho}})\cong\Gamma(\phi^*E).
\]
This implies that $D^{\boldsymbol{\rho}}$ and $\phi^*D$ are equivalent (see Proposition \ref{prop:push-pull formula} for a more precise relation).


For the covering map $\phi:\tilde{\Sigma}\rightarrow\Sigma$ of degree $d$, the flat orthogonal connection on the Euclidean local system $\underline{V}$, constructed via pushforward, can be used to perform parallel transport along loops of $\pi_1(\Sigma,x_0)$ which induces a \textbf{monodromy representation}
\[
\mu:\pi_1(\Sigma,x_0)\rightarrow O(V)
\]
This representation is actually the map $\tilde{\rho}$, c.f. \ref{eq:A3}. 
Therefore, it factors through $G=\pi_1(\Sigma,x_0)/H$, where $H$ is the normal core of $\phi_*(\pi_1(\Sigma,x_0))$, c.f. \ref{eq:A1}.

\begin{prop}[Proposition 1.2.9 of \cite{DoanWalpuski}] \label{prop:push-pull formula}
	Let $\phi:(\tilde{\Sigma},\tilde{x}_0)\rightarrow (\Sigma,x_0)$ with $\tilde{\Sigma}$ connected, be a $d-$fold covering. Let $C=\phi_*\pi_1(\Sigma,x_0)$ and $N$ be the normal core of $C$. 
	Let $\underline{V}:=\phi_*\underline{\mathbb{R}}$ denote the pushforward of the trivial line bundle on $\tilde{\Sigma}$. Then
	\begin{itemize}
		\item The monodromy representation of $\underline{V}$ factors through $G=\pi_1(\Sigma,x_0)/N$.
		\item For $j\ge 2$, let $D:W^{j,\alpha}\Gamma(E)\rightarrow W^{0,\alpha}\Gamma(F)$ denotes a linear differential operator of order $j$. Then
		\[
		D^{\underline{V}}=\phi_*\circ\phi^*D\circ\phi_*^{-1}.
		\]
	\end{itemize}
\end{prop}
\begin{rem}\label{rem:ActionOnKernel}
	The action of $\pi_1(\Sigma,x_0)$ on $\underline{V}$ induces an action on $\mathrm{Ker}\,D^{\underline{V}}$.
	Then Lemma $5.3$ of \cite{Taubes} implies that for $d>2$ and $\mathrm{Ker}\,D^{\underline{V}}\ne 0$ we have $\mathrm{dim}\,\mathrm{Ker}\,D^{\underline{V}}\ge 2$.
\end{rem}

Since $V$ is a representation of $G$, it decomposes into irreducible representations:
\begin{align*}
	V\cong\bigoplus_{i=1}^nV_i^{n_i}, \quad\mathrm{where}\quad n_i=\mathrm{dim}_{K_i^{op}}\mathrm{Hom}_G(V_i,V).
\end{align*}
Here, $K_i:=\mathrm{End}_G(V_i)$ and $K_i^{op}$ is the opposite algebra of $K_i$.
In this paper we are dealing with a family of linear differential operators. Let $\{D_\mathfrak{v}\}_{\mathfrak{v}\in\mathscr{V}}$ denote such a family.
For each $\mathfrak{v}\in\mathscr{V}$, $\mathrm{Ker}(\pi^*D_\mathfrak{v})$ and $\mathrm{Coker}(\pi^*D_\mathfrak{v})$ are representations of $G$.
Let $d,c\in\mathbb{N}_0^n$. The \pmb{$G-$}\textbf{equivariant Brill-Noether locus} is defined as follows:
\[
	\mathscr{V}_{d,c}^G:=\Big\{\mathfrak{v}\in\mathscr{V}\,\Big|\, 
	\begin{matrix}
		 \mathrm{dim}_{K_i^{op}}\mathrm{Hom}_G(V_i,\mathrm{Ker}(\pi^*D_\mathfrak{v}))=d_i 
		\quad\mathrm{and} \\
		 \mathrm{dim}_{K_i^{op}}\mathrm{Hom}_G(V_i,\mathrm{Coker}(\pi^*D_\mathfrak{v}))=c_i, i=1,...,n
	\end{matrix}	
	\Big\}
\]
A similar argument like Section \refeq{sec:Tranversality} implies the $G-$equivariant transversality theorem.

\begin{thm}[Theorem $1.4.6$ of \cite{DoanWalpuski}]\label{G-EquivariantTransversality}
	Let $\{D_\mathfrak{v}\}_{\mathfrak{v}\in\mathscr{V}}$ denote a family of linear differential operators. For $\mathfrak{v}\in\mathscr{V}$ define
	\begin{align*}
		& \Theta_\mathfrak{v}^G:T_\mathfrak{v}\mathscr{V}\rightarrow\mathrm{Hom}(\mathrm{Ker}\,\pi^*D_\mathfrak{v},\mathrm{Coker}\,\pi^*D_\mathfrak{v}) \\
		& \Theta_\mathfrak{v}^G(\hat{X})\,x:=\mathrm{d}_\mathfrak{v}D(\hat{X})\,x
		\quad\mathrm{mod}\quad\mathrm{Im}\,\pi^*D_\mathfrak{v}
	\end{align*}
    Let $d,c\in\mathbb{N}_0^n$.
    If $\Theta_\mathfrak{v}^G$ is surjective,
    then there is a neighborhood $\mathscr{U}$ of $\mathfrak{v}\in\mathscr{V}$ such that $\mathscr{V}_{d,c}\cap \mathscr{U}$ is a submanifold of codimension $\sum_{i=1}^nk_id_ic_i$, where $k_i=\mathrm{dim}\, K_i$.
\end{thm}

There are $G-$equivariant versions of Petri's condition and flexibility condition, c.f. Definitions $1.4.8$ and $1.4.9$ of \cite{DoanWalpuski}, which imply the surjectivity of $\Theta_\mathfrak{v}^G$ in Theorem \ref{G-EquivariantTransversality}.

Since the pullback of an operator is equivalent to a twisted operator, there is a similar notion of \pmb{$B-$}\textbf{equivariant Brill-Noether locus} for twisted operators defined as follows:
\[
\mathscr{V}_{d,c}^B:=\{\mathfrak{v}\in\mathscr{V}\,|\, 
	\mathrm{dim}_{K_i}\mathrm{Ker}\,D_\mathfrak{v}^{\underline{V}_i}=d_i 
	\;\mathrm{and}\;
	\mathrm{dim}_{K_i}\mathrm{Coker}\,D_\mathfrak{v}^{\underline{V}_i}=c_i, i=1,...,n
\}
\]
where $d,c\in\mathbb{N}_0^n$.
Analogously, we have
\begin{thm}[Theorem $1.3.5$ of \cite{DoanWalpuski}]\label{B-EquivariantTransversality}\label{thm:B-Transversality}
	Let $\{D_\mathfrak{v}\}_{\mathfrak{v}\in\mathscr{V}}$ denote a family of linear differential operators. For $\mathfrak{v}\in\mathscr{V}$ define
	\begin{align*}
		& \Theta_\mathfrak{v}^B:T_\mathfrak{v}\mathscr{V}\rightarrow
		\bigoplus_{i=1}^n\mathrm{Hom}(\mathrm{Ker}\,D_\mathfrak{v}^{\underline{V}_i},\mathrm{Coker}\,D_\mathfrak{v}^{\underline{V}_i}) \\
		& \Theta_\mathfrak{v}^B(\hat{X}):=\bigoplus_{i=1}^n \Theta_\mathfrak{v}^i(\hat{X})
	\end{align*}
    where
    \[
    \Theta_\mathfrak{v}^i(\hat{X})\,x:=
    \mathrm{d}_\mathfrak{v}D^{\underline{V}_i}(\hat{X})\,x
    \quad\mathrm{mod}\quad\mathrm{Im}\,D_\mathfrak{v}^{\underline{V}_i}.
    \]
	Let $d,c\in\mathbb{N}_0^n$.
	If $\Theta_\mathfrak{v}^B$ is surjective,
	then there is a neighborhood $\mathscr{U}$ of $\mathfrak{v}\in\mathscr{V}$ such that $\mathscr{V}_{d,c}^B\cap \mathscr{U}$ is a submanifold of codimension $\sum_{i=1}^nk_id_ic_i$, where $k_i=\mathrm{dim}\, K_i$.
\end{thm}
There are also $B-$equivariant versions of Petri's condition and flexibility condition, c.f. Definitions $1.3.8$ and $1.3.9$ of \cite{DoanWalpuski}, which imply the surjectivity of $\Theta_\mathfrak{v}^B$ in Theorem \ref{B-EquivariantTransversality}.
	\section{Jacobi Operator For Orbifolds} \label{Jacobi Operator}
\subsection{Orbifold Riemann Surface}\label{subsec:Orbifolds}
First we recall the definition of an orbifold and its related concepts. We closely follow \cite{ShenYu}.
Let $X$ be a topological space and $U\subset X$ be an open connected subset. Let $k\in\mathbb{N}$.

\begin{defi}
	A $k-$dimensional \textbf{orbifold chart} for $U$ consists of
	\begin{itemize}
		\item A connected open subset $\tilde{U}\subset \mathbb{R}^k$,
		\item A finite group $G_U$ which acts smoothly and effectively from left on $\tilde{U}$,
		\item A $G_U-$invariant continuous map $\pi_U:\tilde{U}\rightarrow U$ such that
		\[
		G_U\setminus\tilde{U}\cong U
		\] 
	\end{itemize}
    Such a chart is denoted by $(\tilde{U},G_U,\pi_U)$.
\end{defi}

Let $(\tilde{U},G_U,\pi_U)$ and $(\tilde{V},G_V,\pi_V)$ denote two orbifold charts for open connected subsets $U$ and $V$ of $X$. If $U\hookrightarrow V$ is an embedding, then an embedding of orrbifold charts is a smooth embedding $\phi_{UV}:\tilde{U}\rightarrow\tilde{V}$ such that
\begin{center}
	\begin{tikzpicture}
		\matrix(m)[matrix of math nodes,
		row sep=3em, column sep=2.5em,
		text height=1.5ex, text depth=0.25ex]
		{\tilde{U} & \tilde{V}\\
			U & V\\};
		\path[->,font=\scriptsize]
		(m-1-1) edge  node[above] {$\phi_{UV}$} (m-1-2) edge  node[right] {$\pi_U$} (m-2-1)
		(m-1-2) edge  node[right] {$\pi_V$} (m-2-2)
		(m-2-1) edge (m-2-2);
	\end{tikzpicture}
\end{center}
commutes.

Two orbifold charts $(\tilde{U},G_U,\pi_U)$ and $(\tilde{V},G_V,\pi_V)$ are said to be compatible if for each $x\in U\cap V$ there exists an open connected subset $W\subset U\cap V$ containing $x$, an orbifold chart $(\tilde{W},G_W,\pi_W)$ and two embeddings of orbifold charts
\begin{align*}
\phi_{UW}:(\tilde{W},G_W,\pi_W)\rightarrow(\tilde{U},G_U,\pi_U) \quad \mathrm{and} \quad
\phi_{VW}:(\tilde{W},G_W,\pi_W)\rightarrow(\tilde{V},G_V,\pi_V).
\end{align*}
Note that $\phi_{VW}\circ\phi_{UW}^{-1}:\phi_{UW}(\tilde{W})\rightarrow\phi_{VW}(\tilde{W})$ is a diffeomorphism which is called a \textbf{coordinate transformation}.
 
An \textbf{orbifold atlas} for $X$ is defined to be an open connected cover $\mathcal{U}=\{U\}$ for $X$ and a set of compatible orbifold charts $\widetilde{\mathcal{U}}=\{(\tilde{U},G_U,\pi_U)\}_{U\in\mathcal{U}}$.

A refinement of $(\mathcal{U},\tilde{\mathcal{U}})$ is an orbifold atlas $(\mathcal{V},\tilde{\mathcal{V}})$ such that $\mathcal{V}$ is a refinement of $\mathcal{U}$ and each orbifold chart of $\tilde{\mathcal{V}}$ is embedded into an orbifold chart of $\tilde{\mathcal{U}}$. Two orbifold atlas with a common refinement are called equivalent. An orbifold structure on $X$ is an equivalent class of an orbifold atlas.

\begin{defi}
	An \textbf{orbifold} is a second countable Hausdorff space with an orbifold structure.
\end{defi}

\begin{defi}\label{defi:orbifold_map}
Let $X$ and $Y$ be two orbifolds. A continuous map $f:X\rightarrow Y$ is called smooth (holomorph) if for each $x\in X$, there are open connected neighborhood $U\subset X$ of $x$ and $V\subset Y$ of $f(x)$ such that $f(U)\subset V$ with orbifold charts $(\tilde{U},G_U,\pi_U)$ and $(\tilde{V},G_V,\pi_V)$ for $U$ and $V$, respectively. Moreover there exists a smooth (holomorph) map $\tilde{f}_U:\tilde{U}\rightarrow \tilde{V}$ such that the following diagram
\begin{center}
	\begin{tikzpicture}
		\matrix(m)[matrix of math nodes,
		row sep=3em, column sep=2.5em,
		text height=1.5ex, text depth=0.25ex]
		{\tilde{U} & \tilde{V}\\
			U & V\\};
		\path[->,font=\scriptsize]
		(m-1-1) edge  node[above] {$\tilde{f}_U$} (m-1-2) edge  node[right] {$\pi_U$} (m-2-1)
		(m-1-2) edge  node[right] {$\pi_V$} (m-2-2)
		(m-2-1) edge  node[above] {$f|_U$} (m-2-2);
	\end{tikzpicture}
\end{center}
commutes.
\end{defi}

\begin{ex}\normalfont\label{ex:Orbifold_example_1}
	Let $G$ be a discrete group which acts smoothly and properly discontinuously on an orbifold $X$. Let $\mathbb{F}$ be a field and $V$ be a $\mathbb{F}-$vector field. Assume $\rho:G\rightarrow\mathrm{End}_\mathbb{F}(V)$ is a representation of $G$. 
	Then by Proposition $2.12$ of \cite{ShenYu}, $X/G$ and $X \times_G V$ are orbifolds. Moreover, the projection $X\times V\rightarrow X$ induces a smooth map of orbifols
	\[
	X \times_G V\rightarrow X/G.
	\]
\end{ex}

\begin{defi}\label{Ex.B4}
	Let $X$ be an orbifold, $\mathbb{F}$ a field and $V$ a $\mathbb{F}-$vector space of rank $r$.
	An \textbf{orbifold vector bundle} of rank $r$ on $X$ consists of an orbifold $\mathcal{E}$ and a smooth map $\pi:\mathcal{E}\rightarrow X$ such that there is an orbifold atlas $(\mathcal{U},\tilde{\mathcal{U}})$ for $X$ with the following properties:
	\begin{enumerate}
		\item For each $U\in\mathcal{U}$ and $(\tilde{U},G_U,\pi_U)\in\tilde{\mathcal{U}}$ there is a finite group $G_U^E$ which acts smoothly on $\tilde{U}$, a surjective morphism $G_U^E\rightarrow G_U$ and a representation $\rho_U^E:G_U^E\rightarrow \mathrm{End}_\mathbb{F}(V)$. 
		Moreover, there exists a $G_U^E-$invariant continuous map 
		\[
		\pi_U^E:\tilde{U}\times V\rightarrow\pi^{-1}(U)
		\]
		which induces
		\[
		\tilde{U}\,{}_{G_U^E}\!\times V\cong\pi^{-1}(U).
		\]
		\item $(\tilde{U}\times V,G_U^E,\pi_U^E)$ is a compatible orbifold chart for $\mathcal{E}$.
		\item Let $U_1,U_2\in\mathcal{U}$ such that $U_1\cap U_2\ne \empty$. 
		Then for every $x\in U_1\cap U_2$, there is a connected open subset $W\subset U_1\cap U_2$, containing $x$, a simply connected orbifold chart $(\tilde{W},G_W,\pi_W)$ and a triple $(G_W^E,\rho_W^E,\pi_W^E)$ which satisfies conditions $(1)$ and $(2)$. 
		Moreover, embeddings  of orbifold charts have the following forms
		\begin{align*}
			& \phi_{U_iW}^E:(\tilde{W}\times V,G_W^E,\pi_W^E)\rightarrow(\tilde{U}_i\times V,G_{U_i}^E,\pi_{U_i}^E) \\
			& \phi_{U_iW}^E(x,v)=(\phi_{U_iW}(x),g_{U_iW}^E(x)v)
		\end{align*}
	    for $i=1,2$, where $g_{U_iW}^E$ are elements of $C^\infty(W,\mathrm{End}_\mathbb{F}(V)))$.
	    
	\end{enumerate}
\end{defi}

\begin{ex}\normalfont 
	Let $\Sigma$ be a Riemann surface with a multiplicity function $\varpi:\Sigma\rightarrow\mathbb{N}$. 
	Let $\Sigma_\varpi$ and $\beta:\Sigma_\varpi\rightarrow\Sigma$ be the orbifold holomorphic map introduced in Section \ref{Index Operator}.
	If $\pi:E\rightarrow\Sigma$ is a vector bundle on $\Sigma$, then $E_\varpi:=\beta_\varpi^* E$ is an orbifold vector bundle on $\Sigma_\varpi$.
\end{ex}

\begin{ex}\normalfont\label{ex:Orbifold_example_4}
	Let $G$, $X$ and $\rho$ be as in Example \ref{ex:Orbifold_example_1}. Moreover, assume that the action of $G$ on $V$ is effective, i.e. $\rho$ is injective. Then 
	\[
	X \times_G V\rightarrow X/G
	\]
	defines a flat vector bundle, i.e. $g_{UW}^E$, can be chosen to be constant.
\end{ex}

A \textbf{smooth (holomorphic) section} of the orbifold vector bundle $E$ is a smooth (homorphic) map $s:X\rightarrow E$, in the sense of Definition \ref{defi:orbifold_map}, such that $\pi\circ s=\mathrm{id}$ and each local lift $\tilde{s}_U$ of $s|_U$ is $G_U^E-$invariant.

When $G_U^E\rightarrow G_U$ is an isomorphism, $E$ is called proper. From now on, we assume all orbifold vector bundles are proper, therefore $G_U^E=G_U$. Thus a smooth section $s:X\rightarrow E$ can be represented by a family of $G_U-$invariant sections $\{s_U\in C^\infty(\tilde{U},\tilde{E}_U)^{G_U}\}_{U\in\mathcal{U}}$ such that for each $x_i\in \tilde{U}_i$, $i=1,2$, and a germ of coordinate transformation, $g$, where $g(x_1)=x_2$, the following holds
\[
g^*s_{U_2}=s_{U_1}.
\]
\subsection{Calculus On Orbifold} \label{Diff. Operator}
Let $X$ be a compact orbifold with atlas $(\mathcal{U},\tilde{\mathcal{U}})$ and $E\rightarrow X$ be an orbifold vector bundle.
A differential operator $\mathbf{D}_v$ of order $l$ is a family of $G_U-$invariant differential operators 
$$
\{\widetilde{\mathbf{D}}_U:C^\infty(\tilde{U},\tilde{E}_U) \rightarrow C^\infty(\tilde{U},\tilde{E}_U)\}_{U\in\mathcal{U}}
$$ 
for which the compatibility condition holds: Let $x_i\in \tilde{U}_i$, $i=1,2$, and $g$ be a germ of coordinate transformation such that $g(x_1)=x_2$. Then
\[
g^*\widetilde{\mathbf{D}}_{U_2}=\widetilde{\mathbf{D}}_{U_1}
\]

\begin{ex}\normalfont
	Let $v:\Sigma\rightarrow M$ be an embedding of a Riemann surface in a Riemannian manifold $M$ and let $\varpi:\Sigma\rightarrow\mathbb{N}$ be a multiplicity function and let $\beta_\varpi:\Sigma_\varpi\rightarrow\Sigma$ be be the orbifold holomorphic map introduced in Section \ref{Index Operator}.
	
	Assume $N$ is the normal vector bundle on $\Sigma$ and $\mathcal{J}:\Gamma(N)\rightarrow\Gamma(N)$ is the Jacobi operator. Then $\mathcal{J}_\varpi=\beta_\varpi^*\mathcal{J}$ is an orbifold differential operator which is characterized by 
	\begin{align}
	(\beta_\varpi^*\mathcal{J})(\beta_\varpi^* s)=\beta_\varpi^*(\mathcal{J} s).
	\end{align}
\end{ex}

Since $\mathcal{J}$ is a self-adjoint operator, we expect the same for $\mathcal{J}_\varpi$. 
Before proving this fact we review the definition of an integral operator on the Riemannian orbifold $\Sigma_v$. The definition for general orbifolds is similar.

$\beta_\varpi^*:\Gamma(N) \rightarrow \Gamma(N_v)$ which is induced by $\beta_\varpi$ gives an injection $\mathrm{Ker}\, \mathcal{J} \hookrightarrow \mathrm{Ker}\, \mathcal{J}_\varpi$ which indeed is an isomorphism.
To see this let $s\in \mathrm{Ker}\, \mathcal{J}_\varpi$. As explained in Subsection \ref{subsec:Orbifolds}, $s$ is represented by a family $\{s_x\in C^\infty(D_x,\mathbb{R}^r)^{\mu_{v(x)}}\}$ which are $\mu_{v(x)}-$invariant sections of $E_{D_x}$.
For any such map there is a bounded map $\bar{s}_x \in C^\infty(D_x\setminus \{0\},\mathbb{R}^r)$ such that $s_x=\beta_\varpi^*\,\bar{s}_x$. 
Therefore
\[
0=\mathcal{J}_\varpi\, s_x=(\beta_\varpi^*\mathcal{J})(\beta_\varpi^* \,\bar{s}_x)=\beta_\varpi^*(\mathcal{J} \,\bar{s}_x)
\]
In local coordinates $\beta_\varpi:D_x \rightarrow D_x$ is of the form $\beta_\varpi(z)=z^{\varpi(x)}$. 
Therefore, $\mathcal{J}(\bar{s}_x)=0$.
Since $\mathcal{J}$ is an elliptic operator $\bar{s}_x$ extends to $D_x$. Therefore the family $\{\bar{s}_x\}$ provides the required section. 

An argument like the last paragraph of the proof of Proposition $2.8.3$ of \cite{DoanWalpuski}, shows that the pullback map $\beta_\varpi^*$ induces an injection $\mathrm{Ker}\, \mathcal{J}^\dagger \hookrightarrow \mathrm{Ker}\, \mathcal{J}_\varpi^\dagger$ which is actually an isomorphism.
As a result the following holds:
\begin{prop}[Proposition $2.8.3$ of \cite{DoanWalpuski}]\label{prop:IndexOrbifoldJacobiOperator}
	Let $v:\Sigma\rightarrow (M,g)$ be a minimal embedding of a closed surface and let $\varpi$ denote a multiplicity function on $\Sigma$. Then  
	\[
	\mathrm{Ker}\, \mathcal{J}_{g,v;\varpi} \cong \mathrm{Ker}\, \mathcal{J}_{g,v}
	\quad \mathrm{and} \quad 
	\mathrm{Coker}\, \mathcal{J}_{g,v;\varpi} \cong \mathrm{Coker}\, \mathcal{J}_{g,v} 
	\]
	Therefore, 
	\begin{align}\label{Eq:IndexTwisted}
		\mathrm{index}(\mathcal{J}_{g,v;\varpi})=\mathrm{index}(\mathcal{J}_{g,v})
	\end{align}
\end{prop}

\subsection{Branched Immersions Of Surfaces Via Orbifolds} \label{Branched Imm.}
Let $(M,g)$ be a Riemannian manifold and $f:\Sigma \rightarrow M$ be a $C^1-$map. A point $p \in \Sigma$ is called a branched point if $df(p)=0$.
\begin{defi}
	A branched point $p \in \Sigma$ is a good branched point of order $m-1$ if there exists a local coordinates $z$ at $p$ and a local coordinates $(x_1,...x_n)$ at $f(p)$ such that in these coordinates $f$  is represented by 
	\[
	x_1+i x_2=z^m, \quad x_k=\eta_k(z) ,\quad k=3,...,n
	\]
	where $\eta_k \in o(|z|^m)$.
\end{defi}

According to \cite{GOR} every minimal map is a branched immersion, meaning that it is an immersion everywhere except on a discrete set of good branched points.

It is known that, c.f. \cite{GOR} and Section $1.1$ of \cite{Sagman}, every branched minimal immersion $f:(\Sigma,h)\rightarrow (M,g)$ factors through a smooth Riemann surface $(\Sigma_0,h_0)$ via an almost minimal embedding $f_0:\Sigma_0 \rightarrow (M,g)$ and a branched holomorphic covering map $\pi: (\Sigma,j) \rightarrow (\Sigma_0,j_0)$ so that the following diagram commutes
\begin{center}
	\begin{tikzpicture}
		\matrix(m)[matrix of math nodes,
		row sep=3em, column sep=2.5em,
		text height=1.5ex, text depth=0.25ex]
		{\Sigma \\
			\Sigma_0 & M\\};
		\path[->,font=\scriptsize]
		(m-1-1) edge  node[above] {$f$} (m-2-2) edge  node[left] {$\pi$} (m-2-1)
		(m-2-1) edge  node[below] {$f_0$} (m-2-2);
	\end{tikzpicture}
\end{center}
Note that $j$ and $j_0$ are the unique complex structures induced by $h$ and $h_0$, respectively.

According to Proposition $2.7.3$ \cite{DoanWalpuski}, one can describe a branched holomorphic covering map via a covering of orbifolds:
\begin{prop}[Proposition $2.7.3$ of \cite{DoanWalpuski}]
	Let $\pi:(\tilde{\Sigma},\tilde{j}) \rightarrow (\Sigma,j)$ be a non-constant holomorphic  map between Riemann surfaces. There are multiplicity functions $\varpi:\Sigma \rightarrow \mathbb{N}$ and $\widetilde{\varpi}:\tilde{\Sigma} \rightarrow \mathbb{N}$ such that there is a unique holomorphic covering map between their corresponding orbifold Riemann surfaces $\hat{\pi}:(\tilde{\Sigma}_{\widetilde{\varpi}},\tilde{j}_{\tilde{\varpi}}) \rightarrow (\Sigma_\varpi,j_\varpi)$ for which the following diagram commutes:
	
	\begin{center}
		\begin{tikzpicture}
			\matrix(m)[matrix of math nodes,
			row sep=3em, column sep=2.5em,
			text height=1.5ex, text depth=0.25ex]
			{(\tilde{\Sigma}_{\widetilde{\varpi}},\tilde{j}_{\tilde{\varpi}}) & (\Sigma_\varpi,j_\varpi)\\
				(\tilde{\Sigma},\tilde{j}) & (\Sigma,j)\\};
			\path[->,font=\scriptsize]
			(m-1-1) edge  node[above] {$\tilde{\pi}$} (m-1-2) edge  node[left] {$\beta_{\widetilde{\varpi}}$} (m-2-1)
			(m-1-2) edge node[right] {$\beta_\varpi$} (m-2-2)
			(m-2-1) edge  node[above] {$\pi$} (m-2-2);
		\end{tikzpicture}
	\end{center}
\end{prop}




\end{document}